\documentclass[11pt,a4paper,reqno]{amsart}
\usepackage{amssymb}
\usepackage{latexsym}
\usepackage{amsmath}
\usepackage{mathrsfs}
\usepackage{cite}
\usepackage{textcomp}
\usepackage{upgreek}
\usepackage{mathabx}
\usepackage{amscd}
\usepackage{color}
\usepackage{cancel}
\usepackage{calc}

\newcommand{\scal}[2]{\langle #1,#2\rangle}
\newcommand{\cc}[1]{\mathbf C^{#1}}
\newcommand{\nn}[1]{\mathbf N^{#1}}
\newcommand{\rr}[1]{\mathbf R^{#1}}

\newcommand{\ro}{\mathbf R}
\newcommand{\no}{\mathbf N}
\newcommand{\nm}[2]{\Vert #1\Vert _{#2}}

\newcommand{\op}{\operatorname{Op}}
\newcommand{\sets}[2]{\{ \, #1\, ;\, #2\, \} }
\newcommand{\Sets}[2]{\left \{ \, #1\, ;\, #2\, \right \} }
\newcommand{\ep}{\varepsilon}

\newcommand{\cdo}{\, \cdot \, }

\newcommand{\wpr}{{\text{\footnotesize $\#$}}}

\newcommand{\eabs}[1]{\langle #1\rangle}     

\newcommand{\vrum}{\vspace{0.2cm}}

\newcommand{\Sh}{\operatorname{Sh}}
\newcommand{\repart}{{\operatorname{Re}}}
\newcommand{\impart}{{\operatorname{Im}}}
\newcommand{\aw}{\operatorname{aw}}

\newcommand{\maclA}{\mathcal A}

\newcommand{\maclH}{\mathcal H}

\newcommand{\maclS}{\mathcal S}

\newcommand{\mascF}{\mathscr F}
\newcommand{\mascP}{\mathscr P}
\newcommand{\mascS}{\mathscr S}
\newcommand{\cT}{\mathcal{T}}

\newcommand{\fka}{\mathfrak a}
\newcommand{\fkb}{\mathfrak b}
\newcommand{\fkc}{\mathfrak c}

\newcommand{\cS}{\mathscr{S}}
\newcommand{\leqs}{\leqslant}
\newcommand{\geqs}{\geqslant}
\newcommand{\la}{\langle}
\newcommand{\ra}{\rangle}

\numberwithin{equation}{section}          

\newtheorem{thm}{Theorem}
\numberwithin{thm}{section}
\newtheorem{prop}[thm]{Proposition}
\newtheorem{cor}[thm]{Corollary}
\newtheorem{lemma}[thm]{Lemma}

\newcommand{\rubrik}{}

\theoremstyle{definition}

\newtheorem{defn}[thm]{Definition}
\newtheorem{example}[thm]{Example}

\theoremstyle{remark}
\newtheorem{rem}[thm]{Remark}

\title{Pseudo-differential operators with isotropic symbols,
Wick and anti-Wick operators, and hypoellipticity}

\author{Nenad Teofanov}

\address{Department of Mathematics and Informatics,
University of Novi Sad, Novi Sad, Serbia}

\email{nenad.teofanov@dmi.uns.ac.rs}

\author{Joachim Toft}

\address{Department of Computer science, Mathematics and Physics,
Linn{\ae}us University, V{\"a}xj{\"o}, Sweden}

\email{joachim.toft@lnu.se}

\author{Patrik Wahlberg}

\address{Department of Computer science, Mathematics and Physics,
Linn{\ae}us University, V{\"a}xj{\"o}, Sweden}

\email{patrik.wahlberg@lnu.se}



\begin{document}

\maketitle


\begin{abstract}
We study the link between $\Psi$dos
and Wick operators via the Bargmann transform. 
We deduce a formula for the symbol of the Wick operator
in terms of the short-time Fourier transform of the Weyl symbol. 
This gives characterizations of Wick symbols of
$\Psi$dos of Shubin type and of infinite order,
and results on composition. 
We prove a series expansion of Wick operators in
anti-Wick 
operators
which leads to a sharp G{\aa}rding inequality and transition of
hypoellipticity between Wick and and Shubin symbols.  
Finally we show continuity results for anti-Wick operators, and
estimates for the Wick symbols of anti-Wick operators. 
\end{abstract}

\par

\section{Introduction}\label{sec0}

\par

In the paper we investigate conjugation with the Bargmann
transformation of pseudo-differential and Toeplitz 
operators on $\rr d$ with isotropic symbols, and we 
explore relations between Wick and anti-Wick operators. 
Particularly we consider
Shubin operators and operators of infinite order.
This gives rise to analytic type pseudo-differential operators on $\cc d$
that are called Wick or Berezin operators
because of the fundamental contributions
by F. Berezin \cite{Berezin71,Berezin72}, which in turns goes back to
some ideas in \cite{Wi} by G. C. Wick. 

\par

Let $a$ be a suitable locally bounded function on $\cc {2d}$ such that
$z\mapsto a(z,w)$ is analytic, $z,w\in \cc d$. Then the Wick operator
$\op _{\mathfrak V}(a)$ with symbol $a$ is the operator which takes an
appropriate entire function $F$ on $\cc d$ into the entire function
\begin{equation}\label{Eq:AnalPseudoIntro}
\op _{\mathfrak V}(a)F (z) = \pi ^{-d}
\int _{\cc d} a(z,w)F(w)e^{(z-w,w)}\, d\lambda (w),
\end{equation}
where $d\lambda$ is the Lebesgue measure
and $(\cdo ,\cdo )$ is the scalar product on $\cc d$.
(See \cite{Ho1} and Section \ref{sec1}
for notation.) Wick operators appear naturally in
several problems in analysis and its applications, e.{\,}g.
in quantum mechanics. 
For example, the harmonic oscillator, the creation
and annihilation operators take the simple forms
$$
F\mapsto \scal z{\nabla _z}F +cF,
\quad
F\mapsto z_jF
\quad \text{and}\quad
F\mapsto \partial _{z_j}F,
$$
respectively, for some constant $c$, in the Wick formulation (see \cite{B1}).

\par

An advantage of the Wick calculus compared to corresponding
operators on functions and distributions defined on $\rr d$
is that in almost all situations, the involved functions are entire, 
which admits the use of the powerful techniques of complex analysis. 
(A more general approach is studied in \cite{Teofanov2}, where the
Wick calculus is formulated in terms of spaces of
formal power series expansions instead of spaces of entire functions.)
The possible lack of analyticity of $a(z,w)$
in \eqref{Eq:AnalPseudoIntro} with respect to the $w$ variable
is removable in the sense
that for any Wick symbol $a$, there is a unique $a_0$ such that
$(z,w)\mapsto a_0(z,\overline w)$ is entire, and
$\op _{\mathfrak V}(a)=\op _{\mathfrak V}(a_0)$. Consequently it is no restriction
to assume that $a(z,w)$ in \eqref{Eq:AnalPseudoIntro} is analytic in $z$ and
conjugate analytic in $w$, which we do in the introduction henceforth. 
Any linear and continuous operator from the Schwartz space,
a Fourier invariant Gelfand-Shilov space or Pilipovi{\'c} space, to
the corresponding distribution spaces, respectively, is in a unique way transformed
into a Wick operator by the Bargmann transform (see \cite{Teofanov2}).

\par

Several operators in quantum mechanics are so-called Shubin operators, i.{\,}e. 
pseudo-differential operators 
$$
\op (\fka )f(x) = (2\pi )^{-\frac d2}\int _{\rr d}\fka (x,\xi )
\widehat f(\xi )e^{i\scal x\xi}\, d\xi ,
\qquad f\in \mascS (\rr d),
$$
where the symbol $\fka$ belongs to the Shubin class
$\Sh _\rho ^{(\omega )}(\rr {2d})$, the set of all $\fka \in C^\infty (\rr {2d})$
such that
$$
|\partial _x^\alpha \partial _\xi ^\beta \fka (x,\xi )|
\lesssim
\omega (x,\xi )(1+|x|+|\xi |)^{-\rho |\alpha +\beta |},\quad \alpha ,\beta \in \nn d.
$$
Here $\omega$ is a suitable weight function on $\rr {2d}$ and $0 \le \rho \le 1$. 
Partial differential operators with polynomial coefficients,
e.{\,}g. the creation and annihilation operators or the harmonic oscillator
mentioned above, are examples of Shubin operators.
In Section \ref{sec2} we prove that the Bargmann image of Shubin
operators with symbols in
$\Sh _\rho ^{(\omega )}(\rr {2d})$ is the set of all
Wick operators in \eqref{Eq:AnalPseudoIntro} such that
$a$ belongs to $\wideparen \maclA _{\Sh ,\rho}^{(\omega )}(\cc {2d})$. 
This means that $\cc {2d} \ni (z,w)\mapsto a(z,\overline w)$
is an entire function that satisfies
\begin{equation}\label{Eq:AnalShubinIntro}
|\partial _z^\beta \overline \partial _w^\gamma a(z,w)|
\lesssim
e^{\frac 12|z-w|^2}\omega (\sqrt 2 \overline z)
\eabs {z+w}^{-\rho |\beta +\gamma |} \eabs {z-w}^{-N}
\end{equation}
for every $N\ge 0$. 

\par

An important subclass of Wick operators are the
anti-Wick operators, which are 
Wick operators where the symbol
$a(z,w)$ does not depend on $z$. That is, for an appropriate
measurable function $a_0$ on $\cc d$, its anti-Wick operator is given by
\begin{equation}\tag*{(\ref{Eq:AnalPseudoIntro})$'$}
\op _{\mathfrak V}^{\aw}(a_0)F (z) = \pi ^{-d}
\int _{\cc d} a_0(w)F(w)e^{(z-w,w)}\, d\lambda (w).
\end{equation}
Again $F$ is a suitable entire function on $\cc d$. The anti-Wick
operators can also be described as the Bargmann image of
Toeplitz operators on $\rr d$. (See e.{\,}g. \cite{LieSol,Shubin1,To11} for the definition
of Toeplitz operators.)

\par

A feature of Toeplitz operators and anti-Wick operators, useful for energy
estimates in quantum mechanics and
time-frequency analysis, is that non-negative symbols give rise to
non-negative operators. (Cf. e.{\,}g. \cite{LieSol,Lieb1,Lieb2}.) 
An operator $T=\op _{\mathfrak V}(a)$ 
with $a$ satisfying \eqref{Eq:AnalShubinIntro} for every $N\ge 0$, is called
\emph{positive} (\emph{non-negative}), if there is a constant
$C> 0$ ($C\ge 0$) such that 
$$
(TF,F)_{A^2} \ge C\nm F{A^2}^2,
$$
for every analytic polynomial $F$ on $\cc d$, where $(\cdo ,\cdo )_{A^2}$
is the scalar product induced by the Hilbert norm
$$
\nm F{A^2} =\pi ^{-\frac d2}
\left ( \int _{\cc d} |F(z)|^2e^{-|z|^2}\, d\lambda (z) \right )^{\frac 12}.
$$

\par

The implication from non-negative symbols to non-negative operators 
is not relevant for Wick operators in \eqref{Eq:AnalPseudoIntro}
when $a(z,w)$ is not constant with respect to $z$, since
the analyticity of the map $z\mapsto a(z,w)$ implies that
$a(z,w)$ is non-real almost everywhere. For such symbols it is
instead natural to check whether positivity of the map
$w\mapsto a(w,w)$ leads to positive operators (see e.{\,}g.
\cite{Berezin71,Berezin72,Fo}). By choosing
$$
d=1,\quad a(z,w) = 1-2z\overline w +2z^2\overline w^2
\quad \text{and}\quad
F(z)=z
$$
we obtain
$$
a(w,w)=(1-|w|^2)^2+|w|^4> 0
\quad \text{but}\quad
(\op _{\mathfrak V}(a)F,F)_{A^2} =-1<0.
$$
Consequently 
$\op _{\mathfrak V}(a)$
may fail to be a non-negative operator even though $a(w,w)$ is positive.

\par

On the other hand, for certain conditions on $a$, we deduce in Section \ref{sec3}
a weaker positivity result for Wick operators, which is equivalent to the 
sharp G{\aa}rding inequality in isotropic pseudo-differential calculus on $\rr d$
(see Theorem 18.6.7 and the proof of Theorem 18.6.8 in \cite{Ho1}). That is
for $a\in \wideparen \maclA _{\Sh ,\rho}^{(\omega )}(\cc {2d})$
with
$\omega (z)=\eabs z^{2\rho}$ and $\rho > 0$ we prove
\begin{align}
\repart  (\op _{\mathfrak V}(a)F,F)_{A^2}
&\ge -C\nm F{A^2}^2
\label{Eq:SharpGardingIntro}
\intertext{and}
|\impart  (\op _{\mathfrak V}(a)F,F)_{A^2}|
&\le C\nm F{A^2}^2,
\quad \text{when}\quad
a(w,w) \ge 0
\label{Eq:SharpGardingIntroImPart}
\end{align}
(cf. Theorem \ref{Thm:ShGarding}).
In particular we obtain energy estimates also for
Wick operators with symbols that are non-negative on the diagonal.

\par

The latter result is obtained by approximating Wick operators by
anti-Wick operators, using for the Wick operator \eqref{Eq:AnalPseudoIntro} with
$a\in \wideparen \maclA _{\Sh ,\rho}^{(\omega )}(\cc {2d})$
the remarkable identity
\begin{equation}\label{Eq:WickToAntiWickIntro}
\op _{\mathfrak V}(a)
=
\sum _{|\alpha |<N} \frac {(-1)^{|\alpha |}}{\alpha !}
\op _{\mathfrak V}^{\aw}(b_\alpha ) + \op _{\mathfrak V}(c_N)
\quad \text{where} \quad
b_\alpha (w)
=
\partial _z^\alpha \overline \partial _{w}^\alpha a(w,w),
\end{equation}
for some $c_N\in \maclA _{\Sh ,\rho}^{(\omega _N)}(\cc {2d})$
with $\omega _N(z)= \omega (z)\eabs z^{-2N\rho}$. 
Here we again assume $\rho > 0$. 
The decay conditions on $b_\alpha$ and $c_N$ are, respectively,
\begin{align}
|\partial _w^\beta \overline \partial _w^\gamma b_\alpha (w)|
& \lesssim
\omega (\sqrt 2 \overline w)\eabs {w}^{-\rho |2\alpha +\beta +\gamma |},
\qquad
\alpha, \beta, \gamma  \in \nn d,
\label{Eq:WickAntiWickbalphaIntro}
\intertext{and}
|\partial _z^\beta \overline \partial _w^\gamma c_N(z,w)|
& \lesssim
e^{\frac 12|z-w|^2}\omega (\sqrt 2 \overline z)\eabs z^{-2N\rho}
\eabs {z+w}^{-\rho |\beta +\gamma |} \eabs {z-w}^{-N}.
\end{align}
Consequently, many Wick operators can essentially be expressed
as linear combinations of anti-Wick operators. The expansion
\eqref{Eq:WickToAntiWickIntro} is deduced in Section \ref{sec3} using
Taylor expansion and integration by parts, see Proposition
\ref{Prop:WickToAntiWick} and Remark \ref{Rem:WickToAntiWick2}.

\par

The conditions on $b_\alpha$ are the same as the conditions on $a$
\eqref{Eq:AnalShubinIntro}, restricted to the diagonal $z=w$, and with improved decay.
On the diagonal, the growth term $e^{\frac 12|z-w|^2}$ disappears, which dominates in 
\eqref{Eq:AnalShubinIntro} when $|z-w|\gtrsim |z|$ or $|z-w|\gtrsim |w|$. 
The right-hand side of \eqref{Eq:WickAntiWickbalphaIntro}
becomes as large as possible when
$\alpha =\beta =\gamma =0$, that is $b_0$ 
is the dominating term in the sum \eqref{Eq:WickToAntiWickIntro}.

\par

The conditions on $c_N$ are the same as the estimates \eqref{Eq:AnalShubinIntro}
again with improved decay due to the factor $\eabs z^{-2N\rho}$. 

\par
 
For polynomial symbols, \eqref{Eq:WickToAntiWickIntro}
agree with the integral formula \cite[Theorem 3]{Berezin71} due to Berezin
which carry over Wick operators into anti-Wick operators. For
the general case, \eqref{Eq:WickToAntiWickIntro} is analogous to the
approximation technique of pseudo-differential operators on $\rr d$
in terms of Toeplitz operators given in \cite[Theorem 24.1]{Shubin1}
and its proof, by Shubin.

\par

The anti-Wick symbols in \eqref{Eq:WickToAntiWickIntro}
$b_\alpha (w)=\partial _z^\alpha \overline \partial _w^\alpha a (w,w)$ extend to have the property that
$\partial _z^\alpha \overline \partial _w^\alpha a (z,w)$ is entire in
$z$ and conjugate entire in $w$.
Note that restriction to the diagonal also appears in the positivity condition
\eqref{Eq:SharpGardingIntro} on Wick symbols.

\par

The sharp G{\aa}rding inequality \eqref{Eq:SharpGardingIntro} is 
reached by using the fact that $\op _{\mathfrak V}^{\aw}(b_0)$
is non-negative, 
and that
if $T$ is either $\op ^{\aw}_{\mathfrak V}(b_\alpha )$
or $\op _{\mathfrak V}(c_N)$ for $\alpha \neq 0$, then
$\nm {TF}{A^2}\lesssim \nm F{A^2}$ when $F\in A(\cc d)$ is a polynomial.

\par

In Section \ref{sec5} we deduce links concerning ellipticity, hypoellipticity
(in Shubin's sense) and weak ellipticity between Shubin and Wick symbols. 
The notion of hypoelliptic symbol resembles hypoelliptic symbols
in Shubin's sense (see \cite{Shubin1}).
More specifically, we say that the symbol $\fka \in \Sh _\rho ^{(\omega )}(\rr {2d})$ is
hypoelliptic of order $\rho _0\ge 0$, whenever there is an $R>0$ such that
$$
|\fka (x,\xi )|\gtrsim \omega (x,\xi )\eabs {(x,\xi )}^{-\rho _0}
\quad \text{and}\quad |\partial ^\alpha \fka (x,\xi )|
\lesssim |\fka (x,\xi )|\eabs {(x,\xi )}^{-\rho |\alpha |}
$$
when  $|(x,\xi )|\ge R$.

\par

A linear operator $T$ from $\mascS '(\rr d)$ to $\mascS '(\rr d)$
is called globally hypoelliptic if
$$
Tf=g,\ f\in \mascS '(\rr d),\ g\in \mascS (\rr d)
\quad \Rightarrow \quad
f\in \mascS (\rr d).
$$
(See e.{\,}g. \cite{BogBuzRod}.)
It can be proved that a pseudo-differential operator with hypoelliptic
symbol in Shubin's sense is globally hypoelliptic as operator (see e.{\,}g.
\cite[Corollary 25.1]{Shubin1}).

\par

We show, similarly to our investigations of
the sharp G{\aa}rding
inequality and for expansion \eqref{Eq:WickToAntiWickIntro}, that ellipticity, hypoellipticity
and to some degree weak ellipticity for the Shubin symbol $\fka$
can be characterized by certain conditions for the corresponding Wick symbol
$a(z,w)$ along the diagonal $z=w$. For example, let
$\fka$ be a polynomial on $\rr d$ with principal symbol $\fka _p$, and let $a(z,w)$
be a polynomial in $z, \overline w \in \cc d$ with principal part
$a_p$. Then $\fka$ is elliptic means that
$\fka _p(x,\xi )\neq 0$ when $(x,\xi )\neq (0,0)$, and 
$a$ is elliptic means that $a_p(z,z)\neq 0$ when $z\neq 0$.
For such $\fka$ we prove
$$
\fka \ \text{is elliptic}
\quad \Leftrightarrow \quad
a\ \text{is elliptic},
$$
when $a(z,w)$ is the Wick symbol corresponding to $\fka$
(which must be a polynomial in $z$ and $\overline w$).

\medspace

Our investigations include the Bargmann transform of certain
operators of infinite order, i.{\,}e. pseudo-differential operators
with ultra-differentiable symbols that are permitted to grow faster
than polynomially at infinity together with their derivatives. Particularly
we consider Wick operators of infinite order, i.{\,}e. the Bargmann
images $\op _{\mathfrak V}(a)$ of operators $\op (\fka )$ of infinite order
in \cite{AbCaTo},
and characterize their images under the  Bargmann transform
(see Theorem \ref{Thm:GevreySymbolsBargmTransfer}). 
Then we deduce in Subsections
\ref{subsec3.2} and \ref{subsec3.3} continuity results for anti-Wick operators
which holds for the symbols $b_\alpha$ in \eqref{Eq:WickToAntiWickIntro} when
$\op _{\mathfrak V}(a)$ is the Bargmann image of an 
operator of infinite order.

\par

In fact, in Subsection \ref{subsec3.2}
we show that $\op _{\mathfrak V}^{\aw}(b_\alpha )$ possess several
other continuity properties than what is valid for $\op _{\mathfrak V}(a)$
in the expansion \eqref{Eq:WickToAntiWickIntro}
(see Propositions \ref {prop:contAW} and \ref{prop:contAW2}). In Subsection
\ref{subsec3.3} we deduce estimates of the Wick symbol $b_{\alpha}^{\aw}$
to the anti-Wick operator $\op _{\mathfrak V}^{\aw}(b_\alpha )$, i.{\,}e. the unique element
$b_{\alpha}^{\aw}\in \wideparen A(\cc {2d})$ such that $\op _{\mathfrak V}(b_{\alpha}^{\aw})
=\op _{\mathfrak V}^{\aw}(b_\alpha )$. We show that usually, $b_{\alpha}^{\aw}$
satisfies stronger conditions than $a$ when $\op _{\mathfrak V}(a)$
is a Wick operator of infinite order (see Theorems \ref{Thm:WickSymbolAntiWickOpFolland},
\ref{Thm:WickSymbolAntiWickOpGS} and \ref{Thm:WickSymbolAntiWickOpSOmega}).

\par

The paper is organized as follows. In Section \ref{sec1} we
recall useful properties for weight functions, Gelfand-Shilov spaces, the Bargmann
transform, pseudo-differential operators, Wick and anti-Wick operators.
Thereafter we characterize in Section \ref{sec2} Shubin operators
and operators of infinite order in terms of appropriate classes of
Wick operators on the Bargmann side.
These considerations are based on a formula for the Wick symbol expressed
in terms of a short-time Fourier transform of the Weyl symbol, and
admits characterization of the Wick symbols corresponding to
Shubin Weyl symbols and symbols for operators of infinite order
(see Proposition \ref{Prop:SBaTaRel}).

\par

In Section \ref{sec2} we also study composition and show for example 
that the well-known closure under composition of Shubin operators 
and operators of infinite orders have simple and natural proofs on the Wick symbol side. 

\par

In Section \ref{sec3} we deduce series expansions of Wick operators
in terms of anti-Wick operators, and between Wick symbols and symbols
to corresponding Shubin operators. 
We also consider anti-Wick operators, and
show continuity results for them. We show
that the upper bounds for the Wick symbols of anti-Wick operators 
are stricter than for general Wick symbols. 

\par

In Section \ref{sec4} we discuss lower bounds for Wick operators
and deduce the sharp G{\aa}rding's inequality.
Section \ref{sec5} concerns ellipticity, hypoellipticity and weak ellipticity. 

\par

Finally we observe in Section \ref{sec6} that a polynomial bound of a Wick symbol implies
that the symbol is a polynomial. For pseudo-differential operators this
corresponds to partial differential operators with polynomial coefficients. 
This gives a characterization of such operators as those having
polynomially bounded Wick symbols. 

\par

\section{Preliminaries}\label{sec1}

\par

In this section we recall some facts on function and distribution
spaces as well as on pseudo-differential operators, Wick and anti-Wick operators.
Subsection \ref{subsec1.1} concerns weight functions and
Subsection \ref{subsec1.2} treats Gelfand-Shilov spaces. 
In Subsection \ref{subsec1.3} we introduce the Bargmann transform and
topological spaces of entire functions on $\cc d$, and 
in Subsection \ref{subsec1.4} we recall the definitions
and some facts on pseudo-differential operators on $\rr d$ as well
as Wick and anti-Wick operators on $\cc d$. 
Subsection \ref{subsec1.5} defines certain symbol classes for pseudo-differential
operators on $\rr d$.

\par

\subsection{Weight functions}\label{subsec1.1}
A \emph{weight} on $\rr d$ is a positive function $\omega \in  L^\infty _{loc}(\rr d)$
such that $1/\omega \in  L^\infty _{loc}(\rr d)$. The weight $\omega$
is called \emph{moderate} if there is a positive locally bounded function
$v$ such that
\begin{equation}\label{eq:2}
\omega(x+y)\le C\omega(x)v(y),\quad x,y\in\rr{d},
\end{equation}
for some constant $C\ge 1$. If $\omega$ and $v$ are weights such
that \eqref{eq:2} holds, then $\omega$ is also called \emph{$v$-moderate}.
The set of all moderate weights on $\rr d$ is denoted by $\mascP _E(\rr d)$.
The set $\mascP (\rr d)$ consists of weights that are $v$-moderate for 
a polynomially bounded weight, that is a weight of the form $v(x) = {\eabs x}^s$
where $\eabs x = (1+|x|^2)^{\frac 12}$ and $s \ge 0$. 
The bracket notation is also used for complex arguments as 
$\eabs z=(1+|z|^2)^{\frac 12}$ when $z\in \cc d$. 
If $s\in \mathbf R$ then $x\mapsto \eabs x^s$ belongs to $\mascP (\rr d)$,
due to Peetre's inequality
\begin{equation}\label{eq:Peetre}
\eabs{x+y}^s \leqs 2^{|s|} \eabs{x}^s\eabs{y}^{|s|}\qquad x,y\in\rr{d},
\qquad s \in \ro. 
\end{equation}

\par

The weight $v$ is called \emph{submultiplicative}
if it is even and \eqref{eq:2}
holds for $\omega =v$.
If \eqref{eq:2} holds and $v$ is submultiplicative then 
\begin{equation}\label{eq:2Next}
\begin{gathered}
\frac {\omega (x)}{v(y)} \lesssim \omega(x+y) \lesssim \omega(x)v(y),
\\[1ex]
\quad
v(x+y) \lesssim v(x)v(y)
\quad \text{and}\quad v(x)=v(-x),
\quad x,y\in\rr{d}.
\end{gathered}
\end{equation}
The notation
$A(\theta )\lesssim B(\theta )$, $\theta \in \Omega$,
means that there is a constant $c>0$ such that $A(\theta )\le cB(\theta )$
for all $\theta \in \Omega$.

\par

If $\omega$ is a moderate weight then by \cite{To11} there is a
submultiplicative weight
$v$ such that \eqref{eq:2} and \eqref{eq:2Next}
hold. If $v$ is submultiplicative then
\begin{equation}\label{Eq:CondSubWeights}
1\lesssim v(x) \lesssim e^{r|x|}
\end{equation}
for some constant $r>0$ (cf. \cite{Groch}). In particular, if $\omega$ is moderate, then
\begin{equation}\label{Eq:ModWeightProp}
\omega (x+y)\lesssim \omega (x)e^{r|y|}
\quad \text{and}\quad
e^{-r|x|}\lesssim \omega (x)\lesssim e^{r|x|},\quad
x,y\in \rr d
\end{equation}
for some $r>0$.
If not otherwise specified the symbol $v$ always denote a submultiplicative weight. 

\par

\subsection{Gelfand-Shilov spaces}\label{subsec1.2}
Let $s ,\sigma > 0$. The 
Gelfand-Shilov space $\maclS _s^\sigma (\rr d)$ ($\Sigma _s^\sigma (\rr d)$) of
Roumieu (Beurling) type consists of all $f\in C^\infty (\rr d)$
such that
\begin{equation}\label{gfseminorm}
\nm f{\maclS _{s,h}^\sigma}\equiv \sup \frac {|x^\alpha \partial ^\beta
f(x)|}{h^{|\alpha  + \beta |}\alpha !^s \, \beta !^\sigma}
\end{equation}
is finite for some (every) $h>0$. The supremum refers to all
$\alpha ,\beta \in \mathbf N^d$ and $x\in \rr d$. 
The seminorms
$\nm \cdo {\maclS _{s,h}^\sigma}$ induce an inductive limit topology for the
space $\maclS _s^\sigma (\rr d)$ and a projective limit topology for
$\Sigma _s^\sigma (\rr d)$. The latter space is a Fr{\'e}chet space under this topology.
The space $\maclS _s^\sigma (\rr d)\neq \{ 0\}$ ($\Sigma _s^\sigma (\rr d)\neq \{0\}$),
if and only if $s+\sigma \ge 1$ ($s+\sigma \ge 1$ and $(s,\sigma) \neq (\frac 12, \frac 12)$).
We write $\maclS _s (\rr d) = \maclS _s^s (\rr d)$ and $\Sigma _s (\rr d) = \Sigma _s^s (\rr d)$. 

\par

The \emph{Gelfand-Shilov distribution spaces} $(\maclS _s^\sigma )'(\rr d)$
and $(\Sigma _s^\sigma )'(\rr d)$ are the dual spaces of $\maclS _s^\sigma (\rr d)$
and $\Sigma _s^\sigma (\rr d)$, respectively.

\par

The embeddings
\begin{multline}\label{GSembeddings}
\maclS _{s_1}^{\sigma _1} (\rr d) \hookrightarrow \Sigma _{s_2}^{\sigma _2}(\rr d)
\hookrightarrow
\maclS _{s_2}^{\sigma _2} (\rr d)
\hookrightarrow
\mascS (\rr d)
\\[1ex]
\hookrightarrow \mascS '(\rr d) 
\hookrightarrow  (\maclS _{s_2}^{\sigma _2})'(\rr d)
\hookrightarrow  (\Sigma _{s_2}^{\sigma _2})'(\rr d)
\hookrightarrow (\maclS _{s_1}^{\sigma _1}) '(\rr d),
\\[1ex]
s_1+\sigma _1 \ge 1,\ s_1<s_2,\ \sigma _1<\sigma _2,
\end{multline}
are dense. 
For topological spaces $A$ and $B$, $A\hookrightarrow B$ means that the inclusion $A\subseteq B$ is continuous.

\par

The spaces $\maclS _s$ and $\Sigma_s$, and their duals spaces, admit
characterizations in terms of  coefficients with respect to expansions with respect to
the Hermite functions
$$
h_\alpha (x) = \pi ^{-\frac d4}(-1)^{|\alpha |}
(2^{|\alpha |}\alpha !)^{-\frac 12}e^{\frac {|x|^2}2}
(\partial ^\alpha e^{-|x|^2}),\quad \alpha \in \nn d. 
$$
The set of Hermite functions on $\rr d$ is an orthonormal basis for
$L^2(\rr d)$. 
We use $\maclH _0(\rr d)$ to denote the space of finite linear combinations
of Hermite functions.
Then $\maclH _0(\rr d)$ is dense in the Schwartz space $\mascS (\rr d)$, 
as well as in $\mascS '(\rr d)$, with respect to its weak$^*$ topology. 
The same conclusion is true for $\Sigma _s(\rr d)$ when $s>\frac 12$,
$\maclS _s(\rr d)$ when $s\ge \frac 12$ and their distribution dual
spaces $\Sigma _s'(\rr d)$ and $\maclS _s'(\rr d)$. 
An $f$ in any of these spaces possess an expansion of the form
\begin{equation}\label{Eq:HermiteExpansions}
f=\sum _{\alpha \in \nn d}c(f,\alpha )h_\alpha ,
\quad
c (f,\alpha )=(f,h_\alpha ), \quad \alpha \in \nn d.
\end{equation}
Here $(\cdo , \cdo )$ denotes the unique extensions of the $L^2$ form, 
which is linear in the first variable and conjugate linear in the second variable, 
from $\maclH _0(\rr d)\times \maclH _0(\rr d)$ to
$\maclS _s'(\rr d)\times \maclS _s(\rr d)$ or
$\Sigma _s'(\rr d)\times \Sigma _s(\rr d)$.
We recall that (cf. \cite[Chapter V.3 ]{RS})
\begin{equation}
\begin{alignedat}{5}
f&\in \mascS (\rr d) & \quad &\Leftrightarrow &\quad
|c(f,\alpha )| &\lesssim \eabs \alpha ^{-N} &
\ &\text{for every}& \ N&\ge 0 ,
\\[1ex]
f&\in \mascS '(\rr d) & \quad &\Leftrightarrow &\quad
|c(f,\alpha )| &\lesssim \eabs \alpha ^{N} &
\ &\text{for some}& \ N&\ge 0.
\end{alignedat}
\end{equation}
The topology on $\mascS (\rr d)$ is equivalent to the Fr\'echet space
topology defined by the sequence space seminorms
\begin{equation*}
\mascS (\rr d) \ni f \mapsto \sum_{\alpha \in \nn d}
\eabs \alpha ^{2N} |c(f,\alpha )|^2, \quad N \ge 0. 
\end{equation*}
For $f \in \mascS '(\rr d)$ the sum in \eqref{Eq:HermiteExpansions}
converges in the weak$^*$ topology. 

\par

The Hermite functions are eigenfunctions to the harmonic
oscillator $H=H_d\equiv |x|^2-\Delta$ and to the Fourier transform
$\mathscr F$, given by
$$
\mathscr Ff (\xi )= \widehat f(\xi ) \equiv (2\pi )^{-\frac d2}\int _{\rr
{d}} f(x)e^{-i\scal  x\xi }\, dx, \quad \xi \in \rr d,
$$
when $f\in L^1(\rr d)$. Here $\scal \cdo \cdo$ denotes the 
scalar product on $\rr d$. In fact
$$
H_dh_\alpha = (2|\alpha |+d)h_\alpha, \quad \alpha \in \nn d.
$$

\par

The Fourier transform $\mathscr F$ extends
uniquely to homeomorphisms on $\mathscr S'(\rr d)$,
from $(\maclS _s^\sigma )'(\rr d)$ to $(\maclS ^s_\sigma )'(\rr d)$ and
from $(\Sigma _s^\sigma )'(\rr d)$ to $(\Sigma ^s_\sigma )'(\rr d)$. 
It also restricts to
homeomorphisms on $\mathscr S(\rr d)$, from
$\maclS _s^\sigma (\rr d)$ to $\maclS ^s_\sigma (\rr d)$, from
$\Sigma _s^\sigma (\rr d)$ to $\Sigma ^s_\sigma (\rr d)$,
and to a unitary operator on $L^2(\rr d)$. Similar facts hold true
when the Fourier transform is replaced by a partial
Fourier transform.

\par

Let $\phi \in \mascS  (\rr d) \setminus 0$ be
fixed. We use the transform 
\begin{equation}\label{eq:cTdef}
\begin{aligned}
\cT_\phi  f(x,\xi) &= (2\pi)^{-\frac d2}e^{i \langle x, \xi \rangle}
(f,e^{i\scal \cdo \xi}\phi (\cdo -x))
\\[1ex]
&=
e^{i \langle x, \xi \rangle}\mascF (f\cdot \overline {\phi (\cdo -x)})(\xi )
=
\mascF (f(\cdo +x)\overline \phi )(\xi ), \quad x, \xi \in \rr d,
\end{aligned}
\end{equation}
where $f \in \mascS '(\rr d)$ and $\phi  \in \mascS (\rr d) \setminus 0$
(cf. \cite{Cappiello1}). If $f ,\phi \in \mascS (\rr d)$ then
\begin{align*}
\cT _\phi f(x,\xi )
&=
(2\pi )^{-\frac d2}e^{i \langle x, \xi \rangle}
\int _{\rr d} f(y)\overline {\phi
(y-x)}e^{-i\scal y\xi}\, dy
\\[1ex]
&=
(2\pi )^{-\frac d2}
\int _{\rr d} f(y+x)\overline {\phi
(y)}e^{-i\scal y\xi}\, dy, \quad x,\xi \in \rr d.
\end{align*}

\par

We notice that the short-time Fourier transform $V_{\phi}f$
of $f$ is given by
\begin{equation}\label{Eq:LinkSTFTandTmap}
V_{\phi}f(x,\xi) = e^{-i \langle x, \xi \rangle} \cT_\phi  f(x,\xi).
\end{equation}
Thus by \cite[Theorem 2.3]{To11} it follows that the definition of the map
$(f,\phi)\mapsto \cT _{\phi} f$ from $\mascS (\rr d) \times \mascS (\rr d)$
to $\mascS(\rr {2d})$ is uniquely extendable to a continuous map from
$\maclS _s'(\rr d)\times \maclS_s'(\rr d)$
to $\maclS_s'(\rr {2d})$, and restricts to a continuous map
from $\maclS _s (\rr d)\times \maclS _s (\rr d)$
to $\maclS _s(\rr {2d})$.
The same conclusion holds with $\Sigma _s$ in place of
$\maclS_s$, at each place.

\par

The adjoint $\cT _\phi ^*$ is given by
$$
(\cT _\phi^* F, g)_{L^2(\rr d)}
=
(F, \cT _\phi g)_{L^2(\rr {2d})}
$$
for $F \in \maclS _s'(\rr {2d})$ and $g \in \maclS _s(\rr d)$,
and similarly with $\Sigma _s$ or with $\mascS$ in place of
$\maclS _s$ at each occurrence. 
When $F$ is a polynomially bounded measurable function we
write
\begin{equation}\label{Eq:Moyalsformula}
\cT _\phi^* F(y) = (2\pi)^{-\frac d2} \iint _{\rr {2d}} 
F(x,\xi) \, e^{i\scal {y-x}\xi}\phi (y-x)
\, d x d \xi ,
\end{equation}
where the integral is defined weakly so that
$(\cT _\phi^* F, g)_{L^2(\rr d)}
=
(F, \cT _\phi g)_{L^2(\rr {2d})}$
for $g \in \cS(\rr d)$. 
The identity \eqref{Eq:Moyalsformula} is called Moyal's formula. 

\par

We have
\begin{equation}\label{eq:reproducing}
(\cT_\psi ^*\circ \cT _\phi )f
=(\psi ,\phi ) f, \qquad f \in \maclS _s'(\rr d),\ \phi ,\psi \in \maclS _s(\rr d),
\end{equation}
and similarly with $\Sigma _s$ or with $\mascS$ in place of
$\maclS _s$ at each occurrence. 

\par

Two important features of $\cT _\phi$ which distinguish
it from the short-time Fourier transform are the differential identities
\begin{align}
\label{eq:diffident}
\partial_x^\alpha \cT _\phi f (x,\xi) & = \cT _\phi (\partial^\alpha f) (x,\xi),
\qquad \alpha \in \nn d
\intertext{and}
\label{eq:diffidentstar}
D_{\xi}^\beta \cT _\phi f (x,\xi) & = \cT_{g_\beta} f(x,\xi), \qquad \beta \in \nn d, \qquad
\phi _\beta (x) = (-x)^\beta \phi (x).
\end{align}

\par

By \eqref{Eq:LinkSTFTandTmap} it follows that
characterizations of Gelfand-Shilov spaces and their distribution spaces
in terms of estimates of their short-time Fourier transforms
carry over to estimates on $\cT _\phi$
in place of $V_\phi$. For example we have
the following (see e.{\,}g. \cite{GZ,Teof} for the proof of (1)
and \cite{Toft18} for the proof of (2)).
See also \cite{CPRT10} for related results.

\par

\begin{prop}\label{stftGelfand2}
Let $s,\sigma >0$, 
$\phi \in \maclS _s^\sigma (\rr d)\setminus 0$
($\phi \in \Sigma _s^\sigma (\rr d)\setminus 0$) and let
$f\in (\maclS _s^\sigma )'(\rr d)$ ($f\in (\Sigma _s^\sigma )'(\rr d)$). 
Then the following is true:
\begin{enumerate}
\item $f\in \maclS _s^\sigma (\rr d)$ ($f\in \Sigma _s^\sigma (\rr d)$) if and only if
\begin{equation}\label{stftexpest2}
|\cT _\phi f(x,\xi )| \lesssim  e^{-r (|x|^{\frac 1s}+|\xi |^{\frac 1\sigma})}, \quad x,\xi \in \rr d,
\end{equation}
for some (every) $r > 0$.
\item $f\in (\maclS _s^\sigma )'(\rr d)$ ($f\in (\Sigma _s^\sigma )'(\rr d)$) if and only if
\begin{equation}\label{stftexpest2Dist}
|\cT _\phi f(x,\xi )| \lesssim  e^{r(|x|^{\frac 1s}+|\xi |^{\frac 1\sigma})}, \quad
x,\xi \in \rr d,
\end{equation}
for every (some) $r > 0$.
\end{enumerate}
\end{prop}

\par

\subsection{The Bargmann transform and spaces of analytic
functions}\label{subsec1.3}

\par

If $\Omega \subseteq \cc d$ is open then $A(\Omega)$
consists of all (complex-valued) analytic functions on $\Omega$.
Complex derivatives are denoted, with $z = x+iy \in \Omega$,
\begin{equation*}
\partial_{z_j} = \frac{1}{2} \left( \partial_{x_j} - i \partial_{y_j} \right), \quad 
\overline{\partial}_{z_j} = \frac{1}{2} \left( \partial_{x_j} + i \partial_{y_j} \right)
\end{equation*}
for $1 \le j \le d$, which admits the Cauchy-Riemann equations to be written as 
$\overline{\partial}_{z_j} f = 0$, $1 \le j \le d$. 

\par

The Bargmann kernel is defined by
$$
\mathfrak A_d(z,y)=\pi ^{-\frac d4} \exp \Big ( -\frac 12(\scal
zz+|y|^2)+2^{1/2}\scal zy\Big ), \quad z \in \cc d, \quad y \in \rr d,
$$
where
$$
\scal zw = \sum _{j=1}^dz_jw_j\quad \text{and} \quad
(z,w)= \scal z{\overline w}
$$
when
$$
z=(z_1,\dots ,z_d) \in \cc d\quad  \text{and} \quad w=(w_1,\dots ,w_d)\in \cc d.
$$
Sometimes $\scal \cdo \cdo$ denotes the duality between a test function
space and its dual. The context precludes confusion between its double use.  
The Bargmann transform $\mathfrak V_df$ of $f\in \maclS _{1/2}'(\rr d)$
is the entire function
\begin{equation}\label{bargdistrform}
\mathfrak V_d f (z) =\scal f{\mathfrak A_d(z,\cdo )},
\quad z \in \cc d.
\end{equation}
The right-hand side is a well defined element in $A(\cc d)$,
since $y\mapsto \mathfrak A_d(z,y)$ belongs to
$\maclS _{1/2} (\rr d)$ for $z \in \cc d$ fixed, and
$\mathfrak A_d(\cdo,y)$ is entire for all $y \in \rr d$.
Let $p\in [1,\infty ]$ and $\omega \in \mascP _E(\rr d)$. Then
$L^p_{(\omega )}(\rr d)$ consists of all $f\in L^1_{loc}(\rr d)$ such
that $\nm f{L^p_{(\omega )}}\equiv \nm {f\cdot \omega }{L^p}$
is finite. If $f\in L^p_{(\omega )}(\rr d)$, then
\begin{multline}
\mathfrak V_d f (z) =\int_{\rr d} \mathfrak A_d(z,y)f(y)\, dy
\\[1ex]
=
\pi ^{-\frac d4}\int _{\rr d}\exp \Big ( -\frac 12(\scal
z z+|y|^2)+2^{1/2}\scal zy \Big )f(y)\, dy,\quad z \in \cc d.
\end{multline}
(Cf. \cite{B1,B2,To11,Toft18}.)

\par

For  $p\in (0,\infty]$, $\omega \in \mascP _E(\cc d)$ and
$\omega _0(z)=\omega (\sqrt 2\overline z)$, let
$A^p_{(\omega )}(\cc d)$ be the set of all $F\in A(\cc d)$
such that
\begin{equation*}
\nm F{A^p_{(\omega )}}
\equiv \pi ^{-\frac dp}
\nm {F\cdot e^{-\frac 12|\cdo |^2}\cdot \omega _0}{L^p}, 
\end{equation*}
and set $A^p=A^p_{(\omega )}$ when $\omega =1$.
It was proved by Bargmann \cite{B1} that
\begin{equation}\label{Eq:L2A2Isometry}
\mathfrak V_d : L^2(\rr d) \to A^2(\cc d)
\end{equation}
is bijective and isometric. The space $A^2(\cc d)$ is the Hilbert space 
of entire functions with scalar product 
\begin{equation*}
(F,G)_{A^2}\equiv  \int _{\cc d} F(z)\overline {G(z)}\, d\mu (z),\quad F,G\in A^2(\cc d), 
\end{equation*}
where $d\mu (z)=\pi ^{-d} e^{-|z|^2}\, d\lambda (z)$ and $d\lambda (z)$ is
the Lebesgue measure on $\cc d$. The space $A^2(\cc d)$ is known as the
Fock space in quantum mechanics (see\cite{Fo}).

\par

In \cite{B1} it was proved that the Bargmann transform maps the Hermite functions
to monomials as
\begin{equation}\label{Eq:BargmannHermiteMap}
\mathfrak V_dh_\alpha = e_\alpha ,\qquad e_\alpha (z)= \frac {z^\alpha}{\alpha !^{\frac 12}},
\quad z\in \cc d,\quad \alpha \in \nn d.
\end{equation}
The orthonormal basis
$\{ h_\alpha \}_{\alpha \in \nn d} \subseteq L^2(\rr d)$ 
is thus mapped to the orthonormal basis
$\{ e_\alpha \} _{\alpha \in \nn d}\subseteq A^2(\cc d)$. 
Bargmann also proved that there is a reproducing formula for
$A^2(\cc d)$. Let $\Pi _A$ be the operator from $L^2(d\mu )$
to $A(\cc d)$, given by
\begin{equation}\label{eq:projection}
\Pi _A F (z)= \int _{\cc d} F(w)e^{(z,w)}\, d\mu (w),
\quad z \in \cc d.
\end{equation}
Then $\Pi _A$ is the orthogonal projection from
$L^2(d\mu)$ to
$A^2(\cc d)$ (cf. \cite{B1}).

\par

When we discuss extensions and restrictions of the
Bargmann transform to Gelfand-Shilov spaces and their
distribution spaces, we use
\begin{equation}\label{Eq:GevreyQuasinormMIxed}
|z|_{s,\sigma} = |\operatorname{Re} z|^{\frac 1s}
+
|\operatorname{Im} z|^{\frac 1\sigma},
\qquad
z\in \cc d,
\end{equation}
and consider the seminorms
\begin{alignat*}{2}
\nm F{\maclA _{\mascS ;r}}
&\equiv
\nm {F\cdot e^{-\frac 12|\cdo |^2}\eabs \cdo 
^{r}}{L^\infty},&
\quad
\nm F{\maclA _{\mascS ;r}'}
&\equiv
\nm {F\cdot e^{-\frac 12|\cdo |^2}\eabs \cdo 
^{-r}}{L^\infty}
\intertext{and}
\nm F{\maclA _{\maclS _{s;r}^\sigma}}
&\equiv
\nm {F\cdot e^{-\frac 12|\cdo |^2+r|\cdo 
|_{s,\sigma}}}{L^\infty}, &
\quad
\nm F{\maclA _{\maclS _{s;r}^\sigma}'}
&\equiv
\nm {F\cdot e^{-\frac 12|\cdo |^2-r|\cdo 
|_{s,\sigma}}}{L^\infty}
\end{alignat*}
when $F\in A(\cc d)$, $r>0$ and $s,\sigma \ge  \frac 12$. Then
$\maclA _{0,s}^\sigma (\cc d)$ for
$s, \sigma > \frac 12$,
$\maclA _{\mascS} (\cc d)$ and
$(\maclA _s^\sigma )' (\cc d)$ for
$s,\sigma \ge \frac 12$
are the sets of all
$F\in A(\cc d)$ such that
\begin{equation*}
\nm F{\maclA _{\maclS _{s;r}^\sigma}}<\infty,
\quad
\nm F{\maclA _{\mascS ;r}}<\infty 
\quad \text{and}\quad
\nm F{\maclA _{\maclS _{s;r}^\sigma}'}<\infty,
\end{equation*}
respectively, for every $r>0$.
The spaces are equipped with the
projective limit topology with respect to $r>0$,
defined by each class of seminorms, respectively. 

\par

In the same way we
let $\maclA _{s}^\sigma (\cc d)$ for
$s,\sigma \ge \frac 12$,
$\maclA _{\mascS}' (\cc d)$ and
$(\maclA _{0,s}^\sigma )' (\cc d)$ for $s,\sigma >\frac 12$
be the sets of all
$F\in A(\cc d)$ such that
\begin{equation*}
\nm F{\maclA _{\maclS _{s;r}^\sigma}}<\infty, 
\quad
\nm F{\maclA _{\mascS ;r}'}<\infty
\quad \text{and}\quad
\nm F{\maclA _{\maclS _{s;r}^\sigma}'}<\infty ,
\end{equation*}
respectively, for some $r>0$. 
Their topologies are the inductive limit
topologies with respect to $r>0$,
defined by each class of seminorms, respectively. 
We also set
$$
\maclA _{0,s} =\maclA _{0,s}^s
\quad \text{and}\quad
\maclA _{s} =\maclA _{s}^s.
$$
Then
\begin{alignat*}{5}
{\mathfrak V}_d \, &: &\,
\mascS (\rr d) &\to \maclA _{\mascS}(\cc d), &
\qquad
{\mathfrak V}_d \, &: &\,
\mascS '(\rr d) &\to \maclA _{\mascS}'(\cc d), & &
\\[1ex]
\mathfrak V_d \, &: &\,
\maclS _s^\sigma (\rr d) &\to \maclA _s^\sigma (\cc d), &
\qquad
\mathfrak V_d \, &: &\,
(\maclS _s^\sigma )'(\rr d)
&\to (\maclA _s^\sigma )' (\cc d) &
\quad
s,\sigma &\ge \frac 12
\intertext{and}
\mathfrak V_d \, &: &\,
\Sigma _s^\sigma (\rr d)
&\to
\maclA _{0,s}^\sigma (\cc d), &
\qquad
\mathfrak V_d \, &: &\,
(\Sigma _s^\sigma )'(\rr d)
&\to (\maclA _{0,s}^\sigma )' (\cc d), &
\quad
s,\sigma &> \frac 12
\end{alignat*}
are homeomorphisms \cite{Toft18}.

\par

From these homeomorphisms, the fact that the map
\eqref{Eq:L2A2Isometry} is a homeomorphism and
duality properties for Gelfand-Shilov spaces, it follows that
$(\cdo ,\cdo )_{A^2}$ on $\maclA _{1/2}(\cc d)\times
\maclA _{1/2}(\cc d)$ is uniquely extendable to
a continuous sesqui-linear form on
$(\maclA _s^\sigma )'(\cc d)\times \maclA _s^\sigma (\cc d)$.
The dual of $\maclA _s^\sigma (\cc d)$
can be identified with $(\maclA _s^\sigma )'(\cc d)$ through
this form. Similar facts hold for $\maclA _{0,s}^\sigma$
in place of $\maclA _s^\sigma$ at each occurrence.
(Cf. e.{\,}g. \cite{To11,Toft18}.)

\par

Finally let $\maclA _{\flat _1;r}(\cc d)$ and $\maclA _{\flat _\infty ;r}(\cc d)$
for $r > 0$ be the Banach spaces which consist
of all $F\in A(\cc d)$ such that
$$
\nm F{\maclA _{\flat _1;r}}\equiv \nm {F\cdot e^{-r|\cdo |}}{L^\infty}
\quad \text{respectively}\quad
\nm F{\maclA _{\flat _\infty ;r}}\equiv \nm {F\cdot e^{-r|\cdo |^2}}{L^\infty}
$$
is finite, and let $\maclA _{\flat _1}(\cc d)$ be the inductive limit
of $\maclA _{\flat _1;r}(\cc d)$ with respect to $r>0$. Also let 
$\maclA _{0,\flat _\infty}(\cc d)$ and $\maclA _{0,\flat _\infty}'(\cc d)$
be the projective respectively inductive limit topologies of
$\maclA _{\flat _\infty ;r}(\cc d)$ with respect to $r>0$.

\par

It is evident that
$\maclA _{\flat _1}(\cc d)$ is densely embedded in $\maclA _s^\sigma (\cc d)$
for every $s,\sigma \ge \frac 12$, as well as in $\maclA _{0,s}^\sigma (\cc d)$
for every $s,\sigma > \frac 12$. The form
$(\cdo ,\cdo )_{A^2}$ on $\maclA _{\flat _1}(\cc d)\times
\maclA _{\flat _1}(\cc d)$ is uniquely extendable to
a continuous sesqui-linear form on
$A(\cc d)\times \maclA _{\flat _1} (\cc d)$ and the dual
of $\maclA _{\flat _1} (\cc d)$ can be identified with $A(\cc d)$. 
The Fr{\'e}chet space topology of $A(\cc d)$
can be defined by the seminorms
$$
F\mapsto  \sup _{|z|\le N}|F(z)|,\qquad N=1,2,\dots .
$$
(Cf. \cite{Toft18}.)

\par

\begin{rem}\label{Rem:BargmannPilipovicSpaces}
The spaces $\maclA _{\flat _1}(\cc d)$ and $\maclA _{0,\flat _\infty}(\cc d)$
are examples of Bargmann images of special Pilipovi{\'c} spaces, a family
of Fourier invariant topological vector spaces which are smaller than any Fourier
invariant Gelfand-Shilov space, and which were introduced
and investigated in \cite{Toft18}. For any $\sigma >0$, the Bargmann image
of the Pilipovi{\'c} spaces $\maclH _{\flat _\sigma}(\rr d)$ and
$\maclH _{0,\flat _\sigma}(\rr d)$ are given by
\begin{align*}
\maclA _{\flat _\sigma}(\cc d)
& \equiv
\sets {F\in A(\cc d)}{|F(z)|\lesssim e^{r|z|^{\frac {2\sigma}{\sigma +1}}}
\ \text {for some}\ r>0 }
\intertext{respectively}
\maclA _{0,\flat _\sigma}(\cc d)
& \equiv
\sets {F\in A(\cc d)}{|F(z)|\lesssim e^{r|z|^{\frac {2\sigma}{\sigma +1}}}
\ \text {for every}\ r>0 }.
\end{align*}

\par

If $\sigma >1$, then the (strong) duals of $\maclA _{\flat _\sigma}(\cc d)$ and
$\maclA _{0,\flat _\sigma}(\cc d)$ are given by
\begin{align*}
\maclA _{\flat _\sigma}'(\cc d)
& \equiv
\sets {F\in A(\cc d)}{|F(z)|\lesssim e^{r|z|^{\frac {2\sigma}{\sigma -1}}}
\ \text {for every}\ r>0 }
\intertext{respectively}
\maclA _{0,\flat _\sigma}'(\cc d)
& \equiv
\sets {F\in A(\cc d)}{|F(z)|\lesssim e^{r|z|^{\frac {2\sigma}{\sigma -1}}}
\ \text {for some}\ r>0 }
\end{align*}
through a unique extension of the $A^2$ scalar product on
$\maclA _{\flat _1}(\cc d)\times \maclA _{\flat _1}(\cc d)$.
In particular, if $\sigma$ tends to $\infty$, it follows that some of
these conditions tend to
\begin{align*}
\maclA _{0,\flat _\infty}(\cc d)
& \equiv
\sets {F\in A(\cc d)}{|F(z)|\lesssim e^{r|z|^2}
\ \text {for every}\ r>0 }
\intertext{respectively}
\maclA _{0,\flat _\infty}'(\cc d)
& \equiv
\sets {F\in A(\cc d)}{|F(z)|\lesssim e^{r|z|^2}
\ \text {for some}\ r>0 }.
\end{align*}
Note that in \cite{Toft18,Teofanov2}, the set $\maclA _{0,\flat _\infty}(\cc d)$
is denoted by $\maclA _{0,\frac 12}(\cc d)$, and its dual
$\maclA _{0,\flat _\infty}'(\cc d)$ is denoted by $\maclA _{0,\frac 12}'(\cc d)$.
\end{rem}

\par

At many places it will be crucial to use the Gaussian window
\begin{equation}\label{Eq:phidef}
\phi (x)=\pi ^{-\frac d4}e^{-\frac 12|x|^2}, \quad x\in \rr d,
\end{equation}
in the transform $\cT_\phi$. 
For this $\phi$ the relationship between
the Bargmann transform and $\cT _\phi$ is
\begin{equation}\label{bargstft1}
\mathfrak V_d = U_{\mathfrak V}\circ \cT _\phi ,\quad \text{and}\quad
U_{\mathfrak V}^{-1} \circ \mathfrak V_d =  \cT _\phi ,
\end{equation}
where $U_{\mathfrak V}$ is the linear, continuous and bijective operator on
$\mathscr D'(\rr {2d})\simeq \mathscr D'(\cc d)$, given by
\begin{equation}\label{UVdef}
U_{\mathfrak V} F (x+i\xi ) = (2\pi )^{\frac d2} e^{\frac 12(|x|^2+|\xi |^2)}e^{i\scal x\xi}
F(\sqrt 2\, x,-\sqrt 2\, \xi ), \quad x,\xi \in \rr d,
\end{equation}
cf. \cite{To11} in combination with \eqref{Eq:LinkSTFTandTmap}.

\par

In analytic operator theory we need subspaces of
$$
\wideparen A(\cc {2d})
\equiv
\Sets {\Theta K}
{K\in A(\cc {2d})},
$$
where the semi-conjugation operator is 
\begin{equation}\label{Eq:ThetaOpDef}
(\Theta K) (z,w) = K(z,\overline w),\qquad z,w\in \cc d.
\end{equation}
If $T$ is a linear and continuous operator
from $\maclS _{1/2}(\rr d)$ to $\maclS _{1/2}'(\rr d)$,
then there is a unique $K\in \wideparen A(\cc {2d})$
such that $\Theta K\in \maclA _{1/2}'(\cc {2d})$
and $\mathfrak V_d \circ T \circ \mathfrak V_d^{-1}$
is given by
\begin{equation}\label{Eq:CompIntOp}
F(z)\mapsto \int _{\cc d}K(z,w)F(w)\, d\mu (w).
\end{equation}
(See e.{\,}g. \cite{Teofanov2}.) For these reasons we let
$$
\wideparen \maclA _{0,s}(\cc {2d}),
\quad
\wideparen \maclA _s(\cc {2d}),
\quad
\wideparen \maclA _{\mascS}(\cc {2d}),
\quad
\wideparen \maclA _{\mascS}'(\cc {2d}),
\quad
\wideparen \maclA _s'(\cc {2d})
\quad \text{and}\quad
\wideparen \maclA _{0,s}'(\cc {2d}) 
$$
be the images of
$$
\maclA _{0,s}(\cc {2d}),
\quad
\maclA _s(\cc {2d}),
\quad
\maclA _{\mascS}(\cc {2d}),
\quad
\maclA _{\mascS}'(\cc {2d}),
\quad
\maclA _s'(\cc {2d})
\quad \text{and}\quad
\maclA _{0,s}'(\cc {2d}) 
$$
respectively, under the map $\Theta$. We also let $\wideparen A^p(\cc {2d})$
and $\wideparen \maclA _{\flat _1}(\cc {2d})$
be the images of $A^p(\cc {2d})$ and $\maclA _{\flat _1}(\cc {2d})$,
respectively, under the map $\Theta$.
The topologies of the former spaces are
inherited from the corresponding latter spaces.

\par

The semi-conjugated
Bargmann (SCB) transform is defined as
$$
\mathfrak V_{\Theta ,d}
=
\Theta \circ \mathfrak V_{2d}.
$$
All properties of the Bargmann transform
carry over naturally to analogous properties for the SCB transform. 

\par

\subsection{Pseudo-differential operators}\label{subsec1.4}

\par

Let $A$ be a real $d\times d$ matrix.
The \emph{pseudo-differential operator}
$\op _A(\fka)$ with \emph{symbol}
$\fka \in \maclS _{1/2} (\rr {2d})$ is the linear and continuous operator on
$\maclS _{1/2} (\rr d)$ given by
\begin{equation}\label{e0.5}
\op _A(\fka )f (x)
=
(2\pi  ) ^{-d}\iint \fka (x-A(x-y),\xi )
f(y)e^{i\scal {x-y}\xi }\,
dyd\xi, \quad x\in \rr d.
\end{equation}
For $\fka \in \maclS _{1/2}'(\rr {2d})$ the
pseudo-differential operator $\op _A(\fka )$
is defined as the continuous
operator from $\maclS _{1/2}(\rr d)$ to $\maclS _{1/2}'(\rr d)$ with
distribution kernel
\begin{equation}\label{atkernel}
K_{\fka ,A}(x,y)
=
(2\pi )^{-\frac d2} \mascF _2^{-1}\fka (x-A(x-y),x-y), \quad 
x,y \in \rr d, 
\end{equation}
where $\mascF _2F$ is the partial Fourier
transform of $F(x,y)\in
\maclS _{1/2}'(\rr {2d})$ with respect to the $y$
variable. This definition makes sense since the
mappings
\begin{equation}\label{homeoF2tmap}
\mascF _2\quad \text{and}\quad F(x,y)\mapsto F(x,x-y)
\end{equation}
are homeomorphisms on $\maclS _{1/2}'(\rr {2d})$.
The map $a\mapsto K_{\fka ,A}$ is
hence a homeomorphism on $\maclS _{1/2}'(\rr {2d})$.

\par

If $A$ and $B$ are real $d\times d$ matrices and
$\fka \in \maclS _{1/2}'(\rr {2d})$,
then there is a unique $\fkb \in \maclS _{1/2}'(\rr {2d})$ such that
$\op _A(\fka )= \op _B(\fkb )$, and that $\fkb$ can be obtained by
\begin{equation}\label{Eq:RealPseudoCalculiTransfer}
\op _A(\fka )= \op _B(\fkb )
\quad \Leftrightarrow \quad
e^{i\scal {AD_\xi }{D_x}}\fka (x,\xi )
= 
e^{i\scal {BD_\xi }{D_x}}\fkb (x,\xi )
\end{equation}
(see \cite{Ho1,CaTo}).

\par

\begin{rem}\label{BijKernelsOps}
By Fourier's inversion formula, \eqref{atkernel} and the 
kernel theorem
\cite[Theorem 2.2]{LozPer}, \cite[Theorem 2.5]{Teof2} for 
operators from
Gelfand-Shilov spaces to their duals,
it follows that the map $\fka \mapsto \op _A(\fka )$
is bijective from $\maclS _{1/2}'(\rr {2d})$
to the set of all linear and continuous operators from 
$\maclS _{1/2}(\rr d)$ to $\maclS _{1/2}'(\rr {2d})$.
\end{rem}

\par

If $A=0$ then
$\op _A(\fka ) = \op _0(\fka ) = \op (\fka ) = \fka (x,D)$
is the Kohn-Nirenberg or standard representation. 
If $A=\frac 12 I_d$
where $I_d$ is the $d\times d$ identity matrix
then $\op _A(\fka ) = \op ^w(\fka )$ is the Weyl quantization. 
In this paper we use mainly the Weyl quantization and we put 
\begin{equation*}
K_{\fka }^w = K_{\fka ,I_d/2} \ .
\end{equation*}

\par

The Weyl product $\fka  \wpr \fkb$ of two Weyl symbols $\fka ,\fkb
\in \maclS _{1/2}(\rr {2d})$ is defined
as the product of symbols corresponding to operator composition. Thus
\begin{equation*}
\op ^w(\fka \wpr \fkb ) = \op ^w(\fka ) \circ \op ^w(\fkb )
\end{equation*}
and the Weyl product can be extended to larger spaces as long as
composition is well defined. 

\medspace

Next we recall the definition of Wick operators. Suppose that
$a\in \wideparen A(\cc {2d})$ satisfies
\begin{equation}\label{Eq:AntiWickL1Cond}
w\mapsto a(z,w)e^{r|w|-|w|^2} \in L^1(\cc d)
\end{equation}
locally uniformly with respect to $z\in \cc d$ for every $r>0$. Then the
\emph{analytic pseudo-differential operator},
or \emph{Wick operator} $\op _{\mathfrak V}(a)$ with symbol $a$
and acting on $F\in \maclA _{\flat _1}(\cc d)$, is defined by
\begin{align}
\op _{\mathfrak V}(a)F(z) &= \int _{\cc d} a(z,w)F(w)e^{(z,w)}\, d\mu (w),
\quad z \in \cc d.
\label{Eq:AnalPseudo}
\end{align}
(Cf. e.{\,}g. \cite{Berezin71,Fo,Teofanov2,To11, Toft18}.) 
The condition \eqref{Eq:AntiWickL1Cond} and
$F\in \maclA _{\flat _1}(\cc d)$ imply that the integrand on
the right-hand side of \eqref{Eq:AnalPseudo} is well defined. 
The locally uniform
condition \eqref{Eq:AntiWickL1Cond} with respect to $z\in \cc d$ 
implies that $\op _{\mathfrak V}(a)F \in A(\cc d)$.

\par

In \cite{Teofanov2} several extensions and restrictions of
$\op _{\mathfrak V}(a)$ are given. 
The following result follows
from \cite[Theorems 2.7 and 2.8]{Teofanov2}. Here
$\mathcal L(\maclA _{\flat _1}(\cc d),A(\cc d))$
is the space of all linear and continuous operators from
$\maclA _{\flat _1}(\cc d)$ to $A(\cc d)$.

\par

\begin{prop}\label{Prop:ContAnalPseudo}
The map $a\mapsto \op _{\mathfrak V}(a)$
from $\wideparen \maclA _{\flat _1}(\cc {2d})$ to
$\mathcal L(\maclA _{\flat _1}(\cc d),A(\cc d))$ is uniquely
extendable to a bijective map from
$\wideparen A(\cc {2d})$ to 
$\mathcal L(\maclA _{\flat _1}(\cc d),A(\cc d))$.
\end{prop}

\par

Let $L_A(\cc {2d})$ be the set of all $a \in L^1_{\rm loc}(\cc {2d})$ such that
$z\mapsto a(z,w)$ is entire for almost every $w\in \cc d$ and
\begin{equation}\label{Eq:LACond}
w \mapsto
\sup _{\alpha \in \nn d}
\left |
\frac { \partial _z^\alpha a(z,w)\cdot e^{r|w|-|w|^2}}{h^{|\alpha |}\alpha !}
\right | 
\in L^1(\cc d)
\end{equation}
for every $h,r>0$ and $z\in \cc d$. 
If $a\in \wideparen A(\cc {2d})$ satisfies \eqref{Eq:AntiWickL1Cond} then
$a\in L_A(\cc {2d})$ as a consequence of Cauchy's integral formula. 
Thus $L_A(\cc {2d})$ is a relaxation of the former condition. 

\par

If $a\in L_A(\cc {2d})$ then $\op _{\mathfrak V}(a):
\maclA _{\flat _1}(\cc d) \to \maclA _{\flat _1}'(\cc d)=A(\cc d)$
is continuous.
Hence the following result is a straight-forward consequence of Proposition 
\ref{Prop:ContAnalPseudo} and the fact that
$\wideparen {\maclA} _{\flat _1}'(\cc {2d})=\wideparen A(\cc {2d})$.

\par

\begin{prop}\label{Prop:LAOpsIdent}
Let $a\in L_A(\cc {2d})$. Then there is a unique $a_0\in \wideparen A(\cc {2d})$
such that $\op _{\mathfrak V}(a) = \op _{\mathfrak V}(a_0)$ as mappings
from $\maclA _{\flat _1}(\cc d)$ to $\maclA _{\flat _1}'(\cc d)$. It holds
\begin{multline}\label{Eq:AntiWickAnalPseudoRel}
\op _{\mathfrak V}(a) = \op _{\mathfrak V}(a_0)
\\[1ex]
\text{where}\quad
a_0(z,w) = \pi ^{-d} \int _{\cc d}a(z,w_1)e^{-(z-w_1,w-w_1)}\, d\lambda (w_1).
\end{multline}
\end{prop}

%
%

\par

\begin{proof}
The operator $\Pi _A$ defined in \eqref{eq:projection} is the orthogonal
projection from $L^2(d\mu)$ to $A^2(\cc d)$ which is uniquely extendable
to a continuous
map from
\begin{equation}\label{Eq:AntiWickL1SymbClass}
L_{0,A}(\cc d) \equiv
\sets {a_0 \in L^1_{\rm loc}(\cc d)}{w \mapsto a_0(w)e^{r|w|-|w|^2} \in L^1(\cc d)
\ \text{for every}\ r>0}
\end{equation}
to $A(\cc d)$ (see e.{\,}g. \cite{To11}). Hence, if $F,G\in \maclA _{\flat _1}(\cc d)$ and $a_0$ is
given by \eqref{Eq:AntiWickAnalPseudoRel} then
\begin{multline*}
(\op _{\mathfrak V}(a)F,G)_{A^2}
=
((\op _{\mathfrak V}(a)\circ \Pi _A)F,G)_{A^2}
\\[1ex]
= \left (
\int _{\cc d} \left ( \int _{\cc d}
a(\cdo ,w_1)e^{(\cdo ,w_1)}e^{(w_1,w)}\, d\mu (w_1)
\right )
F(w)\, d\mu (w),G \right )_{\!\! A^2}
\\[1ex]
=
\left (
\int _{\cc d} a_0(\cdo ,w)e^{(\cdo ,w)} F(w)\, d\mu (w),G
\right )_{\!\! A^2}
=
(\op _{\mathfrak V}(a_0)F,G)_{A^2},
\end{multline*}
and thus $\op _{\mathfrak V}(a) = \op _{\mathfrak V}(a_0)$ follows. The assertion
now follows from Proposition \ref{Prop:ContAnalPseudo} and the fact that $a_0$ in
the integral formula of \eqref{Eq:AntiWickAnalPseudoRel} defines an element in
$\wideparen A(\cc {2d})$.
\end{proof}

\par

We will also consider \emph{anti-Wick operators} \cite{Fo,Berezin71,Berezin72}
defined by
\begin{equation}\label{Eq:AntiWick}
\op _{\mathfrak V}^{\aw}(a_0)F(z) = \int _{\cc d} a_0(w)F(w)e^{(z,w)}
\, d\mu (w),\quad z \in \cc d,
\end{equation}
when $a_0\in L_{0,A}(\cc d)$ and $F$ belongs to $\maclA _0(\cc d)$, the space of
analytic polynomials on $\cc d$. 
Then $a_0\in L_{0,A}(\cc d)$ if and only if $a(z,w)\equiv a_0(w)$ belongs to
$L_A(\cc {2d})$, and then $\op _{\mathfrak V}^{\aw}(a_0) = \op _{\mathfrak V}(a)$.
Consequently, all results for Wick operators with symbols in
$L_A(\cc {2d})$ hold for anti-Wick operators. In particular, if
$a_0\in L_{0,A}(\cc d)$, then $\op _{\mathfrak V}^{\aw}(a_0):
\maclA _{\flat _1}(\cc d) \to A(\cc d)$ is continuous. 
We denote
the Wick symbol of the anti-Wick operator $\op _{\mathfrak V}^{\aw}(a_0)$
by $a_0^{\aw}$. Then \eqref{Eq:AntiWickAnalPseudoRel} takes the form
\begin{multline}\tag*{(\ref{Eq:AntiWickAnalPseudoRel})$'$}
\op _{\mathfrak V}^{\aw}(a_0) = \op _{\mathfrak V}(a_0^{\aw})
\\[1ex]
\text{where}\quad
a_0^{\aw} (z,w) = \pi ^{-d} \int _{\cc d}a_0(w_1)e^{-(z-w_1,w-w_1)}\, d\lambda (w_1).
\end{multline}

\par

Pseudo-differential operators on $\rr d$ may be transferred to
Wick operators on $\cc d$ by means of the Bargmann transform. 

\par

\begin{defn}\label{Def:AnalPseudo}
Let $\fka \in \maclS _{1/2}'(\rr {2d})$.
\begin{enumerate}
\item the \emph{Bargmann assignment} $\mathsf S_{\mathfrak V}\fka$ of $\fka$ is
the unique element $a\in \wideparen A (\cc {2d})$ which fulfills
\begin{equation}\label{Eq:BargmannAssignment}
\op _{\mathfrak V}(a) = \mathfrak V_d \circ \op^w (\fka )\circ \mathfrak V_d^*
\quad \Leftrightarrow \quad a = \mathsf S_{\mathfrak V}\fka  \text ;
\end{equation}

\vrum

\item the \emph{Bargmann kernel assignment} $K_{\mathfrak V,\fka}$ of $\fka$ is
the unique element $K \in \wideparen A (\cc {2d})$, which is the kernel of
the map $\mathfrak V_d \circ \op^w (\fka )\circ \mathfrak V_d^*$ with respect to
the sesquilinear $A^2$ form.
\end{enumerate}
\end{defn}

\par

By the definitions we have
\begin{equation}\label{SVandSVK}
K_{\mathfrak V,\fka}(z,w) = e^{(z,w)}\mathsf S_{\mathfrak V}\fka (z,w).
\end{equation}

\par

\begin{example}\label{Ex:BargAssignment}
The creation and annihilation operators
$$
2^{-\frac 12}(x_j-\partial _{x_j})
\quad \text{and}\quad
2^{-\frac 12}(x_j+\partial _{x_j}),
$$
are transfered to the operators
\begin{equation}\label{Eq:LinearWickOps}
F\mapsto z_jF
\quad \text{and}\quad
F\mapsto \partial _{z_j}F,
\end{equation}
by the Bargmann transform (see \cite{B1}). 
The Wick symbols of the operators in \eqref{Eq:LinearWickOps}
are $z_j$ and $\overline w_j$, respectively \cite{Berezin71,To11}. By combining these identities
with the fact that the Weyl symbol of $i^{-1}\partial _{x_j}$ equals $\xi _j$
we get
\begin{equation}\label{Eq:BargmannAssignBasicMaps}
\begin{gathered}
\mathsf S_{\mathfrak V}(2^{-\frac 12}(x_j-i\xi _j)) = z_j, 
\qquad
\mathsf S_{\mathfrak V}(2^{-\frac 12}(x_j+i\xi _j)) = \overline w_j,
\\[1ex]
\mathsf S_{\mathfrak V}(x_j) = 2^{-\frac 12}(z_j+\overline w_j) 
\qquad \text{and}\qquad
\mathsf S_{\mathfrak V}(\xi _j) = 2^{-\frac 12}i(z_j-\overline w_j).
\end{gathered}
\end{equation}
\end{example}

\par

We need to compare $K_{\fka}^w$ and $K_{\mathfrak V,\fka}$. 
On the one hand we have for $f,g \in \mascS (\rr d)$
\begin{equation*}
(\op^w (\fka ) f, g)_{L^2(\rr d)} 
= (K_{\fka}^w, g \otimes \overline f)_{L^2(\rr {2d})}
= (\mathfrak V_{2d} K_{\fka}^w, \mathfrak V_{2d} (g \otimes \overline f)_{A^2(\cc {2d})}
\end{equation*}
and on the other hand
\begin{align*}
(\op^w (\fka ) f, g)_{L^2(\rr d)} 
& = (\op _{\mathfrak V}(a) \mathfrak V_d f, \mathfrak V_d g)_{A^2(\cc d)}
\\[1ex]
& = ( K_{\mathfrak V,\fka}, \mathfrak V_d g \otimes \overline{\mathfrak V_d f})_{A^2 (\cc {2d})}
\\[1ex]
& = ( \Theta K_{\mathfrak V,\fka},
\Theta (\mathfrak V_d g \otimes \overline{\mathfrak V_d f} ))_{A^2 (\cc {2d})}. 
\end{align*}
Since 
\begin{align*}
\Theta (\mathfrak V_d g \otimes \overline{\mathfrak V_d f} )(z,w)
= \mathfrak V_d g(z) \overline{\mathfrak V_d f(\overline w)}
= \mathfrak V_{2d} (g \otimes \overline f) (z,w)
\end{align*}
we obtain 
\begin{equation}\label{eq:kernelreal2complex}
K_{\mathfrak V,\fka} = \Theta \mathfrak V_{2d} K_{\fka}^w
= \mathfrak V_{\Theta,d} K_{\fka}^w. 
\end{equation}

\par

\subsection{Symbol classes for pseudo-differential
operators on $\rr d$}\label{subsec1.5}

\par

In order to define a generalized family of Shubin symbol classes
\cite{Shubin1}, we
need to add a restriction of the involved weights. Let $\rho \in [0,1]$, and let
$\mascP _{\Sh ,\rho }(\rr d)$
be the set of all $\omega \in \mascP (\rr d)\cap C^\infty (\rr d)$ such that
for every multi-index $\alpha \in \nn d$,
$$
|\partial ^\alpha \omega (x)|\lesssim \omega (x)\eabs x^{-\rho |\alpha |},
\quad x\in \rr d.
$$
For $\omega \in \mascP _{\Sh ,\rho }(\rr {d})$ the Shubin symbol class
$\Sh _\rho ^{(\omega )}(\rr {d})$ is the set of all $f \in C^\infty (\rr {d})$
such that for every $\alpha \in \nn {2d}$,
$$
|\partial ^\alpha f(x)|\lesssim \omega (x) \eabs x^{-\rho |\alpha |},
\qquad x\in \rr {d}.
$$

\par

Let $\rho \in [0,1]$, $\omega \in \mascP _{\Sh ,\rho }(\rr {2d})$ and
$A$ be a real $d\times d$ matrix.
Then it follows from \cite{Shubin1} or \cite[Section 18.5]{Ho1}
that $e^{i\scal {AD_\xi }{D_x}}$ is a homeomorphism on 
$\Sh _\rho ^{(\omega )}(\rr {2d})$, which implies that the set
$$
\sets {\op _A(\fka )}{\fka \in \Sh _\rho ^{(\omega )}(\rr {2d})}
$$
is independent of the choice of $A$, in view of
\eqref{Eq:RealPseudoCalculiTransfer}. If $B$ is another
real $d\times d$ matrix and
$\fka ,\fkb \in \Sh _\rho ^{(\omega )}(\rr {2d})$ satisfy
\eqref{Eq:RealPseudoCalculiTransfer}, then it follows from
\cite[Section 18.5]{Ho1} that
\begin{gather}
\fka - \fkb \in \Sh _\rho ^{(\omega _\rho)}(\rr {2d}),
\quad \text{where}\quad
\omega _\rho (x,\xi )= \omega (x,\xi )\eabs {(x,\xi )}^{-2\rho}.
\label{Eq:DifferencesShubinSymbols}
\intertext{In particular}
|\fka (x,\xi ) - \fkb (x,\xi ) |
\lesssim \omega (x,\xi )\eabs {(x,\xi )}^{-2\rho}.
\label{Eq:DifferencesShubinSymbols2}
\end{gather}

\par

We also need the symbol classes defined in \cite[Definition~1.8]{AbCaTo}
with symbols satisfying estimates of the form 
\begin{equation}\label{eq:symbolgevrey}
|\partial _x^\alpha \partial _\xi^\beta  \fka (x,\xi )|
\lesssim h^{|\alpha +\beta |} \alpha !^\sigma \beta !^s
e^{r (|x|^{\frac1{s}}+|\xi |^{\frac1{\sigma}})}, \qquad x,\xi  \in \rr d. 
\end{equation}
(See also \cite{CaTo} for the restricted case when $s=\sigma$.)

\par

\begin{defn}\label{Def:GevreyGSSymbols}
Let $s,\sigma >0$. Then
\begin{enumerate}
\item $\Gamma ^{\sigma ,s;0}_{s,\sigma}(\rr {2d})$ consists of all
$\fka \in C^\infty (\rr d)$ such that
\eqref{eq:symbolgevrey} holds for every $h > 0$ and some $r > 0$;

\vrum

\item $\Gamma ^{\sigma ,s}_{s,\sigma ;0}(\rr {2d})$ consists of all
$\fka \in C^\infty (\rr d)$ such that
\eqref{eq:symbolgevrey} holds for some $h > 0$ and every $r > 0$; 

\vrum

\item  $\Gamma ^{\sigma ,s}_{s,\sigma}(\rr {2d})$ consists of all
$\fka \in C^\infty (\rr d)$ such that
\eqref{eq:symbolgevrey} holds for some $h > 0$ and some $r > 0$.
\end{enumerate}
\end{defn}

\par

\begin{rem}
The symbol classes $\Sh _\rho ^{(\omega )}(\rr {2d})$ have isotropic
behaviour with respect to phase space $T^* \rr d\simeq \rr {2d}$, 
and the same holds for the symbol classes in Definition
\ref{Def:GevreyGSSymbols} when $\sigma = s$. 
See also \cite{CaTo} for the restricted case when $s=\sigma$,
and \cite{AbCoTe} for a bilinear extension.
Important classes similar to those given by Definition
\ref{Def:GevreyGSSymbols} are considered in \cite{Prangoski}.
\end{rem}

\par

Pseudo-differential operators with symbols in the classes in
Definition \ref{Def:GevreyGSSymbols} are examples of so called
operators of infinite order.  
These operators are continuous on appropriate Gelfand-Shilov (distribution)
spaces \cite{AbCaTo,CaTo}.
The next result characterizes the symbol classes in
Definition \ref{Def:GevreyGSSymbols} by means of estimates of form
\begin{equation}\label{Eq:STFTgevrey}
|\cT_\psi \fka (x,\xi , \eta ,y)|
\lesssim
e^{r_1(|x|^{\frac1{s}}+|\xi|^{\frac1{\sigma}}) - r_2(|\eta |^{\frac1{\sigma}} + |y|^{\frac1{s}}) },
\quad x,\xi ,y, \eta \in \rr d.
\end{equation}
We omit the proof since the result is a special case of
\cite[Proposition~2.1$'$]{AbCaTo}. 
We refer to 
\cite[Subsection~1.1]{AbCaTo} for the definition of the Gelfand-Shilov spaces
$\maclS _{s,\sigma}^{\sigma ,s}(\rr {2d})$,
$\Sigma _{s,\sigma}^{\sigma ,s}(\rr {2d})$ and their distribution spaces.

\par

\begin{prop}\label{Prop:GSclassesChar}
Let $s,\sigma >0$ and let $\fka \in C^\infty (\rr {2d})$.  
Then the following is true:
\begin{enumerate}
\item if $\psi \in \maclS _{s,\sigma}^{\sigma ,s} (\rr {2d})\setminus 0$, then
$\fka  \in \Gamma ^{\sigma ,s}_{s,\sigma ;0}(\rr {2d})$
if and only if \eqref{Eq:STFTgevrey} holds for every
$r_1 > 0$ and some $r_2 > 0$;

\vrum

\item if $\psi \in \Sigma _{s,\sigma}^{\sigma ,s} (\rr {2d})\setminus 0$, then
$\fka  \in \Gamma ^{\sigma ,s;0}_{s,\sigma }(\rr {2d})$ if and only if
\eqref{Eq:STFTgevrey} holds for some $r_1 > 0$ and all $r_2 > 0$;

\vrum

\item if $\psi \in \Sigma _{s,\sigma}^{\sigma ,s} (\rr {2d})\setminus 0$, then
$\fka \in \Gamma ^{\sigma ,s}_{s,\sigma}(\rr {2d})$
if and only if \eqref{Eq:STFTgevrey} holds for some $r_1 > 0$ and some $r_2 > 0$.
\end{enumerate}
\end{prop}

\par

\subsection{Elliptic, weakly elliptic and hypoelliptic elements in $\Sh ^{(\omega )}_\rho (\rr d)$}

\par

Let $\rho \ge 0$ and $\omega \in \mascP _{\Sh ,\rho}(\rr d)$. Then
$f\in \Sh ^{(\omega )}_\rho (\rr d)$ is called 
\emph{weakly elliptic} of order $\rho _0\ge 0$,
(in $\Sh ^{(\omega )}_\rho (\rr d)$),
or \emph{$\rho _0$-weakly elliptic},
if there is an $R>0$ such that
$$
|f(x)| \gtrsim \eabs x^{-\rho _0}\omega (x),\qquad |x|\ge R.
$$
A weakly elliptic function of order $0$ is called \emph{elliptic}.

\par

Let $A$ and $B$ be real $d\times d$ matrices, $\rho >0$, $\rho _0\in [0,2\rho )$,
$\omega \in \mascP _{\Sh ,\rho}(\rr {2d})$ and suppose that $\fka ,\fkb \in
\Sh _\rho ^{(\omega )}(\rr {2d})$ satisfy \eqref{Eq:RealPseudoCalculiTransfer}.
It follows from \eqref{Eq:DifferencesShubinSymbols}
that $\fka$ is weakly elliptic of order $\rho _0$, if and only if
$\fkb$ is weakly elliptic of order $\rho _0$. In particular,
$\fka$ is elliptic, if and only if $\fkb$ is elliptic.

\par


Next we define Shubin hypoelliptic symbols (cf. Definitions 5.1 and 25.1 in \cite{Shubin1}).

\par

\begin{defn}\label{Def:Hypoelliptic}
Let $\rho >0$, $\rho _0 \ge 0$, $\omega _0\in \mascP _{\Sh ,\rho}(\rr {d})$
and $f\in \Sh _\rho ^{(\omega _0)}(\rr {d})$. Then
$f$ is called \emph{hypoelliptic} (in \emph{Shubin's sense}
in $\Sh _\rho ^{(\omega _0)}(\rr {d})$) of order $\rho _0$,
if there is an $R>0$
such that for every $\alpha \in \nn d$, it holds
\begin{alignat*}{2}
|\partial ^\alpha f(x)| &\lesssim |f(x)|\eabs x^{-\rho |\alpha |}, &
\qquad |x|&\ge R,
\intertext{and}
|f(x)| &\gtrsim \omega _0(x)\eabs x^{-\rho _0}, &
\qquad |x| &\ge R .
\end{alignat*}
\end{defn}

\par

Elliptic and hypoelliptic symbols are important since they give rise to parametrices. 
For $\rho$, $\omega$ as above and $\fka \in \Sh ^{(\omega )}_\rho (\rr {2d})$ elliptic,
there is an elliptic symbol $\fkb \in \Sh ^{(1/\omega )}_\rho (\rr {2d})$
such that
$$
\op _A(\fka )\circ \op _A(\fkb ) =I+\op _A(\fkc _1)
\quad \text{and}\quad
\op _A(\fkb )\circ \op _A(\fka ) =I+\op _A(\fkc _2)
$$
for some $\fkc _1,\fkc _2\in \mascS (\rr {2d})$.
An operator $\op (\fkc )$ with $c\in \mascS (\rr {2d})$
is regularizing in the sense that $\op (\fkc )$ is continuous from $\mascS '(\rr d)$
to $\mascS (\rr d)$. (Cf. e.{\,}g. \cite{BonChe,Shubin1}.)

\par

\section{Reformulation of pseudo-differential calculus
using the Bargmann transform}\label{sec2}

\par

In this section we characterize the Bargmann assignment of
pseudo-diffe\-ren\-ti\-al operator symbols from Subsection \ref{subsec1.5}, 
using estimates of complex derivatives. 
In Subsection \ref{subsec2.1}
we show how pseudo-diffe\-ren\-ti\-al operators on $\rr d$ with Shubin symbols
are transformed to Wick operators by the Bargmann transform. 
In Subsection \ref{subsec2.2}
we deduce similar links between pseudo-diffe\-ren\-ti\-al operators
of infinite order, given in the second part of Subsection \ref{subsec1.5},
and suitable classes of Wick operators.
Subsection \ref{subsec2.3} treats composition formulae for symbols
of Wick operators, which leads to algebraic
properties for operators in Subsection \ref{subsec2.1} and \ref{subsec2.2}.
As an application we obtain short proofs of composition results for
pseudo-differential operators on $\rr d$ from Subsection \ref{subsec1.5}. 

\par

\subsection{Wick symbols of Shubin pseudo-differential
operators}\label{subsec2.1}

\par

The following proposition is essential in the characterization of Shubin
type pseudo-differential operators on $\rr d$ by means of the corresponding
Wick symbols. The Shubin classes can be characterized using the
transform $\cT_\phi$ by means of estimates of the form
\begin{align}
|\partial_x^\alpha \partial_\xi^\beta \cT _\phi f(x,\xi)|
&\lesssim
\omega(x) \eabs{x}^{- \rho |\alpha|} \eabs{\xi}^{-N},
\label{eq:Gineq1}
\\[1ex]
|\partial_x^\alpha\cT _\phi f(x,\xi)|
&\lesssim
\omega(x) \eabs{x}^{- \rho |\alpha|} \eabs{\xi}^{-N}
\label{eq:Gineq2}
\intertext{and}
|\cT _\phi f(x,\xi)|
&\lesssim
\omega(x) \eabs{\xi}^{-N}. 
\label{eq:Gineq3}
\end{align}

The proof of the following result is similar to the proof of
\cite[Proposition~3.2]{Cappiello1}. 

\par

\begin{prop}
\label{prop:symbchar1}
Let $0 \le \rho \leqs 1$, let $\omega \in \mascP _{\Sh ,\rho }(\rr {d})$, 
and suppose $f\in \cS '(\rr d)$ and $\phi \in\cS (\rr d) \setminus 0$. 
The following conditions are equivalent:
\begin{enumerate}
\item $f \in \Sh _\rho^{(\omega )}(\rr d)$,

\vrum

\item \eqref{eq:Gineq1} holds true for any $N \geqs 0$ and
$\alpha, \beta \in \nn d$,

\vrum

\item \eqref{eq:Gineq2} holds true for any $N \geqs 0$ and
$\alpha \in \nn d$, 
\end{enumerate}

\medspace

and the following conditions are equivalent:
\begin{enumerate}
\item[{(1)$'$}] $f \in \Sh _0^{(\omega )}(\rr d)$,

\vrum

\item[{(2)$'$}] \eqref{eq:Gineq3} holds true for any $N \geqs 0$.
\end{enumerate}
\end{prop}

\par

\begin{proof}
First we prove that (1) implies (2).
Suppose $f \in \Sh _\rho^{(\omega )}(\rr d)$ and let
$\alpha,\beta, \gamma \in \nn d$ be arbitrary. We will show
\begin{equation*}
|\xi^\gamma \partial_x^\alpha \partial_\xi^\beta \cT _\phi f(x,\xi)|
\lesssim
\omega(x) \eabs{x}^{-\rho|\alpha|}. 
\end{equation*}

\par

By \eqref{eq:diffident}, \eqref{eq:diffidentstar} and integration by parts we get
\begin{multline*}
|\xi^\gamma \partial_x^\alpha \partial_\xi^\beta \cT _\phi f(x,\xi)|
 =
\left
| \xi^\gamma \cT_{\phi _\beta}(\partial^\alpha f)(x,\xi)
\right |
\\[1ex]
 =
(2 \pi)^{-\frac d2} \left| \int_{\rr d} \left((i\partial_{y})^\gamma e^{- i \la \xi,y \ra }\right)
\overline{\phi _\beta(y)} \, \partial^\alpha f(x+y)\, dy \right |
\\[1ex]
\lesssim
\int_{\rr{d}} \left| \partial_{y}^\gamma \left[\overline{\phi _\beta(y)} \,
\partial^\alpha f(x+y)\right] \right|\, dy
\\[1ex]
 =
\int_{\rr{d}} \left| \sum_{\kappa \leqs \gamma} \binom{\gamma}{\kappa}
\partial^{\gamma-\kappa}
\overline{\phi _\beta(y)} \, \partial^{\alpha+\kappa} f(x+y)\right|\, dy 
\\[1ex]
 \lesssim
\sum_{\kappa \leqs \gamma} \binom{\gamma}{\kappa}  \int_{\rr d} \left |
\partial^{\gamma-\kappa} 
\phi _\beta(y) \right| \, \omega(x+y) \eabs{x+y}^{-\rho|\alpha+\kappa|}\, dy.
\end{multline*}
Since $\omega$ is polynomially moderate, Peetre's inequality
\eqref{eq:Peetre} and the fact that $\phi \in \cS$ give
\begin{multline*}
|\xi^\gamma \partial_x^\alpha \partial_\xi^\beta \cT _\phi f(x,\xi)|
\\[1ex]
\lesssim
\omega(x) \eabs{x}^{-\rho |\alpha |} \sum_{\kappa \leqs \gamma}
\binom{\gamma}{\kappa}
\int_{\rr d} \left| \partial^{\gamma-\kappa} \phi _\beta(y) \right| \, \omega (y)
\, \eabs{y}^{|m|+\rho |\alpha+\kappa|}\, dy
\\[1ex]
\asymp \omega(x) \eabs{x}^{- \rho |\alpha|}. 
\end{multline*}
Thus $f\in \Sh _\rho^{(\omega )}(\rr d)$ implies 
implies \eqref{eq:Gineq1}, and as a special case \eqref{eq:Gineq2}, 
and $f \in \Sh _0^{(\omega )}(\rr d)$ implies \eqref{eq:Gineq3}.  
We have proved that (1) implies (2) which in turn implies (3),
and that (1)$'$ implies (2)$'$. 

\par

Conversely, suppose (3), that is $f \in \cS'(\rr d)$ and \eqref{eq:Gineq2}
holds for all $ N \ge 0$ and all $\alpha \in \nn d$, 
which is a weaker assumption than (2). 
We obtain from \eqref{eq:reproducing}
\begin{align*}
f(x)
& = \|\phi \|_{L^2}^{-2} \, \cT _\phi^* \cT _\phi f(x)
\\[1ex]
& = \|\phi \|_{L^2}^{-2} \, (2 \pi)^{-\frac d2} \iint _{\rr {2d}}
\cT _\phi f(y,\xi) \, e^{i \la \xi,x-y \ra} \, \phi (x-y) \, dyd\xi
\end{align*}
which is an absolutely convergent integral due to \eqref{eq:Gineq2}
and the fact that $\phi \in\cS(\rr d)$.  We may differentiate under the
integral, so integration by parts, \eqref{eq:Gineq2} and Peetre's
inequality give for some $N_0 \ge 0$, any $\alpha \in \nn d$ and any $x \in \rr d$
\begin{multline*}
\left |
\partial^\alpha f(x)
\right |
 =
\|\phi \|_{L^2}^{-2} \, (2 \pi)^{-\frac d2}  \left| \iint _{\rr {2d}} \cT _\phi f(y,\xi)
\, \partial_y^\alpha
\left ( e^{i \la \xi,x-y \ra} \, \phi (x-y) \right ) \, dyd\xi
\right |
\\[1ex]
 =
\|\phi \|_{L^2}^{-2} \, (2 \pi)^{-\frac d2}  \left| \iint _{\rr {2d}} \partial_y^\alpha
\cT _\phi f(y,\xi) \, e^{i \la \xi,x-y \ra} \, \phi (x-y) \, dyd\xi \right |
\\[1ex]
 =
\|\phi \|_{L^2}^{-2} \, (2 \pi)^{-\frac d2}  \left| \iint _{\rr {2d}} \partial_y^\alpha
\cT _\phi f(x-y,\xi) \, e^{i \la \xi,y \ra} \, \phi (y) \, dyd\xi \right |
\\[1ex]
 \lesssim
\iint _{\rr {2d}}  \omega(x-y) \eabs{x-y}^{- \rho |\alpha|} \,
\eabs{\xi}^{-d-1} \, |\phi (y)| \, dyd\xi
\\[1ex]
 \lesssim
\omega (x) \eabs {x}^{- \rho |\alpha|} \iint _{\rr {2d}}  \eabs{\xi}^{-d-1}
\eabs{y}^{N_0+\rho |\alpha|} \, |\phi (y)| \, dyd\xi
\\[1ex]
\asymp
\omega(x) \eabs{x}^{- \rho |\alpha|}.  
\end{multline*}
Thus $f \in \Sh _\rho^{(\omega )}(\rr d)$ and we have proved the
equivalence of (1), (2) and (3). 

\par

It remains to show that (2)$'$ implies (1)$'$, that is
\eqref{eq:Gineq3} for all $N \ge 0$ implies
$f \in \Sh _0^{(\omega )}(\rr d)$. 
We have for some $N_0 \geqs 0$, any $\alpha \in \nn d$, $x \in \rr d$ and $N \geqs 0$,
\begin{align*}
\left |
\partial^\alpha f(x)
\right |
& =
\|\phi \|_{L^2}^{-2} \, (2 \pi)^{-\frac d2}  \left| \iint _{\rr {2d}} \cT _\phi f(y,\xi) \, \partial _x^\alpha
\left ( e^{i \la \xi,x-y \ra} \, \phi (x-y) \right ) \, dyd\xi
\right |
\\[1ex]
& \lesssim
\sum_{\beta \leqs \alpha}
\binom{\alpha}{\beta}
\iint _{\rr {2d}} \left|  \cT _\phi f(y,\xi) \right| \eabs{\xi}^{|\beta|} 
\left| \partial^{\alpha-\beta} \phi (x-y) \right| \, dyd\xi
\\[1ex]
& \lesssim
\sum_{\beta \leqs \alpha}
\binom{\alpha}{\beta}
\iint _{\rr {2d}} \omega(y) \eabs{\xi}^{|\alpha|-N} 
\left| \partial^{\alpha-\beta} \phi (x-y) \right| \, dyd\xi
\\[1ex]
& \lesssim 
\omega (x) \sum_{\beta \leqs \alpha}
\binom{\alpha}{\beta}
\iint _{\rr {2d}} \eabs{\xi}^{|\alpha|-N} 
\eabs{x-y}^{N_0} \left| \partial^{\alpha-\beta} \phi (x-y) \right| \, dyd\xi
\\[1ex]
& \lesssim
\omega(x)
\end{align*}
provided $N$ is sufficiently large, since $\phi \in \cS$. 
This shows that $f \in \Sh _0^{(\omega )}(\rr d)$. 
\end{proof}

\par

We may now characterize the Shubin classes $\Sh _\rho^{(\omega )}(\rr {2d})$
by estimates on their Bargmann (kernel) assignments of the forms
\begin{align}
\big | 
( \partial_z  + \overline \partial _w)^\alpha (\partial _z- \overline \partial _w)
^\beta & {\mathsf S}_{\mathfrak V} \fka (z,w) 
\big |
\notag
\\[1ex]
&\lesssim
e^{\frac1{2} |z-w|^2}\omega (\sqrt 2\, \overline z)
\eabs{z+w}^{- \rho |\alpha + \beta|} \eabs{z-w}^{-N} ,
\label{eq:characShubinBargmann}
\\[1ex]
\left| \partial_z^\alpha \overline \partial _w
^\beta \mathsf S_{\mathfrak V} \fka (z,w) \right |
&\lesssim
e^{ \frac1{2} |z-w|^2}\omega (\sqrt 2\, \overline z)\eabs{z+w}^{- \rho |\alpha + \beta|}
\eabs{z-w}^{-N} ,
\label{eq:characShubinBargmann2}
\\[1ex]
\left | \mathsf S_{\mathfrak V} \fka (z,w) \right|
&\lesssim
e^{ \frac1{2} |z-w|^2} \omega (\sqrt 2\, \overline z)
\eabs{z-w}^{-N}
\label{eq:characShubinBargmann0}
\intertext{and}
\left | K_{\mathfrak V,\fka} (z,w) \right|
&\lesssim
\omega (\sqrt 2\, \overline z)
\eabs{z-w}^{-N} e^{ \frac1{2} \left( |z|^2 + |w|^2 \right)}. 
\label{eq:characShubinBargmannkernel0}
\end{align}

\par

\begin{thm}\label{Thm:ShubinAnalChar}
Let $0 \le \rho \leqs 1$, $\omega \in \mascP _{\Sh ,\rho}(\rr {2d})$ and
$\fka \in \mascS '(\rr {2d})$. 
The following conditions are equivalent:
\begin{enumerate}
\item $\fka \in \Sh ^{(\omega )}_\rho (\rr {2d})$,

\vrum

\item \eqref{eq:characShubinBargmann} holds true for every $N\ge 0$, $z,w\in \cc d$ and $\alpha,\beta \in \nn d$, 

\vrum

\item  \eqref{eq:characShubinBargmann2} holds true for every $N\ge 0$, $z,w\in \cc d$ and $\alpha,\beta \in \nn d$, 
\end{enumerate}

\medspace

and the following conditions are equivalent:
\begin{enumerate}
\item[(1)$'$] $\fka \in \Sh ^{(\omega )}_0 (\rr {2d})$,

\vrum

\item[(2)$'$] \eqref{eq:characShubinBargmann0} holds true for any $N \in \no$
and $z,w\in \cc d$,

\vrum

\item[(3)$'$]  \eqref{eq:characShubinBargmannkernel0} holds true for any $N \in \no$
and $z,w\in \cc d$.
\end{enumerate}
\end{thm}

\par

For the proof we need the following proposition of independent interest. Here
we recall that $\mathsf S_{\mathfrak V}$ is bijective from $\maclS _{1/2}'(\rr {2d})$
to the set
\begin{equation}\label{Eq:SBGelfandShilovLargeDist}
\sets{a\in \wideparen A(\cc {2d})}
{|a(z,w)|\lesssim e^{(\frac 12+r)|z-w|^2}\ \text{for every}\ r>0 }.
\end{equation}

\par

\begin{prop}\label{Prop:SBaTaRel}
Let $\psi (x,\xi ) = (\frac 2\pi )^{\frac d2} e^{-(|x|^2+|\xi |^2)}$, $x,\xi \in \rr d$,
$\fka \in \maclS _{1/2}'(\rr {2d})$ and $a$ belongs to the set in
\eqref{Eq:SBGelfandShilovLargeDist}. Then
\begin{align}
\mathsf S_{\mathfrak V} \fka (z,w)
&=
(2 \pi)^{\frac d2} e^{\frac1{2} |z-w|^2} 
\cT_\psi \fka \left( \frac{x+y}{\sqrt{2}}, - \frac{\xi+\eta}{\sqrt{2}},
\sqrt{2}(\eta-\xi),  \sqrt{2} (y-x) \right),
\label{eq:BargmannsymbolSTFT}
\intertext{and}
(\mathsf S_{\mathfrak V}^{-1}a)(x,-\xi )
&=
\left( \frac{2}{\pi} \right)^d \int _{\cc {d}}
a \left( \frac z{\sqrt 2} - w, \frac z{\sqrt 2} + w \right)
e^{- 2 |w|^2}\, d\lambda (w),
\label{eq:InvBargmannsymbolSTFT}
\end{align}
with $z=x+i\xi $, $w=y+i\eta$ and $x,y,\xi ,\eta \in \rr d$.
\end{prop}

\par

\begin{proof}
Let 
$\phi (x,y) = \pi^{-\frac d2} e^{-\frac1{2}(|x|^2+|y|^2)}$ for $x,y \in \rr {d}$,
and let $K_{\fka}^w$ be the kernel of $\op^w (\fka )$. Then
$\psi =\mascF _2(\phi \circ \kappa)$, where $\kappa (x,y)=(x+y/2,x-y/2)$.
By \eqref{bargstft1} (or \cite[Eq.~(1.35)]{Teofanov2})
and \cite[Lemma~4.1]{Cappiello1}
we have 
\begin{align*}
& \mathfrak V _{\Theta ,d} K_{\fka}^w (z,w) 
= \mathfrak V _{2d} K_{\fka}^w (z,\overline w) 
= \mathfrak V _{2d} K_{\fka}^w ( (x,y) + i (\xi,-\eta))
\\[1ex]
& = (2 \pi)^d e^{\frac1{2} \left( |z|^2 + |w|^2 \right) + i \left( \langle x, \xi \rangle
- \langle y, \eta \rangle \right)} 
\cT_\phi K_{\fka}^w \left( \sqrt{2} (x,y), - \sqrt{2} (\xi,-\eta) \right)
\\[1ex]
& = (2 \pi)^{\frac d2} e^{\frac1{2} \left( |z|^2 + |w|^2 \right) + i \left( \langle y, \xi \rangle -
\langle x, \eta \rangle \right)} 
\cT_\psi \fka \left( \frac{x+y}{\sqrt{2}}, - \frac{\xi+\eta}{\sqrt{2}}, \sqrt{2}(\eta-\xi),
\sqrt{2} (y-x) \right),
\\[1ex]
& = (2 \pi)^{\frac d2} e^{\frac1{2} \left( |z|^2 + |w|^2 \right) + i \, \impart (z,w)} 
\cT_\psi \fka \left( \frac{x+y}{\sqrt{2}}, - \frac{\xi+\eta}{\sqrt{2}}, \sqrt{2}(\eta-\xi),
\sqrt{2} (y-x) \right).
\end{align*}
Together with the identity
$$
|z|^2 + |w|^2 + 2i \, \impart (z,w) = |z-w|^2 +2(z,w)
$$
this gives
\begin{multline}
\mathfrak V _{\Theta ,d} K_{\fka}^w (z,w)
\\[1ex]
= (2 \pi)^{\frac d2} e^{\frac1{2}|z-w|^2 + (z,w)} 
\cT_\psi \fka \left( \frac{x+y}{\sqrt{2}}, - \frac{\xi+\eta}{\sqrt{2}}, \sqrt{2}(\eta-\xi),
\sqrt{2} (y-x) \right ).
\end{multline}
A combination of this identity with \eqref{SVandSVK} and \eqref{eq:kernelreal2complex}
gives \eqref{eq:BargmannsymbolSTFT}.

\par

In order to prove \eqref{eq:InvBargmannsymbolSTFT}, we use
Moyal's formula \eqref{Eq:Moyalsformula}, \eqref{eq:reproducing} and the fact that $\nm \psi {L^2}=1$. 
This implies that the inverse of $\cT _\psi$ is given by
\begin{multline*}
(\cT _\psi ^{-1}F) (x,\xi ) = (\cT _\psi ^*F) (x,\xi )
\\[1ex]
=
(2\pi )^{-d}\iiiint _{\rr {4d}}
F(x_1,\xi _1,\eta _1,y_1)\psi (x-x_1,\xi -\xi _1)
e^{i(\scal {x-x_1}{\eta _1}+\scal {y_1}{\xi -\xi _1}}\, dx_1d\xi _1d\eta _1dy_1.
\end{multline*}
Writing
$$
G(z,w) = F(x,\xi ,\eta ,y),\qquad z=x+i\xi ,\ w=y+i\eta, 
$$
we obtain
\begin{equation}\label{Eq:TInvOp1}
\cT _\psi ^*F(x,\xi )
 =
2^d(2\pi )^{-\frac {3d}2}\iint _{\cc {2d}} G(w_1,w_2)e^{-|z-w_1|^2}
e^{i\impart \scal {z-w_1}{w_2}}\, d\lambda (w_1)d\lambda (w_2). 
\end{equation}

\par

If $\fka = \cT _\psi ^*F$ and $a=\mathsf S_{\mathfrak V} \fka$,
then \eqref{eq:BargmannsymbolSTFT} shows that
$$
a(z,w)
=
(2\pi )^{\frac d2}e^{\frac 12|z-w|^2}G \left( \frac {\overline z+\overline w}{\sqrt 2},\sqrt 2(w-z) \right)
$$
which gives
$$
G(z,w)
=
(2\pi )^{-\frac d2}e^{-\frac 14|w|^2}
a \left( \frac {2\overline z -w}{2\sqrt 2},\frac {2\overline z +w}{2\sqrt 2} \right).
$$

\par

Inserting this into \eqref{Eq:TInvOp1} we get
\begin{multline*}
\cT _\psi ^*F (x,-\xi )
\\[1ex]
=
\frac 1{2^d\pi ^{2d}}\iint _{\cc {2d}}
a \left( \frac {2\overline w_1 -w_2}{2\sqrt 2},\frac {2\overline w_1 +w_2}{2\sqrt 2} \right)
e^{-|\overline z-w_1|^2}e^{-\frac 14|w_2|^2}
e^{i\impart \scal {\overline z-w_1}{w_2}}\, d\lambda (w_1)d\lambda (w_2),
\end{multline*}
and by taking
$$
\frac {2\overline w_1 -w_2}{2\sqrt 2}-\frac z{\sqrt 2}
\quad \text{and}\quad
\frac {2\overline w_1 +w_2}{2\sqrt 2}-\frac z{\sqrt 2}
$$
as new variables of integration, we obtain using \eqref{eq:projection}
\begin{align*}
& \cT _\psi ^*F (x,-\xi )
\\[1ex]
& =
\frac {2^d}{\pi ^{2d}}\iint _{\cc {2d}}
a \left( w_1+\frac z{\sqrt 2},w_2+\frac z{\sqrt 2} \right)
e^{-(|w_1|^2+|w_2|^2)}
e^{2i\impart ({w_1},{w_2})}\, d\lambda (w_1)d\lambda (w_2)
\\[1ex]
& =
2^d \iint _{\cc {2d}}
a \left( w_1+\frac z{\sqrt 2},w_2+\frac z{\sqrt 2} \right)
e^{2i\impart ({w_1},{w_2})}\, d\mu (w_1)d\mu (w_2)
\\[1ex]
& =
2^d \int _{\cc {d}} \left( \int _{\cc {d}}
a \left( w_1+\frac z{\sqrt 2},w_2+\frac z{\sqrt 2} \right)
e^{ (w_1,w_2)} \, e^{ (-w_2,w_1)} \, d\mu (w_1) \right) d\mu (w_2)
\\[1ex]
& =
2^d \int _{\cc {d}} 
a \left( -w_2+\frac z{\sqrt 2},w_2+\frac z{\sqrt 2} \right)
e^{ -|w_2|^2}  \, d\mu (w_2)
\\[1ex]
& =
\left( \frac{2}{\pi} \right)^d \int _{\cc {d}} 
a \left(\frac z{\sqrt 2} - w,\frac z{\sqrt 2} + w \right)
e^{ -2 |w|^2}  \, d\lambda (w). 
\end{align*}
\end{proof}

\par

\begin{proof}[Proof of Theorem \ref{Thm:ShubinAnalChar}]
Combining Propositions \ref{prop:symbchar1} and \ref{Prop:SBaTaRel},
writing $z + w = 2 z + w - z$, we obtain that 
$\fka \in \Sh _\rho^{(\omega )}(\rr {2d})$ if and only if for all $\alpha,\beta \in \nn d$ 
and $N \in \no$ we have
\begin{multline*}
\left | (\partial_x + \partial_y)^\alpha (\partial_\xi + \partial_\eta)^\beta
\left( e^{ - \frac1{2} |z-w|^2} \mathsf S_{\mathfrak V} \fka (z,w) \right)\right |
\\[1ex]
\lesssim
\omega \left( \frac{\overline{z+w}}{\sqrt{2}} \right)
\eabs{z+w}^{- \rho |\alpha + \beta|} \eabs{z-w}^{-N}
\\[1ex]
\lesssim
\omega (\sqrt 2\, \overline z)\eabs{z+w}^{- \rho |\alpha + \beta|} \eabs{z-w}^{-N+k}
\end{multline*}
for some $k \in \no$ that can be absorbed into $N$. 

\par

Note that multi-index powers of the differential operators $\partial_x + \partial_y$ and
$\partial_\xi + \partial_\eta$ acting on the factor $e^{ - \frac1{2} |z-w|^2} = e^{- \frac1{2}
\left( |x-y|^2 + |\xi-\eta|^2 \right)}$ are zero. 
Thus we obtain the equivalent condition
\begin{multline*}
\left| (\partial_x + \partial_y)^\alpha (\partial_\xi + \partial_\eta)^\beta
\mathsf S_{\mathfrak V} \fka (z,w) \right |
\\[1ex]
\lesssim
\omega (\sqrt 2\, \overline z)\eabs{z+w}^{- \rho |\alpha + \beta|}
\eabs{z-w}^{-N} e^{ \frac1{2} |z-w|^2} .
\end{multline*}

\par

Using the (conjugate) analyticity of $\mathsf S_{\mathfrak V} \fka (z,w)$
with respect to $z \in \cc d$ ($w \in \cc d$)
we can formulate this as \eqref{eq:characShubinBargmann}. 
We have now shown the equivalence between (1) and (2). 

\par

The equivalence between (2) and (3)
follows from the binomial formulae
\begin{align*}
(\partial _z+ t\overline \partial _w)^\alpha
&=
\sum _{\gamma \le \alpha}
{\alpha \choose \gamma}t^{|\gamma|}\partial _z^{\alpha -\gamma}
\overline \partial _w^{\gamma},\qquad t\in \{ -1,1\} ,
\\[1ex]
\partial _z^\alpha
&=
2^{-|\alpha |}
\sum _{\gamma \le \alpha}
{\alpha \choose \gamma}
(\partial _z+  \overline \partial _w)^{\alpha -\gamma}
(\partial _z-  \overline \partial _w)^{\gamma}
\intertext{and}
\overline \partial _w^\beta
&=
2^{-|\beta |}
\sum _{\gamma \le \beta}
{\beta \choose \gamma} (-1)^{|\gamma |}
(\partial _z+  \overline \partial _w)^{\beta -\gamma}
(\partial _z-  \overline \partial _w)^{\gamma}. 
\end{align*}

\par

It remains to consider the case $\rho = 0$.
We obtain from 
Propositions \ref{prop:symbchar1} and \ref{Prop:SBaTaRel}
that $\fka \in \Sh _0^{(\omega )}(\rr {2d})$ if and only if for all $N \in \no$ we have
\begin{equation*}
\left | \mathsf S_{\mathfrak V} \fka (z,w) \right|
\lesssim
\omega (\sqrt 2\, \overline z)
\eabs{z-w}^{-N} e^{ \frac1{2} |z-w|^2}, \quad z, \zeta \in \cc d. 
\end{equation*}
%
This shows the equivalence between (1)$'$ and (2)$'$. 

\par

Finally the equivalence of (2)$'$ and (3)$'$
is an immediate consequence of 
\eqref{SVandSVK} and 
$$
|e^{(|z|^2+|w|^2)/2} e^{-(z,w)}| = e^{(|z|^2-2 \, \repart (z,w)+|w|^2)/2}= e^{|z-w|^2/2}.
\qquad \qedhere
$$
\end{proof}

\par

Let $\wideparen \maclA _{\Sh ,\rho}^{(\omega )}(\cc {2d})$,
be the set of all $a\in \wideparen A(\cc {2d})$ such that
\begin{equation}\label{Eq:ShubinWickEstimates}
\left| \partial_z^\alpha \overline \partial _w
^\beta a(z,w) \right |
\le C
e^{ \frac1{2} |z-w|^2}\omega (\sqrt 2\, \overline z)\eabs{z+w}^{- \rho |\alpha + \beta|}
\eabs{z-w}^{-N} ,\quad N\ge 0.
\end{equation}
The smallest constant $C\ge 0$ defines a semi-norm parameterized by $\alpha$,
$\beta$ and $N$,
and we equip $\wideparen \maclA _{\Sh ,\rho}^{(\omega )}(\cc {2d})$ with
the Fr{\'e}chet space topology defined by these semi-norms.
The following result is an immediate consequence of Theorem
\ref{Thm:ShubinAnalChar} and its proof.

\par

\begin{prop}
Let $0 \le \rho \leqs 1$ and $\omega \in \mascP _{\Sh ,\rho}(\rr {2d})$. Then
${\mathsf S}_{\mathfrak V}$ is a homeomorphism from
$\Sh ^{(\omega )}_\rho (\rr {2d})$ to $\wideparen
\maclA _{\Sh ,\rho}^{(\omega )}(\cc {2d})$.
\end{prop}

\par

\subsection{Wick operators corresponding to Gevrey
type pseudo-differential operators}\label{subsec2.2}

\par

Using \eqref{eq:BargmannsymbolSTFT} and \eqref{Eq:GevreyQuasinormMIxed}
we obtain the following theorem expressed with estimates of the form 
\begin{equation}\label{Eq:BargmannSTFT}
|a(z,w)|
\lesssim \exp \left( \frac1{2} |z-w|^2  + r_1 |z+w|_{s,\sigma} - r_2 |z-w|_{s,\sigma}  \right)
\end{equation}
(cf. Definition \ref{Def:GevreyGSSymbols}). The verification is left for the reader.

\par

\begin{thm}\label{Thm:GevreySymbolsBargmTransfer}
The following is true:
\begin{enumerate}
\item if $s,\sigma \ge \frac 12$, then $\mathsf S_{\mathfrak V}$ is homeomorphic
from $\Gamma ^{\sigma ,s}_{s,\sigma ;0}(\rr {2d})$ to the set of all
$a\in \wideparen A(\cc {2d})$ such that \eqref{Eq:BargmannSTFT} holds for all
$r_1 > 0$ and some $r_2 > 0$;

\vrum

\item if $s,\sigma > \frac 12$, then $\mathsf S_{\mathfrak V}$ is homeomorphic
from $\Gamma ^{\sigma ,s ;0}_{s,\sigma}(\rr {2d})$ to the set of all
$a\in \wideparen A(\cc {2d})$ such that 
\eqref{Eq:BargmannSTFT} holds for some $r_1 > 0$ and every $r_2 > 0$;

\vrum

\item if $s,\sigma >\frac 12$, then  $\mathsf S_{\mathfrak V}$ is homeomorphic
from $\Gamma ^{\sigma ,s}_{s,\sigma}(\rr {2d})$ to the set of all
$a\in \wideparen A(\cc {2d})$ such that 
\eqref{Eq:BargmannSTFT} holds for some $r_1 > 0$ and some $r_2 > 0$. 
\end{enumerate}
\end{thm}

\par

\begin{rem}
The restrictions on $s$ and $\sigma$ in Theorem
\ref{Thm:GevreySymbolsBargmTransfer} are needed since
we must choose $\psi$ in \eqref{Eq:STFTgevrey} as 
the Gauss function in Proposition \ref{Prop:SBaTaRel}.
According to the proof of Theorem \ref{Thm:ShubinAnalChar}
this is necessary for the use of 
the formula \eqref{bargstft1}
that relates $\cT _\phi K_{\fka}^w$ and the Bargmann transform
$\mathfrak V _{2d} K_{\fka}^w$.
For this $\psi$ we have $\psi \in \maclS _s^\sigma (\rr d)$
($\psi \in \Sigma _s^\sigma (\rr d)$),
if and only if $s,\sigma \ge \frac 12$ ($s,\sigma > \frac 12$).
\end{rem}

\par

Theorem \ref{Thm:GevreySymbolsBargmTransfer} can
be combined
with continuity results in \cite{AbCaTo} to deduce
continuity of Wick operators 
acting on the Bargmann images
of $\Sigma _s^\sigma (\rr d)$, $\maclS _s^\sigma (\rr d)$,
$(\maclS _s^\sigma )'(\rr d)$
and $(\Sigma _s^\sigma )'(\rr d)$, respectively.
The following result follows by a straight-forward combination of
Theorems 3.8, 3.15 and 3.16 in \cite{AbCaTo},  
\eqref{Eq:BargmannAssignment} and Theorem
\ref{Thm:GevreySymbolsBargmTransfer}.

\par

\begin{prop}
Let $a\in \wideparen A(\cc {2d})$. Then the following
is true:
\begin{enumerate}
\item if $s,\sigma \ge \frac 12$ and
\eqref{Eq:BargmannSTFT} holds for all
$r_1 > 0$ and some $r_2 > 0$, then
$\op _{\mathfrak V}(a)$ is continuous
on $\maclA _s^\sigma (\cc d)$ and on
$(\maclA _s^\sigma )'(\cc d)$;

\vrum

\item if $s,\sigma > \frac 12$ and \eqref{Eq:BargmannSTFT} 
holds for
some $r_1 > 0$ and all $r_2 > 0$, then
$\op _{\mathfrak V}(a)$ is continuous
on $\maclA _{0,s}^\sigma (\cc d)$ and on $(\maclA _{0,s}^\sigma )'(\cc d)$;

\vrum

\item if $s,\sigma >\frac 12$ and \eqref{Eq:BargmannSTFT} holds for some
$r_1 > 0$ and some $r_2 > 0$, then $\op _{\mathfrak V}(a)$ is continuous
from $\maclA _{0,s}^\sigma (\cc d)$ to $\maclA _{s}^\sigma (\cc d)$, and from
$(\maclA _{s}^\sigma )'(\cc d)$ to $(\maclA _{0,s}^\sigma )'(\cc d)$. 
\end{enumerate}
\end{prop}

\par

\subsection{Composition of Wick operators}\label{subsec2.3}

\par

Let $a_1,a_2 \in \wideparen A(\cc {2d})$. If composition is well defined then
the complex twisted product $a_1 \wpr _{\mathfrak V}a_2$ is defined by
$$
\op _{\mathfrak V}(a_1)\circ \op _{\mathfrak V}(a_2) =
\op _{\mathfrak V}(a_1 \wpr _{\mathfrak V}a_2).
$$
By straight-forward computations it follows that
the product $\wpr _{\mathfrak V}$ is given by
\begin{equation}\label{Eq:DefCompTwistProd}
a_1 \wpr _{\mathfrak V} a_2 (z,w)
= \pi ^{-d}\int_{\cc d} a_1(z,u) a_2 (u,w) e^{-(z-u,w-u)}\, d \lambda (u),
\quad z,w \in \cc d,
\end{equation}
provided the integral is well defined. Inserting derivatives,
\eqref{Eq:DefCompTwistProd} takes the form
\begin{multline}\tag*{(\ref{Eq:DefCompTwistProd})$'$}
(\partial _z^{\alpha _1} \overline \partial _w^{\beta _1} a_1) \wpr _{\mathfrak V}
(\partial _z^{\alpha _2} \overline \partial _w^{\beta _2}a_2) (z,w)
\\[1ex]
= \pi ^{-d}\int_{\cc d} (\partial _z^{\alpha _1} \overline \partial _u^{\beta _1} a_1)(z,u)
(\partial _u^{\alpha _2} \overline \partial _w^{\beta _2} a_2)(u,w) e^{-(z-u,w-u)}
\, d \lambda (u), \quad z,w \in \cc d.
\end{multline}

\par

The following lemma is a product rule 
for the complex twisted product. 

\par

\begin{lemma}\label{Lemma:DiffCompTwistProd}
Let $a_1,a_2\in \wideparen A(\cc {2d})$ and suppose 
the integral in \eqref{Eq:DefCompTwistProd}$'$ is well defined for all $z,w\in \cc d$
and all $\alpha _1,\alpha _2,\beta _1,\beta _2\in \nn d$ such that
$$
|\alpha _1+\alpha _2+\beta _1+\beta _2|\le 1.
$$
Suppose also that the integrand in \eqref{Eq:DefCompTwistProd} is zero at infinity. 
Then
\begin{alignat}{2}
\partial _{z_j}(a_1 \wpr _{\mathfrak V}a_2)
&=
(\partial _{z_j}a_1) \wpr _{\mathfrak V}a_2
+
a_1 \wpr _{\mathfrak V}(\partial _{z_j}a_2), &
\quad
j &= 1,\dots ,d
\label{Eq:DiffCompTwistProd1}
\intertext{and}
\overline \partial _{w_j}(a_1 \wpr _{\mathfrak V}a_2)
&=
(\overline \partial _{w_j}a_1) \wpr _{\mathfrak V}a_2
+
a_1 \wpr _{\mathfrak V}(\overline \partial _{w_j}a_2), &
\quad
j &= 1,\dots ,d.
\label{Eq:DiffCompTwistProd2}
\end{alignat}
\end{lemma}

\par

\begin{proof}
If
$$
F_{a_1,a_2}(z,w,u)
=
a_1(z,u) a_2 (u,w) e^{(z,u-w) + (u,w)}
$$
then
$$
\pi ^d(a_1 \wpr _{\mathfrak V}a_2)(z,w)
= \int_{\cc d}F_{a_1,a_2}(z,w,u)e^{-|u|^2}\, d\lambda (u).
$$
This gives
$$
\pi ^d \partial _{z_j}(a_1 \wpr _{\mathfrak V}a_2) (z,w) = b_1(z,w)
+b_2(z,w)-b_3(z,w),
$$
where
\begin{align*}
b_1(z,w)
&=
\int_{\cc d} F_{\partial _{z_j}a_1,a_2}(z,w,u)e^{-|u|^2}\, d\lambda (u),
\\[1ex]
b_2(z,w)
&=
\int_{\cc d} F_{a_1,a_2}(z,w,u)\overline u_je^{-|u|^2}\, d\lambda (u)
\intertext{and}
b_3(z,w)
&=
\overline w_j\int_{\cc d} F_{a_1,a_2}(z,w,u)e^{-|u|^2}\, d\lambda (u)
\\[1ex]
&=
\overline w_j \pi^d (a_1 \wpr _{\mathfrak V}a_2)(z,w).
\end{align*}

\par

The conjugate analyticity of $u \mapsto a_1(z,u)$ and $u \mapsto e^{(z,u-w)}$ implies 
$\partial _{u_j}a_1(z,u) = \partial _{u_j}e^{(z,u-w)} = 0$ which gives
\begin{multline*}
\partial _{u_j}F_{a_1,a_2}(z,w,u)
\\[1ex]
=
\left( a_1(z,u)  \partial _{u_j}a_2 (u,w) 
+
\overline w_ja_1(z,u) a_2 (u,w) \right) e^{(z,u-w) + (u,w)}
\\[1ex]
=
F_{a_1,\partial _{z_j}a_2}(z,w,u) + \overline w_jF_{a_1,a_2}(z,w,u).
\end{multline*}
Consider $b_2(z,w)$. Integration by parts gives
\begin{multline*}
b_2(z,w)
=
\int_{\cc d} F_{a_1,a_2}(z,w,u)\overline u_je^{-|u|^2}\, d\lambda (u)
\\[1ex]
= - \int_{\cc d} F_{a_1,a_2}(z,w,u) \, \partial _{u_j}e^{-|u|^2} \, d\lambda (u)
\\[1ex]
=
\int_{\cc d} \partial _{u_j}F_{a_1,a_2}(z,w,u) e^{-|u|^2}\, d\lambda (u)
\\[1ex]
=
\int_{\cc d} F_{a_1,\partial _{z_j}a_2}(z,w,u)
e^{-|u|^2}\, d\lambda (u)
+
\overline w_j \int_{\cc d} F_{a_1,a_2}(z,w,u)
e^{-|u|^2}\, d\lambda (u)
\\[1ex]
=
\int_{\cc d} F_{a_1,\partial _{z_j}a_2}(z,w,u)
e^{-|u|^2}\, d\lambda (u)
+
b_3(z,w).
\end{multline*}

\par

A combination of these identities now gives
\begin{multline*}
\pi ^d \partial _{z_j} (a_1 \wpr _{\mathfrak V}a_2) (z,w)
\\[1ex]
=
\int_{\cc d} (
F_{\partial _{z_j}a_1,a_2}(z,w,u)
+
F_{a_1,\partial _{z_j}a_2}(z,w,u)
)e^{-|u|^2}\, d\lambda (u)
\\[1ex]
=\pi ^d(\partial _{z_j}a_1) \wpr _{\mathfrak V}a_2(z,w)
+
\pi ^d a_1 \wpr _{\mathfrak V}(\partial _{z_j}a_2) (z,w),
\end{multline*}
and \eqref{Eq:DiffCompTwistProd1} follows.

\par

The assertion \eqref{Eq:DiffCompTwistProd2} is proved by similar arguments. 
\end{proof}

\par

The characterization 
in Theorem \ref{Thm:ShubinAnalChar} (3)
can be applied
to prove the following composition result,  
which is a generalization of \cite[Theorem~23.6]{Shubin1} to include
the case when $\rho = 0$. 

\par

\begin{prop}\label{prop:composition}
Let $0 \le \rho \leqs 1$ and $\omega_j \in \mascP _{\Sh ,\rho}(\rr {2d})$ for $j = 1,2$. 
If $\fka _j \in \Sh_\rho^{(\omega _j)}(\rr {2d})$ for $j = 1,2$, then $\fka _1 \wpr \fka _2
\in \Sh _\rho^{(\omega _1\omega _2)}(\rr {2d})$. 
\end{prop}

\par

\begin{proof}
If $\fka _0 = \fka _1 \wpr \fka _2$ and $a_j = \mathsf S_{\mathfrak V} \fka _j$, $j=0,1,2$, then 
$a_0 = a_1 \wpr _{\mathfrak V} a_2$. 
From Lemma \ref{Lemma:DiffCompTwistProd} and \eqref{Eq:DefCompTwistProd}
we obtain for $\alpha, \beta \in \nn d$,
\begin{align*}
& \partial_z  ^\alpha \overline \partial _w
^\beta a_0 (z,w)
\\[1ex]
& = \sum_{\gamma \leqs \alpha} \sum_{\kappa \leqs \beta}
\binom{\alpha}{\gamma} \binom{\beta}{\kappa}
\left( (\partial _z^{\alpha - \gamma} \overline \partial _w^{\beta-\kappa}
a_1) \wpr _{\mathfrak V} 
(\partial _z^{\gamma} \overline \partial _w^{\kappa}a_2) \right)(z,w)
\\[1ex]
& = \pi^{-d} \sum_{\gamma \leqs \alpha} \sum_{\kappa \leqs \beta}
\binom{\alpha}{\gamma} \binom{\beta}{\kappa}
\int_{\cc d} \partial _z^{\alpha-\gamma} {\overline \partial} _u^{\beta-\kappa} a_1(z,u)
\partial _u^{\gamma} \overline \partial _w^{\kappa} a_2 (u,w) e^{(z,u-w) + (u,w)} d \mu (u). 
\end{align*}

\par

Since $\omega _2 \in \mascP (\rr {2d})\simeq \mascP (\cc {d})$ is moderate, 
Theorem \ref{Thm:ShubinAnalChar} gives for some $N_0 \ge 0$ and any $N_1, N_2 \geqs 0$
\begin{align*}
|\partial _u^{\alpha-\gamma} \overline \partial _w^{\beta-\kappa} a_1(z,u)|
& \lesssim
\omega _1(\sqrt 2\, \overline z)\eabs {z+u}^{-\rho |\alpha+\beta-\gamma -\kappa |}
\eabs {z-u}^{-N_1} e^{\frac1{2} | z - u |^2}
\intertext{and}
|\partial _u^\gamma \overline \partial _w^\kappa a_2(u,w)|
& \lesssim
\omega _2(\sqrt 2\, \overline z)\eabs {z-u}^{N_0}\eabs {u+w}^{-\rho |\gamma +\kappa |}
\eabs {u-w}^{-N_2} e^{\frac1{2} | u - w |^2}. 
\end{align*}
%

\par

This gives
\begin{multline}\label{Eq:Derb0Est}
\left | \partial _z^\alpha \overline \partial _w
^\beta a_0 (z,w) \right|
\\[1ex]
\lesssim
\omega _1(\sqrt 2\, \overline z)\omega _2(\sqrt 2\, \overline z) e^{\frac 12|z-w|^2}
\int_{\cc d} 
F(z,w,u)e^{\Phi (z,w,u)}\, d \lambda (u)
\end{multline}
where for any $N_1 \geqs 0$
\begin{align*}
F(z,w,u)
&=
\eabs{z+u}^{-\rho|\alpha+\beta-\gamma-\kappa|} \eabs{z-u}^{N_0-N_1}
\eabs{u+w}^{-\rho|\gamma+\kappa|} \eabs{u-w}^{-N_2}
\intertext{and}
\Phi (z,w,u)
&=
- \frac1{2} |z-w|^2 + \frac1{2} |z-u|^2 + \frac1{2} |u-w|^2 - |u|^2
\\[1ex]
& \qquad \qquad \qquad \qquad \qquad \qquad \qquad
+ \repart (z,u-w) + \repart (u,w) = 0.
\end{align*}
By Peetre's inequality and the facts that $\gamma \le \alpha$ and $\kappa \le \beta$
we get
\begin{multline*}
\eabs {z+u}^{\rho |\gamma +\kappa|}\eabs{u+w}^{-\rho|\gamma+\kappa|}
\lesssim
\eabs {z-w}^{\rho |\gamma +\kappa|}
\\[1ex]
\lesssim
\eabs {z-u}^{\rho |\gamma +\kappa|}\eabs {u-w}^{\rho |\gamma +\kappa|}
\\[1ex]
\le 
\eabs {z-u}^{\rho |\alpha +\beta|}\eabs {u-w}^{\rho |\alpha +\beta|}
\end{multline*}
and
\begin{equation*}
\eabs {z+u}^{-\rho |\alpha +\beta |}
\lesssim
\eabs {z+w}^{-\rho |\alpha +\beta |}\eabs {u-w}^{\rho |\alpha +\beta |}
\end{equation*}
wherefrom
\begin{equation}\label{Eq:FuncFEst}
F(z,w,u)
\le 
\eabs{z+w}^{-\rho |\alpha +\beta|}
\eabs{z-u}^{\rho |\alpha +\beta| +N_0-N_1}
\eabs{u-w}^{2\rho |\alpha +\beta |-N_2}.
\end{equation}
Hence a combination of \eqref{Eq:Derb0Est} and \eqref{Eq:FuncFEst}
gives for any $N \geqs 0$
\begin{align*}
& (\omega _1(\sqrt 2\, \overline z)\omega _2(\sqrt 2\, \overline z))^{-1}
\eabs{z+w}^{\rho|\alpha+\beta|}
\left| \partial _z^\alpha \overline \partial _w
^\beta a_0 (z,w) \right|
\\[1ex]
& \lesssim
e^{\frac1{2} |z-w|^2}
\int_{\cc d} \eabs{z-u}^{\rho|\alpha+\beta| + N_0 - N_1} \eabs{u-w}
^{2\rho|\alpha+\beta| - N_2}\, d \lambda (u)
\\[1ex]
& \lesssim
 \eabs{z-w}^{-N} e^{\frac1{2} |z-w|^2}
\int_{\cc d} \eabs{z-u}^{\rho|\alpha+\beta| + N_0+N - N_1} \eabs{u-w}
^{2\rho|\alpha+\beta| + N - N_2}\, d \lambda (u). 
\end{align*}
By letting
$$
N_1 \geqs \rho|\alpha+\beta| + N_0 + N
\quad \text{and}\quad 
N_2 > 2\rho|\alpha+\beta| + N  + 2 d
$$
we obtain
\begin{equation*}
\left| \partial _z^\alpha \overline \partial _w^\beta
a_0 (z,w) \right| 
\lesssim
\omega _1(\sqrt 2\, \overline z)\omega _2(\sqrt 2\, \overline z)
\eabs{z+w}^{-\rho|\alpha+\beta|}
\eabs{z-w}^{-N} e^{\frac1{2} |z-w|^2}.
\end{equation*}
According to 
Theorem \ref{Thm:ShubinAnalChar} (3)
this estimate
implies that $\fka _0 \in \Sh _\rho^{(\omega _1\omega _2)}(\rr {2d})$. 
\end{proof}

\par

\begin{rem}
Eq. \eqref{Eq:DefCompTwistProd}
combined with Theorem \ref{Thm:GevreySymbolsBargmTransfer}
can be used to show composition results for pseudo-differential
operators with symbols in $\Gamma ^{\sigma ,s}_{s,\sigma ;0}(\rr {2d})$.
In fact we may use an argument similar to the proof of Proposition
\ref{prop:composition}, but simpler since derivatives can be avoided.  
We obtain
\begin{alignat*}{3}
\fka _1 \wpr \fka _2 \in \Gamma ^{\sigma ,s}_{s,\sigma ;0}(\rr {2d})
\quad \text{when}\quad
\fka _1,\fka _2 \in \Gamma ^{\sigma ,s}_{s,\sigma ;0}(\rr {2d}),
\quad s,\sigma \ge \frac 12,
\end{alignat*}
and similarly with $\Gamma ^{\sigma ,s;0}_{s,\sigma}$
in place of
$\Gamma ^{\sigma ,s}_{s,\sigma ;0}$, provided $\sigma >\frac 12$.
Thereby we regain parts of \cite[Theorem~3.18]{AbCaTo} for certain
restrictions on $s$ and $\sigma$.
\end{rem}

\par


\section{Relations and estimates for Wick and
anti-Wick operators}\label{sec3}

\par

In this section we first show how to approximate a Wick operator by means of 
a sum of anti-Wick operators. 
Then we prove continuity results for anti-Wick operators with symbols having
exponential type bounds.  Finally we deduce estimates for the Wick symbol
of these anti-Wick operators. 


\subsection{Expansion of Shubin type Wick operators with respect to 
anti-Wick operators}\label{subsec3.1}

\par

The first result can be stated for semi-conjugate analytic symbols on $\cc {2d}$. 

\begin{prop}\label{Prop:WickToAntiWick}
Suppose $s\ge \frac 12$, $a \in \wideparen \maclA _s'(\cc {2d})$,
let $N\ge 0$ be an integer, and let
\begin{align*}
a_\alpha (w) &= \partial _z^\alpha \overline \partial _w^\alpha a (w,w), \quad \alpha \in \nn d, 
\intertext{and}
b_\alpha (z,w) &= |\alpha | 
\int _0^1 (1-t)^{|\alpha |-1} \partial _z^\alpha \overline \partial _w^\alpha a
(w+t(z-w),w)\, dt, \quad \alpha \in \nn d\setminus 0.
\end{align*}
Then
\begin{equation}\label{Eq:WickToAntiWick}
\op _{\mathfrak V}(a)
=
\sum _{|\alpha |\le N}
\frac {(-1)^{|\alpha |}\op _{\mathfrak V}^{\aw}(a_\alpha )}{\alpha !}
+
\sum _{|\alpha | = N+1}\frac {(-1)^{|\alpha |}\op _{\mathfrak V}(b_\alpha )}{\alpha !}.
\end{equation}
\end{prop}

\par

\begin{proof}
Taylor expansion gives
\begin{align*}
a(z,w) & = \sum _{|\alpha |\le N} \frac {(-1)^{|\alpha |} c_{\alpha}(z,w)} {\alpha !}
+
\sum _{|\alpha | = N+1}\frac {(-1)^{|\alpha |}c_{0,\alpha}(z,w)} {\alpha !},
\intertext{where}
c_{\alpha} (z,w) &= (-1)^{|\alpha |} (z-w)^\alpha \partial _z^\alpha a (w,w)
\intertext{and}
c_{0,\alpha} (z,w) &= (-1)^{|\alpha |} |\alpha |(z-w)^\alpha 
\int _0^1 (1-t)^{|\alpha |-1} \partial _z^\alpha a (w+t(z-w),w)\, dt. 
\end{align*}
Hence
$$
\op _{\mathfrak V}(a)
=
\sum _{|\alpha |\le N}\frac {(-1)^{|\alpha |}\op _{\mathfrak V}(c_{\alpha} )}{\alpha !}
+
\sum _{|\alpha | = N+1}\frac {(-1)^{|\alpha |}\op _{\mathfrak V}(c_{0,\alpha} )}{\alpha !},
$$
and the result follows if we prove
\begin{equation}\label{Eq:WickAntiWickExpIdent}
\op _{\mathfrak V}(c_{\alpha} )
=
\op _{\mathfrak V}^{\aw}(a_\alpha ) 
\quad \text{and}\quad
\op _{\mathfrak V}(c_{0,\alpha} )
=
\op _{\mathfrak V}(b_\alpha ) .
\end{equation}

\par

It follows from \eqref{Eq:AntiWickAnalPseudoRel} that
$$
\op _{\mathfrak V}(b_\alpha ) = \op _{\mathfrak V}(c_{1,\alpha} )
\quad \text{and}\quad
\op _{\mathfrak V}(c_{0,\alpha} ) = \op _{\mathfrak V}(c_{2,\alpha} )
$$
where
\begin{align}
c_{j,\alpha}(z,w)
&=
(-1)^{|\alpha |} \pi ^{-d} |\alpha |
\int _0^1 (1-t)^{|\alpha |-1} h_{j,\alpha}(a;t,z,w)\, dt, \label{eq:cjdef}
\intertext{$j=1,2$, with}
h_{1,\alpha}(a;t,z,w)
&=
(-1)^{|\alpha |} \int _{\cc d}
\partial _z^\alpha \overline \partial _w^\alpha a
(w_1+t(z-w_1),w_1)
e^{-(z-w_1,w-w_1)}\, d\lambda (w_1) \label{eq:h1def}
\intertext{and}
h_{2,\alpha}(a;t,z,w)
&=
\int _{\cc d}(z-w_1)^\alpha
\partial _z^\alpha a
(w_1+t(z-w_1),w_1)
e^{-(z-w_1,w-w_1)}\, d\lambda (w_1). \nonumber
\end{align}
Since
$$
(z-w_1)^\alpha e^{-(z-w_1,w-w_1)}
=
\overline \partial _{w_1}^\alpha  e^{-(z-w_1,w-w_1)}
$$
integration by parts yields
\begin{multline*}
h_{2,\alpha}(a;t,z,w)
=
\int _{\cc d}
\partial _z^\alpha a
(w_1+t(z-w_1),w_1)
\overline \partial _{w_1}^\alpha e^{-(z-w_1,w-w_1)} \, d\lambda (w_1)
\\[1ex]
=
(-1)^{|\alpha |}
\int _{\cc d}
\partial _z^\alpha \overline \partial _{w}^\alpha a
(w_1+t(z-w_1),w_1)
e^{-(z-w_1,w-w_1)}\, d\lambda (w_1)
=
h_{1,\alpha}(a;t,z,w),
\end{multline*}
and the second equality in \eqref{Eq:WickAntiWickExpIdent}
follows. The first equality in \eqref{Eq:WickAntiWickExpIdent}
follows by similar arguments. The details are left for the reader.
\end{proof}

\par
%
%
%
%

\par

\begin{rem}\label{Rem:WickToAntiWick2}
Proposition \ref{Prop:WickToAntiWick} and its proof
show that
\begin{equation}\tag*{(\ref{Eq:WickToAntiWick})$'$}
\op _{\mathfrak V}(a)
=
\sum _{|\alpha |\le N}
\frac {(-1)^{|\alpha |}\op _{\mathfrak V}^{\aw}(a_{\alpha})}{\alpha !}
+
\sum _{|\alpha | = N+1}\frac {(-1)^{|\alpha |}\op _{\mathfrak V}(c_{1,\alpha})}{\alpha !}
\end{equation}
where $c_{1,\alpha}$ is defined by \eqref{eq:cjdef} and \eqref{eq:h1def}. 
\end{rem}

\par

In the following result we estimate $a_\alpha$ in Proposition \ref{Prop:WickToAntiWick}
and $c_{1,\alpha}$ in \eqref{eq:cjdef}
when $a =\mathsf S_{\mathfrak V} \fka$ satisfies
\eqref{eq:characShubinBargmann2} for every $N\ge 0$ and $\alpha,\beta \in \nn d$. 
By Theorem \ref{Thm:ShubinAnalChar} this means
that $\op _{\mathfrak V}(a)$
is the Bargmann transform of a Shubin type operator.

\par

\begin{prop}\label{Prop:ShubinWickExpEst}
Let $0 \le \rho \le 1$, $\omega \in \mascP _{\Sh ,\rho}(\rr {2d})$,
$a \in \wideparen \maclA _{\Sh ,\rho}^{(\omega )}(\cc {2d})$, 
and let $a_\alpha$ and $b_\alpha$
be as in Proposition \ref{Prop:WickToAntiWick} for $\alpha \in \nn d$. 
Then  $\op _{\mathfrak V}(b_\alpha)
= \op _{\mathfrak V}(c_{1,\alpha})$ for
a unique $c_{1,\alpha}\in \wideparen A(\cc {2d})$,
\begin{align}
|\partial _w^\beta \overline \partial _w^\gamma a_\alpha (w)|
&\lesssim 
\omega (\sqrt 2\overline {w}) \eabs {w}^{-\rho (2 |\alpha| +|\beta +\gamma |)},
\quad
\alpha ,\beta ,\gamma \in \nn d,
\label{Eq:ShubinWickExpEst1}
\intertext{and}
|\partial _z^\beta \overline \partial _w^\gamma  c_{1,\alpha} (z,w)|
&\lesssim
e^{\frac{1}{2} |z-w|^2}
\omega (\sqrt 2\overline {z}) \eabs {z+w}^{-\rho (2 |\alpha| + |\beta +\gamma|)}
\eabs {z-w}^{-N},
\quad
\alpha ,\beta ,\gamma \in \nn d.
\label{Eq:ShubinWickExpEst2}
\end{align}
\end{prop}

\par

\begin{rem}\label{Rem:UniquenessIdentRemainderTerm}
The Wick symbol $c_{1,\alpha}$  in Proposition \ref{Prop:ShubinWickExpEst} is
uniquely defined and given by \eqref{eq:cjdef} in view of
Proposition \ref{Prop:LAOpsIdent}, when $h_{1,\alpha}$ is 
defined by \eqref{eq:h1def}. 
The conditions in Proposition \ref{Prop:ShubinWickExpEst}
imply that $c_{1,\alpha} \in \wideparen \maclA _{\Sh ,\rho}^{(\omega _\alpha )} (\cc {2d})$
where $\omega _\alpha =\eabs \cdo ^{-2\rho |\alpha |}\cdot \omega$.
\end{rem}

\par

\begin{proof}[Proof of Proposition \ref{Prop:ShubinWickExpEst}]
The estimate \eqref{Eq:ShubinWickExpEst1} is an immediate consequence of 
\begin{equation*}
\partial _w^\beta \overline \partial _w^\gamma a_\alpha (w) 
= \partial _w^{\alpha+\beta} \overline \partial _w^{\alpha +\gamma} a (w,w) 
\end{equation*}
and \eqref{eq:characShubinBargmann2}. 

\par

In order to prove \eqref{Eq:ShubinWickExpEst2} we first note that 
the uniqueness assertion for $c_{1,\alpha}$ is a consequence of
Remark \ref{Rem:UniquenessIdentRemainderTerm}.
Let $h_{1,\alpha}(a;z,w)$
be the same as in the proof of Proposition \ref{Prop:WickToAntiWick}.
Integration by parts gives
$$
\partial _z^\beta \overline \partial _w^\gamma h_{1,\alpha}(a;t,z,w)
=
h_{1,\alpha}(\partial _z^\beta \overline \partial _w^\gamma a;t,z,w),
$$
which reduce the problem to prove that \eqref{Eq:ShubinWickExpEst2}
holds for $\beta =\gamma =0$.

\par

The assumption $a \in \wideparen \maclA _{\Sh ,\rho}^{(\omega )}(\cc {2d})$
combined with $\omega$
and $\eabs \cdo ^{-|\alpha |}$ being moderate imply
$$
| \partial_z^\alpha \overline \partial_w^\beta a(z,w) |
\lesssim
e^{\frac1{2} |z-w|^2}
\omega (\sqrt 2\, \overline w)\eabs{w}^{- \rho |\alpha +\beta |}
\eabs{z-w}^{-N}
$$
for every $N\ge 0$. This gives
\begin{multline*}
e^{\repart (z,w)} |h_{1,\alpha} (a;t,z,w)|
\\[1ex]
\lesssim
\int _{\cc d}
\omega (\sqrt 2\overline {w_1})
e^{\frac{t^2}{2} |z-w_1|^2}\eabs {w_1}^{-2\rho |\alpha |}\eabs {t(z-w_1)}^{-N}
e^{\repart (z+w-w_1,w_1)}\, d\lambda (w_1),
\end{multline*}
that is
\begin{multline}\label{Eq:LongEsthalpha}
e^{-\frac 14|z-w|^2} |h_{1,\alpha} (a;t,z,w)|
\\[1ex]
\lesssim
\int _{\cc d}
\omega (\sqrt 2\overline {w_1})
e^{\frac{t^2}{2} |z-w_1|^2}\eabs {w_1}^{-2\rho |\alpha |}\eabs {t(z-w_1)}^{-N}
e^{-|w_1-z_2|^2}\, d\lambda (w_1)
\\[1ex]
=
\int _{\cc d}
\omega (\sqrt 2(\overline {z_2+w_1}))
e^{\frac{t^2}{2} |z_1-w_1|^2}\eabs {z_2+w_1}^{-2\rho |\alpha |}\eabs {t(z_1-w_1)}^{-N}
e^{-|w_1|^2}\, d\lambda (w_1)
\end{multline}
for every $N\ge 0$, where $z_1=\frac 12(z-w)$ and $z_2=\frac 12(z+w)$.

\par

If $t\in [0,\frac 12]$, then the last estimate together with the
moderateness of $\omega$ give
\begin{multline*}
e^{-|z_1|^2} |h_{1,\alpha} (a;t,z,w)|
\lesssim
\omega (\sqrt 2\overline {z_2})\eabs {z_2}^{-2\rho |\alpha |}
\int _{\cc d}
e^{\frac 18|w_1|^2}
e^{\frac{1}{8} |z_1-w_1|^2}
e^{-|w_1|^2}\, d\lambda (w_1)
\\[1ex]
\lesssim
\omega (\sqrt 2\overline {z_2})\eabs {z_2}^{-2\rho |\alpha |} e^{\frac{1}{4} |z_1|^2}\int _{\cc d}
e^{\frac{1}{4} |w_1|^2}
e^{-\frac 78|w_1|^2}\, d\lambda (w_1)
\\[1ex]
\lesssim
\omega (\sqrt 2\overline {z_2})\eabs {z_2}^{-2\rho |\alpha |} e^{\frac{1}{2} |z_1|^2}
\eabs {z_1}^{-N},
\end{multline*}
for every $N\ge 0$. 
The moderateness of $\omega$ again gives
\begin{equation}\label{Eq:halphaEst}
|h_{1,\alpha} (a;t,z,w)|\lesssim e^{\frac{1}{2} |z-w|^2}
\omega (\sqrt 2\overline {z}) \eabs {z+w}^{-2\rho |\alpha |}\eabs {z-w}^{-N}
\end{equation}
or every $N\ge 0$, when $t\in [0,\frac 12]$.

\par

Suppose instead $t\in [\frac 12,1]$. Then
$\eabs {t(z_1-w_1)}^{-N}\asymp \eabs {z_1-w_1}^{-N}$. 
Moderateness again gives
$$
\omega (\sqrt 2(\overline {z_2+w_1}))
\eabs {z_2+w_1}^{-2\rho |\alpha |}\eabs {z_1-w_1}^{-N_0}
\lesssim
\omega (\sqrt 2\overline z)
\eabs {z}^{-2\rho |\alpha |}
$$
for some $N_0$. Hence \eqref{Eq:LongEsthalpha} gives
\begin{multline*}
e^{-|z_1|^2} \omega (\sqrt 2\overline z)^{-1}\eabs {z}^{2\rho |\alpha |}
| h_{1,\alpha} (a;t,z,w)|
\\[1ex]
\lesssim
\int _{\cc d}
e^{\frac{1}{2} |z_1-w_1|^2}\eabs {z_1-w_1}^{-N}
e^{-|w_1|^2}\, d\lambda (w_1)
\\[1ex]
=
e^{|z_1|^2}\int _{\cc d}
\eabs {z_1-w_1}^{-N}
e^{- \frac{1}{2} |w_1+z_1|^2}\, d\lambda (w_1)
\asymp
e^{|z_1|^2}\eabs {z_1}^{-N}
\end{multline*}
for every $N\ge 0$. This gives \eqref{Eq:halphaEst} also for $t\in [\frac 12,1]$.

\par

The result now follows by using \eqref{Eq:halphaEst} when estimating
$|c_{1,\alpha}(z,w)|$ in \eqref{eq:cjdef}
and evaluating the arising integral.
\end{proof}

\par

The next result, analogous to Proposition \ref{Prop:ShubinWickExpEst},
will be useful in Section \ref{sec5} when we discuss hypoellipticity
for Shubin operators in the Wick setting.

\par

\begin{prop}\label{Prop:AsymptoticExpansion1}
Let $\rho \ge 0$, $\omega \in \mascP _{\Sh ,\rho}(\cc d)$,
$\omega _t=\omega \cdot \eabs \cdo ^{-2\rho t}$ when $t\ge 0$,
$a\in \wideparen \maclA _{\Sh ,\rho}^{(\omega )}(\cc {2d})$,
$\fka = \mathsf S_{\mathfrak V}^{-1}a$ and $N\ge 0$ be an integer.
Then
\begin{equation}\label{Eq:ExpansionWickTaylor}
\fka (x,-\xi ) = \sum _{|\alpha |\le N}
\frac {(-1)^{|\alpha |}
(\partial _z^\alpha \overline \partial _w^\alpha a)(2^{-\frac 12}z,2^{-\frac 12}z)}
{2^{|\alpha |}\alpha !} +c_N(z), \quad z = x + i \xi, 
\end{equation}
where 
\begin{equation}\label{Eq:EstExpansionWickTaylor}
\begin{alignedat}{2}
\partial _z^\alpha \overline \partial _w^\alpha a
\in
\wideparen \maclA _{\Sh ,\rho}^{(\omega _{|\alpha |})}(\cc {2d})
\quad &\text{and} &\quad
(x,\xi )\mapsto c_N(x-i\xi ) \in \Sh _{\rho}^{(\omega _{N+1})}(\rr {2d}). 
\end{alignedat}
\end{equation}
\end{prop}

\par

\begin{proof}
The first claim in \eqref{Eq:EstExpansionWickTaylor} $\partial _z^\alpha \overline \partial _w^\alpha a \in
\wideparen \maclA _{\Sh ,\rho}^{(\omega _{|\alpha |})}(\cc {2d})$ is an immediate consequence 
of the definition \eqref{Eq:ShubinWickEstimates} and Peetre's inequality. 

\par

By Taylor expanding the right-hand side of
\eqref{eq:InvBargmannsymbolSTFT} 
we obtain
\begin{equation}\label{Eq:ExpansionWickTaylor1}
\fka (x,-\xi ) = \sum _{|\alpha +\beta |\le 2N+1}
\frac {(2/\pi)^d I_{\alpha ,\beta}\cdot 
(\partial _z^\alpha \overline \partial _w^\beta a)(2^{-\frac 12}z,2^{-\frac 12}z)}
{\alpha !\beta !} +c_N(z),
\end{equation}
where
\begin{equation*}
I_{\alpha ,\beta}
=
\int _{\cc {d}}(-w)^\alpha \overline w^\beta e^{- 2 |w|^2}
\, d\lambda (w),
\end{equation*}
and
\begin{equation}\label{Eq:cNdef}
c_N(z)
=
2(N+1)
\sum _{|\alpha +\beta |= 2N+2}
\frac {(-1)^{|\beta |}}{\alpha !\beta !}
\int _0^1(1-\theta )^{2N+1}
H_{\alpha ,\beta } (z ,\theta )\, d\theta
\end{equation}
with 
\begin{equation}\label{Eq:Hdef}
H_{\alpha ,\beta } (z ,\theta ) =
\left ( \frac 2\pi \right )^d \int _{\cc {d}}
(\partial _z^\alpha \overline \partial _w^\beta a)
\left ( \frac z{\sqrt 2}- \theta w,\frac z{\sqrt 2}+\theta w \right )
w^\alpha \overline w^\beta 
e^{- 2 |w|^2 }\, d\lambda (w).
\end{equation}

\par

The orthonormality of $\{ e_\alpha \} _{\alpha \in \nn d}\subseteq A^2(\cc d)$
(cf. \eqref{Eq:BargmannHermiteMap})
yields $I_{\alpha,\beta} = 0$ if $\alpha \neq \beta$ and 
\begin{align*}
I_{\alpha,\alpha} 
& = \int _{\cc {d}}(-w)^\alpha \overline w^\alpha e^{- 2 |w|^2} \, d\lambda (w)
\\[1ex]
& = (-1)^{|\alpha|} 2^{-d -|\alpha|} \alpha! \pi^d \int _{\cc {d}} |e_\alpha(w) |^2
\, d\mu (w)
\\[1ex]
& = (-1)^{|\alpha|} 2^{-d -|\alpha|} \alpha! \pi^d. 
\end{align*}
Comparing \eqref{Eq:ExpansionWickTaylor1} with \eqref{Eq:ExpansionWickTaylor}
we see that the sum in the latter formula has been proven correct. 

It remains to study the remainder $c_N$.
We need to prove that $\fkc (x,\xi )=c_N(x-i\xi )$ belongs to
$\Sh _{\rho}^{(\omega _{N+1})}(\rr {2d})$. If
$$
h_{\alpha ,\beta } (z ,w,\theta ) =
(\partial _z^\alpha \overline \partial _w^\beta a)
\left ( \frac z{\sqrt 2}- \theta w,\frac z{\sqrt 2}+\theta w \right )
w^\alpha \overline w^\beta e^{- 2 |w|^2 }
$$
then
$$
H_{\alpha ,\beta } (z ,\theta )
=
\left ( \frac 2\pi \right )^d \int _{\cc {d}}
h_{\alpha ,\beta } (z ,w,\theta )
\, d\lambda (w).
$$

\par

First we notice that
\begin{align*}
\partial _z^\alpha \overline \partial _z^\beta c_N(z)
&=
2(N+1)
\sum _{|\gamma +\delta |= 2N+2}
\frac {(-1)^{|\delta |}}{\gamma !\delta !}
\int _0^1(1-\theta )^{2N+1}
\partial _z^{\alpha}\overline \partial _z^{\beta}
H_{\gamma ,\delta } (z ,\theta )\, d\theta, 
\\[1ex]
\partial _z^\alpha \overline \partial _z^\beta H_{\gamma ,\delta} (z,\theta )
&=
\left ( \frac 2\pi \right )^d
\int _{\cc {d}} \partial _z^\alpha \overline \partial _z^\beta h_{\gamma ,\delta}
(z,w,\theta ) \, d\lambda (w)
\intertext{and}
\partial _z^\alpha \overline \partial _z^\beta h_{\gamma ,\delta}
(z,w,\theta )
&=
2^{-\frac {|\alpha +\beta |}2}(\partial _z^{\alpha +\gamma}
\overline \partial _z^{\beta +\delta}
a)\left( \frac z{\sqrt 2} - \theta w,\frac z{\sqrt 2} + \theta w \right )w^\gamma
\overline w^\delta e^{- 2 |w|^2 }.
\end{align*}
From the definition \eqref{Eq:ShubinWickEstimates} this implies that
for every $M\ge 0$ and some $M_0\ge 0$ we have
\begin{multline*}
|\partial _z^\alpha \overline \partial _z^\beta h_{\gamma ,\delta}
(z,w,\theta )|
\\[1ex]
\lesssim
e^{-2(1-\theta ^2)|w|^2}\omega (\overline z-\sqrt 2\theta \overline w)
\eabs {z}^{-\rho (|\alpha +\beta |+2N+2)}\eabs {\theta w}^{-M-M_0}|w|^{2N+2}
\\[1ex]
\lesssim
e^{-2(1-\theta )|w|^2}\omega (\overline z)
\eabs {z}^{-\rho (|\alpha +\beta |+2N+2)}\eabs {\theta w}^{-M}|w|^{2N+2}.
\end{multline*}
This gives
$$
|\partial _z^\alpha \overline \partial _z^\beta H_{\gamma ,\delta} (z,\theta )|
\lesssim
\omega (\overline z)
\eabs {z}^{-\rho (|\alpha +\beta |+2N+2)}\cdot J(\theta ),
$$
where
$$
J(\theta ) = \int _{\cc d} e^{-2(1-\theta )|w|^2}\eabs {\theta w}^{-M}|w|^{2N+2}
\, d\lambda (w).
$$

\par

For $\theta \in [0,\frac 12]$ we get
$$
J(\theta ) = \int _{\cc d} e^{-|w|^2}|w|^{2N+2}\, d\lambda (w),
$$
which is finite and independent of $\theta$. If instead $\theta \in [\frac 12,1]$,
and choosing $M>2d+2N+2$, then
$$
J(\theta ) \le
\int _{\cc d} \eabs {\theta w}^{-M}|w|^{2N+2}\, d\lambda (w),
$$
which is again finite and independent of $\theta$.

\par

A combination of these estimates give
$$
|\partial _z^\alpha \overline \partial _z^\beta H_{\gamma ,\delta} (z,\theta )|
\lesssim
\omega (\overline z)
\eabs {z}^{-\rho (|\alpha +\beta |+2N+2)},
$$
which in turn implies
$$
|\partial _z^\alpha \overline \partial _z^\beta c_N(z)|
\lesssim
\omega (\overline z)
\eabs {z}^{-\rho (|\alpha +\beta |+2N+2)}.
$$
This means that $\fkc \in \Sh _{\rho}^{(\omega _{N+1})}(\rr {2d})$.
\end{proof}

\par

\subsection{Continuity of anti-Wick operators with exponentially
bounded symbols}\label{subsec3.2}

\par

Next we consider anti-Wick symbols that satisfy
exponential bounds of the form
\begin{align}
|a_0 (w)|
&\lesssim
e^{-r_0 |w|^{\frac 1s}},
\label{Eq:NonAWGSEst1}
\intertext{or}
|a_0 (w)|
&\lesssim
e^{r_0 |w|^{\frac 1s}}.
\label{Eq:NonAWGSEst2}
\end{align}

\par

In order to formulate our results we introduce new spaces of entire functions. 
Let $s > \frac{1}{2}$, $t_0, r > 0$, and 
let $\maclA _{s, t_0,r}(\cc d)$ be the Banach space of all $F \in A(\cc d)$ such that 
\begin{equation*}
\nm F{\maclA _{s, t_0,r}} \equiv
\nm {F\cdot e^{- t_0 |\cdo |^2 + r |\cdo|^{\frac{1}{s}}}}{L^\infty} 
< \infty.
\end{equation*}
Set
\begin{equation*}
\maclA _{0,(s, t_0)}(\cc d) = \bigcap_{r > 0} \maclA _{s, t_0,r} (\cc d)
\quad \text{and}\quad
\maclA _{(s, t_0)}'(\cc d) = \bigcap_{r > 0} \maclA _{s, t_0,-r} (\cc d)
\end{equation*}
equipped with the projective limit topology. 
Likewise we set 
\begin{equation*}
\maclA _{(s, t_0)}(\cc d) = \bigcup_{r > 0} \maclA _{s, t_0,r} (\cc d)
\quad \text{and}\quad
\maclA _{0,(s, t_0)}'(\cc d) = \bigcup_{r > 0} \maclA _{s, t_0,-r} (\cc d)
\end{equation*}
equipped with the inductive limit topology. 

\par

Referring to Section \ref{subsec1.3} it is clear that 
the spaces $\maclA _{0,(s, t_0)}(\cc d)$, $\maclA _{(s, t_0)}(\cc d)$, $\maclA _{(s, t_0)}'(\cc d)$
and $\maclA _{0,(s, t_0)}'(\cc d)$ are generalizations of 
\begin{align*}
\maclA _{0,(s, \frac{1}{2})}(\cc d) &= \mathfrak V_d ( \Sigma_s (\rr d) ) = \maclA _{0,s}(\cc d)
\\[1ex]
\maclA _{(s, \frac{1}{2})}(\cc d) &= \mathfrak V_d ( \maclS _s (\rr d) ) = \maclA _{s}(\cc d)
\\[1ex]
\maclA _{(s, \frac{1}{2})}'(\cc d) &= \mathfrak V_d ( \maclS _s '(\rr d) ) = \maclA _{s}'(\cc d)
\intertext{and}
\maclA _{0,(s, \frac{1}{2})}'(\cc d) &= \mathfrak V_d ( \Sigma_s' (\rr d) ) = \maclA _{0,s}'(\cc d),
\end{align*}
respectively. 

\par

\begin{prop}\label{prop:contAW}
Let $a_0\in L^\infty _{loc}(\cc d)$, $0 < t_0 < 1$ and
\begin{equation}\label{eq:t0t1}
t_1 = \frac{1}{4(1-t_0)}.
\end{equation}
Then the following is true:
\begin{enumerate}
\item if \eqref{Eq:NonAWGSEst2} holds for some $r_0 > 0$ then 
\begin{equation}\label{Eq:ContAW1}
\begin{alignedat}{2}
\op _{\mathfrak V}^{\aw}(a_0) &: & \maclA _{0,(s, t_0)}(\cc d) &\to \maclA _{0,(s, t_1)}(\cc d),
\\[1ex]
\op _{\mathfrak V}^{\aw}(a_0) &: & \maclA _{0,(s, t_0)}'(\cc d) &\to \maclA _{0,(s, t_1)}'(\cc d)
\end{alignedat}
\end{equation}
are continuous;

\vrum

\item if \eqref{Eq:NonAWGSEst2} holds for every $r_0 > 0$ then 
\begin{equation}\label{Eq:ContAW2}
\begin{alignedat}{2}
\op _{\mathfrak V}^{\aw}(a_0) &: &\maclA _{(s, t_0)}(\cc d) &\to \maclA _{(s, t_1)}(\cc d),
\\[1ex]
\op _{\mathfrak V}^{\aw}(a_0) &: &\maclA _{(s, t_0)}'(\cc d) &\to \maclA _{(s, t_1)}'(\cc d)
\end{alignedat}
\end{equation}
are continuous. 
\end{enumerate}
\end{prop}

\par

\begin{proof}
We only prove that the first map in \eqref{Eq:ContAW1} is continuous. The
other continuity assertions follow by similar arguments and are left
for the reader.

\par

Let $r_2 > 0$ be given, $r_1 > r_0$ and $F\in \maclA _{0,(s, t_0)}(\cc d)$. 
We have for $z \in \cc d$
\begin{align*}
& |\op _{\mathfrak V}^{\aw}(a_0)F(z)| e^{- t_1 |z |^2 + r_2 |z|^{\frac{1}{s}}}
\\[1ex]
& \lesssim e^{- t_1 |z |^2 + r_2 |z|^{\frac{1}{s}}} 
\int_{\cc d} |a_0 (w)| \, |F(w)| \, e^{\repart (z,w) - |w|^2}\,  d\lambda (w)
\\[1ex]
& \lesssim e^{- t_1 |z |^2 + r_2 |z|^{\frac{1}{s}}}  \nm F{\maclA _{s, t_0,r_1}}
\int_{\cc d} e^{r_0 |w|^{\frac{1}{s}}
+ t_0 |w|^2 - r_1 |w|^{\frac{1}{s}} + \repart (z,w) - |w|^2}\,  d\lambda (w)
\\[1ex]
& = e^{r_2 |z|^{\frac{1}{s}}}  \nm F{\maclA _{s, t_0,r_1}}
\int_{\cc d} e^{-(r_1-r_0) |w|^{\frac{1}{s}}  -(1- t_0) |w|^2 +
\repart (z,w) - t_1 |z |^2 }\,  d\lambda (w)
\\[1ex]
& = e^{r_2 |z|^{\frac{1}{s}}}  \nm F{\maclA _{s, t_0,r_1}}  
\int_{\cc d} e^{-(r_1-r_0)
|w|^{\frac{1}{s}}  - \left |  \sqrt{1- t_0} w - \frac{1}{2 \sqrt{1- t_0} } z \right|^2 }
\,  d\lambda (w)
\\[1ex]
& = e^{r_2 |z|^{\frac{1}{s}}}  \nm F{\maclA _{s, t_0,r_1}}  
\int_{\cc d} e^{-(r_1-r_0) \left| w + \frac{1}{2 (1- t_0) } z \right |
^{\frac{1}{s}}  - (1- t_0)\left|  w  \right|^2 }\,  d\lambda (w)
\\[1ex]
& \le
e^{(r_2 - c_1(r_1 - r_0)) |z|^{\frac{1}{s}}}  
\nm F{\maclA _{s, t_0,r_1}}  
\int_{\cc d} e^{ c_2 (r_1-r_0) | w |^{\frac{1}{s}} - (1- t_0)\left|  w  \right|^2 }\, d\lambda (w)
\\[1ex]
&\asymp 
\nm F{\maclA _{s, t_0,r_1}}  
e^{(r_2 - c_1(r_1 - r_0)) |z|^{\frac{1}{s}}}  
\end{align*}
for some constants $c_1, c_2>0$. By choosing $r_1$ sufficiently large we 
get
\begin{equation*}
\nm {\op _{\mathfrak V}^{\aw}(a_0) F} {\maclA _{s, t_1,r_2}}  
\lesssim
\nm F{\maclA _{s, t_0,r_1}}. 
\end{equation*}
The estimates and \eqref{Eq:AntiWick} imply $\op _{\mathfrak V}^{\aw}(a_0)F \in A(\cc d)$. 
\end{proof}

\par

\begin{rem}
Note that \eqref{eq:t0t1} implies $t_1 > \frac{1}{4}$ 
and $t_0 \le t_1$ with equality if and only if $t_0 = \frac{1}{2}$. 
Hence $\maclA _{0,(s, t_0)}(\cc d) \subseteq \maclA _{0,(s, t_1)}(\cc d)$,
and similarly for the other spaces.
\end{rem}

\par

The particular case $t_0 = \frac{1}{2}$ gives

\par

\begin{cor}\label{cor:contBargmannGS}
Let $a_0\in L^\infty _{loc}(\cc d)$.
If \eqref{Eq:NonAWGSEst2} holds for some (every) $r_0 > 0$ 
then $\op _{\mathfrak V}^{\aw}(a_0)$ is continuous on $\maclA _{0,s}(\cc d)$
(on $\maclA _{s}(\cc d)$). 
\end{cor}

\par

With a technique similar to the proof of Proposition \ref{prop:contAW}
one shows the following result. 

\par

\begin{prop}\label{prop:contAW2}
Let $a_0\in L^\infty _{loc}(\cc d)$, $0 < t_0 < 1$ and suppose \eqref{eq:t0t1} holds. 
Then the following is true:
\begin{enumerate}
\item if \eqref{Eq:NonAWGSEst1} holds for all $r_0 > 0$ then 
\begin{equation}\label{Eq:RegAW1}
\op _{\mathfrak V}^{\aw}(a_0): \maclA _{0,(s, t_0)}'(\cc d) \to \maclA _{0,(s, t_1)}(\cc d)
\end{equation}
is continuous;

\vrum

\item if \eqref{Eq:NonAWGSEst1} holds for some $r_0 > 0$ then 
\begin{equation}\label{Eq:RegAW2}
\op _{\mathfrak V}^{\aw}(a_0): \maclA _{(s, t_0)}'(\cc d) \to \maclA _{(s, t_1)}(\cc d)
\end{equation}
is continuous. 
\end{enumerate}
\end{prop}

\par

Again the particular case $t_0 = \frac{1}{2}$ gives

\par

\begin{cor}\label{cor:regBargmannGS}
Let $a_0\in L^\infty _{loc}(\cc d)$. Then the following is true:
\begin{enumerate}
\item if \eqref{Eq:NonAWGSEst1} holds for all $r_0 > 0$
then 
\begin{equation*}
\op _{\mathfrak V}^{\aw}(a_0): \maclA _{0,s}'(\cc d) \to \maclA _{0,s}(\cc d)
\end{equation*}
is continuous;

\item if \eqref{Eq:NonAWGSEst1} holds for some $r_0 > 0$
then 
\begin{equation*}
\op _{\mathfrak V}^{\aw}(a_0): \maclA _{s}'(\cc d) \to \maclA _{s}(\cc d)
\end{equation*}
is continuous.  
\end{enumerate}
\end{cor}

\par

\subsection{Estimates of Wick symbols of anti-Wick operators with
exponentially bounded symbols}\label{subsec3.3}

\par

For anti-Wick operators in \cite[Eq.~(2.94)]{Fo} we have the following result.

\par

\begin{thm}\label{Thm:WickSymbolAntiWickOpFolland}
If $a_0\in L^\infty _{loc}(\cc d)$ satisfies 
\begin{equation}\label{Eq:ExpQuadratic}
|a_0(w)| \lesssim e^{r|w|^2}, \quad w \in \cc d, \quad \text{for some} \quad r < 1, 
\end{equation}
then $a_0\in L_{0,A}(\cc d)$ and \eqref{Eq:AntiWickAnalPseudoRel}$'$ holds for some
$a_0^{\aw}\in \wideparen A(\cc {2d})$ with
\begin{equation*}
|a_0^{\aw}(z,w)|\lesssim e^{r_0 |z+w|^2-\operatorname{Re}(z,w)},
\qquad
r_0=4^{-1}(1-r)^{-1}.
\end{equation*}
\end{thm}

\par

\begin{proof}
The claim $a_0\in L_{0,A}(\cc d)$ is an immediate consequence of the assumption \eqref{Eq:ExpQuadratic}
and the definition \eqref{Eq:AntiWickL1SymbClass}. 
The integral in \eqref{Eq:AntiWickAnalPseudoRel}$'$ can be estimated as
\begin{multline*}
\left |
\int _{\cc d}a_0(w_1)e^{-(z-w_1,w-w_1)}\, d\lambda (w_1)
\right |
\\[1ex]
\lesssim
\int _{\cc d} e^{r|w_1|^2} \left | e^{-(z-w_1,w-w_1)} \right |
\, d\lambda (w_1)
\\[1ex]
=
e^{-\operatorname{Re}(z,w)}
\int _{\cc d} e^{-(1-r)|w_1|^2} e^{\operatorname{Re}(z+w,w_1)}
\, d\lambda (w_1)
\\[1ex]
=
e^{\frac 1{4(1-r)}|z+w|^2-\operatorname{Re}(z,w)}
\int _{\cc d} e^{-(1-r)|w_1-(z+w)/(2(1-r))|^2}
\, d\lambda (w_1)
\\[1ex]
\asymp e^{r_0|z+w|^2-\operatorname{Re}(z,w)} .\qedhere
\end{multline*}
\end{proof}

\par

\begin{rem}
The condition on $a_0^{\aw}$ in Theorem \ref{Thm:WickSymbolAntiWickOpFolland}
implies that $a_0^{\aw}$ belongs to
$\wideparen {\maclA} _{0,\frac 12}^\prime (\cc {2d})$
(see \cite{Teofanov2}).
In particular it follows that $\op _{\mathfrak V}^{\aw}(a_0)=\op _{\mathfrak V}(a_0^{\aw})$
is continuous from $\maclA _{0,\frac 12}(\cc d)$ to $\maclA _{0,\frac 12}'(\cc d)$
(cf. \cite[Theorem 2.10]{Teofanov2} and
Remark \ref{Rem:BargmannPilipovicSpaces}).
\end{rem}


\par

The following result concerns exponentially moderate weight functions. 

\par

\begin{thm}\label{Thm:WickSymbolAntiWickOpSOmega}
Let $a_0\in L_{0,A}(\cc d)$, $a_0^{\aw} \in \wideparen A(\cc {2d})$
is given by \eqref{Eq:AntiWickAnalPseudoRel}$'$
and $\omega \in \mascP _E(\cc d)$. If
\begin{equation*}
|a_0(w)|\lesssim  \omega (2w), \quad w \in \cc d, 
\end{equation*}
then
\begin{equation*}
|a_0^{\aw}(z,w)|\lesssim e^{\frac 14|z-w|^2} \omega (z+w), \quad z,w \in \cc d. 
\end{equation*}
\end{thm}

\par

\begin{proof}
Let $r\ge 0$ be chosen such that $\omega (z+w)\lesssim \omega (z)e^{r|w|}$,
$z,w\in \cc d$. From \eqref{Eq:AntiWickAnalPseudoRel}$'$ we get
\begin{multline*}
|a_0^{\aw}(z,w)|
\lesssim
\int _{\cc d} \omega (2w_1)
e^{-\operatorname{Re}(z-w_1,w-w_1)}\, d\lambda (w_1)
\\[1ex]
=
e^{-\operatorname{Re}(z,w)}
\int _{\cc d} 
\omega (2w_1)
e^{\operatorname{Re}(z+w,w_1)-|w_1|^2}\, d\lambda (w_1)
\\[1ex]
=
e^{-\operatorname{Re}(z,w)+\frac 14|z+w|^2}
\int _{\cc d} 
\omega (2w_1)
e^{-|w_1-(z+w)/2|^2}\, d\lambda (w_1)
\\[1ex]
=
e^{\frac 14|z-w|^2}
\int _{\cc d} 
\omega (2w_1+z+w)
e^{-|w_1|^2}\, d\lambda (w_1)
\\[1ex]
\lesssim
e^{\frac 14|z-w|^2}
\omega (z+w)\int _{\cc d}  e^{2r|w_1|-|w_1|^2}\, d\lambda (w_1)
\asymp
e^{\frac 14|z-w|^2}
\omega (z+w). \qedhere
\end{multline*}
\end{proof}

\par

The anti-Wick operators in Propositions \ref{prop:contAW} and
\ref{prop:contAW2} can also be described as Wick operators
with symbols that have smaller growth bounds than 
$\wideparen \maclA _s(\cc {2d})$ and its dual.
The following result extends Theorem \ref{Thm:WickSymbolAntiWickOpSOmega}
for weights of the form $e^{c |z|^{\frac1s}}$ with $c \in \ro$
from $s \ge 1$ to $s  \ge \frac 12$. 

\par

\begin{thm}\label{Thm:WickSymbolAntiWickOpGS}
Let $s\ge \frac 12$ ($s> \frac 12$), $a_0\in L_{0,A}(\cc {d})$ and
let $a_0^{\aw}$ be given by \eqref{Eq:AntiWickAnalPseudoRel}$'$.
Then the following is true:
\begin{enumerate}
\item  if \eqref{Eq:NonAWGSEst1} holds for
some (every) $r_0 > 0$  then
\begin{equation}\label{Eq:AWGSEst1}
|a_0^{\aw} (z,w)|
\lesssim
e^{\frac 14|z-w|^2-r |z+w|^{\frac 1s}}
\end{equation}
for some (every) $r>0$;

\vrum

\item if \eqref{Eq:NonAWGSEst2} holds for every (some) $r_0>0$ then
\begin{equation}\label{Eq:AWGSEst2}
|a_0^{\aw} (z,w)|
\lesssim
e^{\frac 14|z-w|^2+r |z+w|^{\frac 1s}}
\end{equation}
for every (some) $r>0$. 
\end{enumerate}
\end{thm}

\par

\begin{rem}
Thanks to the parameter $\frac 14$ in the factor $e^{\frac 14|z-w|^2}$
rather than $\frac 12$, the estimates \eqref{Eq:AWGSEst2} are much stronger 
than the estimates \eqref{Eq:BargmannSTFT} with $\sigma = s$. 
Corollary \ref{cor:contBargmannGS} can thus be seen as a
consequence of Theorems \ref{Thm:GevreySymbolsBargmTransfer}
and \ref{Thm:WickSymbolAntiWickOpGS}, and  \cite[Definition~2.4,
and Theorems~4.10 and 4.11]{CaTo}. 
\end{rem}

\par

\begin{rem}\label{Rem:WickSymbolAntiWickOpGS}
The estimates for $a_0^{\aw}$ in Theorem
\ref{Thm:WickSymbolAntiWickOpGS} may seem
weak since the dominating factor $e^{\frac 14|z-w|^2}$
is present in \eqref{Eq:AWGSEst1} and \eqref{Eq:AWGSEst2}
but absent in the original estimates \eqref{Eq:NonAWGSEst1}
and \eqref{Eq:NonAWGSEst2} for $a_0$.

\par

On the other hand, Wick symbols for operators with
continuity involving the spaces $\maclA _s(\cc d)$ and
$\maclA _s'(\cc d)$, as well as  $\maclA _{0,s}(\cc d)$ and
$\maclA _{0,s}'(\cc d)$, usually satisfies conditions of the
form
$$
|a(z,w)| \lesssim e^{\frac 12|z-w|^2 \pm r_1|z+w|^{\frac 12}\pm |z-w|^{\frac 1s}}
$$
in view of \cite[Theorems~2.9 and 2.10]{Teofanov2},
and Theorem \ref{Thm:GevreySymbolsBargmTransfer}. Here
the dominating factor is $e^{\frac 12|z-w|^2}$, which is larger
than the factor $e^{\frac 14|z-w|^2}$ in Theorem
\ref{Thm:WickSymbolAntiWickOpGS}.

\par

This factor has a large impact on functions on $\rr d$ that are transformed back 
by the inverse of the Bargmann transform. 
For instance, if $\ep >0$, then the Bargmann image of any
non-trivial Gelfand-Shilov space and its distribution space
contain
\begin{equation*}
\sets {F\in A(\cc d)}{|F(z)|\lesssim e^{(\frac 12-\ep )|z|^2}}
\end{equation*}
and are contained in 
\begin{equation*}
\sets {F\in A(\cc d)}{|F(z)|\lesssim e^{(\frac 12+\ep )|z|^2}}. 
\end{equation*}
The same holds true for the Bargmann images of
$\mascS (\rr d)$ and $\mascS '(\rr d)$.
\end{rem}

\par

Theorem \ref{Thm:WickSymbolAntiWickOpGS} is
a straight-forward consequence of the following two
propositions, which give more details on the relationships between
$r$ and $r_0$ in
\eqref{Eq:NonAWGSEst1}, \eqref{Eq:NonAWGSEst2},  \eqref{Eq:AWGSEst1}
and \eqref{Eq:AWGSEst2}.

\par

\begin{prop}\label{Prop:WickSymbolAntiWickOpGS1}
Let $s\ge \frac 12$ and let $r_0,r\in (0,\infty )$ be
such that
\begin{alignat}{5}
r_0&\in (0,\infty ) &
\quad &\text{and} & \quad
r&<\frac {r_0}{4(1+r_0)}, &
\quad
\quad &\text{when} & \quad
s &={\textstyle{\frac 12}},
\label{Eq:CondAntiWickOpGS1}
\intertext{and}
r_0&\in (0,\infty ) &
\quad &\text{and} & \quad
r &\le 2^{-\frac 1s}r_0, &
\quad
\quad &\text{when} & \quad
s&\in (\textstyle{\frac 12},\infty ),
\label{Eq:CondAntiWickOpGS2}
\end{alignat}
with strict inequality in \eqref{Eq:CondAntiWickOpGS2}
when $s<1$.
If $a_0\in L^\infty _{loc}(\cc {d})$ satisfies 
\eqref{Eq:NonAWGSEst1} and $a_0^{\aw} \in \wideparen A(\cc {2d})$
is given by \eqref{Eq:AntiWickAnalPseudoRel}$'$, then
\eqref{Eq:AWGSEst1} holds.
\end{prop}

\par

\begin{prop}\label{Prop:WickSymbolAntiWickOpGS2}
Let $s\ge \frac 12$ and $r_0,r\in (0,\infty )$ be
such that
\begin{alignat}{5}
r_0&\in (0,1) &
\quad &\text{and} & \quad
r&>\frac {r_0}{4(1-r_0)}, &
\quad
\quad &\text{when} & \quad
s &={\textstyle{\frac 12}},
\tag*{(\ref{Eq:CondAntiWickOpGS1})$'$}
\intertext{and}
r_0&\in (0,\infty ) &
\quad &\text{and} & \quad
r &\ge 2^{-\frac 1s}r_0, &
\quad
\quad &\text{when} & \quad
s&\in (\textstyle{\frac 12},\infty ),
\tag*{(\ref{Eq:CondAntiWickOpGS2})$'$}
\end{alignat}
with strict inequality in \eqref{Eq:CondAntiWickOpGS2}$'$
when $s<1$.
If $a_0\in L^\infty _{loc}(\cc {d})$ satisfies 
\eqref{Eq:NonAWGSEst2} and $a_0^{\aw} \in \wideparen A(\cc {2d})$
is given by \eqref{Eq:AntiWickAnalPseudoRel}$'$, then
\eqref{Eq:AWGSEst2} holds.
\end{prop}

\par

For the proofs of Propositions \ref{Prop:WickSymbolAntiWickOpGS1}
and \ref{Prop:WickSymbolAntiWickOpGS2}
we use the inequalities
\begin{alignat}{3}
|z|^\theta -|w|^\theta
&\le 
|z+w|^\theta
\le
|z|^\theta +|w|^\theta ,&
\qquad
\theta &\in (0,1], &\ z,w &\in \cc d
\label{Eq:TriangPowIneq1}
\\[1ex]
|z+w|^\theta
&\le
(1+\ep )|z|^\theta +(1+\ep ^{-1})|w|^\theta ,&
\qquad
\theta &\in [1,2], &\ z,w &\in \cc d,
\label{Eq:TriangPowIneq2}
\intertext{and}
|z+w|^\theta
&\ge
(1-\ep )|z|^\theta +(1-\ep ^{-1})|w|^\theta ,&
\qquad
\theta &\in [1,2], &\ z,w &\in \cc d,
\label{Eq:TriangPowIneq3}
\end{alignat}
for every $\ep >0$.

\par

\begin{proof}[Proof of Proposition \ref{Prop:WickSymbolAntiWickOpGS1}]
Suppose that $a_0$ satisfies
\eqref{Eq:NonAWGSEst1} for some $r_0 >0$.
First we consider the case $s>\frac 12$.
If $s < 1$ let $\ep _1 >  0$ and $\ep _2 = \ep_1^{-1}$, 
and if $s \ge 1$ let $\ep _1 = 0$ and $\ep _2 = 2$, 
and let $c=2^{-\frac 1s}$.
Then \eqref{Eq:AntiWickAnalPseudoRel}$'$,
\eqref{Eq:TriangPowIneq1}
and \eqref{Eq:TriangPowIneq3} give
\begin{multline}\label{Eq:WickAWickComp1}
|a_0^{\aw}(z,w)|
\lesssim
\int _{\cc d} e^{-r_0 |w_1|^{\frac 1s}} 
e^{-\operatorname{Re}(z-w_1,w-w_1)}\, d\lambda (w_1)
\\[1ex]
=
e^{\frac 14|z+w|^2-\operatorname{Re}(z,w)}
\int _{\cc d} e^{-r_0 |w_1|^{\frac 1s} -|w_1-(z+w)/2|^2}
\, d\lambda (w_1)
\\[1ex]
=
e^{\frac 14|z-w|^2}
\int _{\cc d} e^{-r_0 |w_1+(z+w)/2|^{\frac 1s} -|w_1|^2}
\, d\lambda (w_1)
\\[1ex]
\le
e^{\frac 14|z-w|^2}e^{-cr_0(1-\ep _1) |z+w|^{\frac 1s}}
\int _{\cc d} e^{-r_0(1-\ep _2) |w_1|^{\frac 1s} -|w_1|^2}
\, d\lambda (w_1)
\\[1ex]
\asymp
e^{\frac 14|z-w|^2}e^{-cr_0(1-\ep _1) |z+w|^{\frac 1s}}.
\end{multline}
If $s\ge 1$, then $\ep _1=0$ and $\ep _2=2$,
and the result follows from \eqref{Eq:WickAWickComp1}.
If instead $s<1$, then
the result follows by choosing $\ep _1>0$ small enough,
and we have proved the result in the case $s>\frac 12$.

\par

Next suppose that $s=\frac 12$. 
For $\ep _1 >  0$ and $\ep _2 = \ep_1^{-1}$
\eqref{Eq:WickAWickComp1} gives
\begin{equation*}
|a_0^{\aw}(z,w)|
\lesssim
e^{\frac 14|z-w|^2}
e^{-\frac 14r_0(1-\ep _1) |z+w|^2}
\int _{\cc d} e^{-(r_0(1-\ep _2)+1) |w_1|^2}
\, d\lambda (w_1).
\end{equation*}
For any $\ep _2<\frac {1+r_0}{r_0}$ it follows that the integral 
converges, and
$$
1-\ep _1=1-\ep _2^{-1}<(1+r_0)^{-1}. 
$$
By the assumptions there is $\delta > 0$ such that 
\begin{equation*}
r = \frac{r_0(1-\delta)}{4 (1+r_0)}. 
\end{equation*}
Since
$$
1-\ep _1 \nearrow (1+r_0)^{-1}
\quad \text{as}\quad
\ep _2 \nearrow \frac {1+r_0}{r_0}
$$
we may pick $0 < \ep _2<\frac {1+r_0}{r_0}$ such that 
\begin{equation*}
\frac{1-\delta}{1+r_0} \le 1-\ep _1 
\end{equation*}
and the result follows in the case $s=\frac 12$. 
\end{proof}

\par

\begin{proof}[Proof of Proposition \ref{Prop:WickSymbolAntiWickOpGS2}]
First we consider the case when
$s>\frac 12$.
Suppose that $a_0$ satisfies \eqref{Eq:NonAWGSEst2}
for some $r_0 >0$, let $\ep _1,\ep _2\ge 0$ be
such that $\ep _1=\ep _2=0$ when $s\ge 1$ and
$\ep_1\ep _2=1$ when $s<1$, and let $c=2^{-\frac 1s}$. Then
\eqref{Eq:AntiWickAnalPseudoRel}$'$, \eqref{Eq:TriangPowIneq1}
and \eqref{Eq:TriangPowIneq2} give
\begin{multline}\label{Eq:WickAWickComp3}
|a_0^{\aw}(z,w)|
\lesssim
\int _{\cc d} e^{r_0 |w_1|^{\frac 1s}} 
e^{-\operatorname{Re}(z-w_1,w-w_1)}\, d\lambda (w_1)
\\[1ex]
=
e^{\frac 14|z+w|^2-\operatorname{Re}(z,w)}
\int _{\cc d} e^{r_0 |w_1|^{\frac 1s} -|w_1-(z+w)/2|^2}
\, d\lambda (w_1)
\\[1ex]
=
e^{\frac 14|z-w|^2}
\int _{\cc d} e^{r_0 |w_1+(z+w)/2|^{\frac 1s} -|w_1|^2}
\, d\lambda (w_1)
\\[1ex]
\le
e^{\frac 14|z-w|^2}e^{cr_0(1+\ep _1) |z+w|^{\frac 1s}}
\int _{\cc d} e^{r_0(1+\ep _2) |w_1|^{\frac 1s} -|w_1|^2}
\, d\lambda (w_1)
\\[1ex]
\asymp
e^{\frac 14|z-w|^2}e^{cr_0(1+\ep _1) |z+w|^{\frac 1s}}.
\end{multline}
If $s\ge 1$, then $\ep _1=\ep _2=0$, and the result follows
from \eqref{Eq:WickAWickComp3}. If instead $s<1$, then
the result follows by choosing $\ep _1>0$ small enough,
and the result follows in the case $s>\frac 12$.

\par

Next suppose that $s=\frac 12$. Then
\eqref{Eq:WickAWickComp3} gives
\begin{equation*}
|a_0^{\aw}(z,w)|
\lesssim
e^{\frac 14|z-w|^2}
e^{\frac 14r_0(1+\ep _1) |z+w|^2}
\int _{\cc d} e^{r_0(1+\ep _2) |w_1|^2 -|w_1|^2}
\, d\lambda (w_1).
\end{equation*}
For any $\ep _2<\frac {1-r_0}{r_0}$ the integral converges, and 
$$
1+\ep _1=1+\ep _2^{-1}>(1-r_0)^{-1}. 
$$
Since
$$
1+\ep _1 \searrow (1-r_0)^{-1}
\quad \text{as}\quad
\ep _2 \nearrow \frac {1-r_0}{r_0},
$$
the result follows in the case $s=\frac 12$
by letting $r=\frac {r_0(1+\ep _1)}4$. 
\end{proof}

\par

\section{A lower bound for Wick operators}\label{sec4}

\par

In this section we apply the asymptotic expansions in the previous section
for Shubin-Wick operators to deduce a sharp G{\aa}rding inequality.

\par

First we have the following result. We put
$\wideparen \maclA _{\Sh ,\rho}(\cc {2d})
=
\wideparen \maclA _{\Sh ,\rho}^{(\omega )}(\cc {2d})$ when $\omega =1$.

\par

\begin{prop}\label{Prop:ContApSpaces}
Let $\omega \in \mascP (\cc d)$, $p\in [1,\infty]$,
$a\in \wideparen \maclA _{\Sh ,0}(\cc {2d})$
and $a_0\in L^\infty (\cc d)$. Then $\op _{\mathfrak V}(a)$ and
$\op _{\mathfrak V}^{\aw}(a_0)$ are both continuous on
$A^p_{(\omega )}(\cc d)$.
\end{prop}

\par

The claimed continuity of $\op _{\mathfrak V}(a)$ 
is a straight-forward consequence of
\cite[Theorem 3.3]{Teofanov2}, in combination with
Proposition \ref{prop:symbchar1} and the relationship
$ K(z,w)=a(z,w) e^{(z,w)} $ between the kernel and symbol of a
Wick operator (cf. \eqref{Eq:AnalPseudoIntro}). In order to be
self-contained we include an alternative and shorter proof.

\par

\begin{proof}
Let $F\in A^p_{(\omega )}(\cc d)$, $G(z)
=
e^{-\frac 12|z|^2}|F(z)\omega (\sqrt 2\overline z)|$,
\begin{align*}
H_1(z) & = e^{-\frac 12|z|^2}|\op _{\mathfrak V}(a)F(z)\omega (\sqrt 2\overline z)|
\quad \text{and}\quad \\
H_2(z) & = e^{-\frac 12|z|^2}|\op _{\mathfrak V}^{\aw}(a_0)F(z)\omega (\sqrt 2\overline z)|.
\end{align*}
We have
$$
\omega (\sqrt 2\overline z) \lesssim \omega (\sqrt 2\overline w)\eabs {z-w}^{N_0}
$$
for some $N_0\ge 0$. By Theorem \ref{Thm:ShubinAnalChar} and \eqref{eq:characShubinBargmann0} we get
\begin{multline*}
H_1(z) \lesssim e^{-\frac 12|z|^2}\int _{\cc d} e^{\frac 12|z-w|^2}
\eabs {z-w}^{-N}|F(w)\omega (\sqrt 2\overline z)|
e^{\repart (z,w)-|w|^2}\, d\lambda (w)
\\[1ex]
= (\eabs \cdo ^{N_0-N}*G)(z),
\end{multline*}
for every $N\ge 0$. By choosing $N>2d+N_0$ and using Young's inequality we get
$\nm {H_1}{L^p}\lesssim \nm {G}{L^p}$ which means
$\nm {\op _{\mathfrak V}(a)F}{A^p_{(\omega )}}\lesssim \nm F{A^p_{(\omega )}}$,
and the asserted continuity for $\op _{\mathfrak V}(a)$ follows.

\par

In the same way we get
\begin{multline*}
H_2(z) \lesssim \nm {a_0} {L^\infty} e^{-\frac 12|z|^2}
\int _{\cc d} |F(w)\omega (\sqrt 2\overline w)|\eabs {z-w}^{N_0}
e^{\repart (z,w)-|w|^2}\, d\lambda (w)
\\[1ex]
\asymp ((\eabs \cdo ^{N_0}e^{-\frac 12|\cdo |^2})*G)(z),
\end{multline*}
and another application of Young's inequality shows that 
$\nm {H_2}{L^p_{(\omega )}}\lesssim \nm {G}{L^p_{(\omega )}}$ that is
$\nm {\op _{\mathfrak V}^{\aw}(a_0)F}{A^p_{(\omega )}}\lesssim \nm F{A^p_{(\omega )}}$.
\end{proof}

\par

We have finally a version of the sharp G{\aa}rding inequality.

\par

\begin{thm}\label{Thm:ShGarding}
Let $\rho >0$, $\omega (z) = \eabs z ^{2\rho}$ and let
$a\in \wideparen \maclA _{\Sh ,\rho}^{(\omega )}(\cc {2d})$ be such that
$a(w,w)\ge -C_0$ for all $w \in \cc d$, for some constant $C_0\ge 0$. Then
\begin{alignat}{2}
\repart \big ( (\op _{\mathfrak V}(a)F,F)_{A^2} \big )
&\ge
-C\nm F{A^2}^2, &
\qquad
F &\in \maclA _{\mascS}(\cc d)
\label{Eq:ShGarding1}
\intertext{and}
\big | \operatorname{Im} \big ( (\op _{\mathfrak V}(a)F,F)_{A^2} \big ) \big |
&\le
C\nm F{A^2}^2, &
\qquad
F &\in \maclA _{\mascS}(\cc d)
\label{Eq:ShGarding2}
\end{alignat}
for some constant $C\ge 0$.
\end{thm}

\par

\begin{proof}
Let $b_0(w)=a(w,w)$. Then
$\op _{\mathfrak V}(a) = \op _{\mathfrak V}^{\aw}(b_0)+\op _{\mathfrak V}(a_1)$
for some $a_1\in \wideparen \maclA _{\Sh ,\rho}(\cc {2d})
\subseteq \wideparen \maclA _{\Sh ,0}(\cc {2d})$, in view of Proposition
\ref{Prop:ShubinWickExpEst}. 
Since $\Pi_A F = F$ for $F \in A^2(\cc d)$ (cf. \eqref{eq:projection}),
the assumption $b_0\ge - C_0$ implies
$(\op _{\mathfrak V}^{\aw}(b_0)F,F)_{A^2}\ge - C_0
\| F \|_{A^2}^2 $ for every $F\in \maclA _{\mascS}(\cc d)$.
The operator $\op _{\mathfrak V}(a_1)$ is continuous on $A^2(\cc d)$
in view of Proposition \ref{Prop:ContApSpaces}. A combination of these facts
gives the result.
\end{proof}

\par

\section{Ellipticity and hypoellipticity for Shubin and
Wick operators}\label{sec5}

\par

In this section we show that the Bargmann assignment 
$\mathsf S_{\mathfrak V}$ maps the sets of hypoelliptic symbols and
weakly elliptic symbols
in the Shubin class $\Sh _\rho ^{(\omega )}(\rr {2d})$
bijectively into the sets of hypoelliptic symbols and 
weakly elliptic Wick symbols
in $\wideparen \maclA _{\Sh ,\rho}^{(\omega )}(\cc {2d})$, respectively.
Then we explain some consequences for polynomial symbols.

\par

\subsection{Transition of weakly elliptic symbols}\label{subsec4.1}

For symbols in
$\wideparen \maclA _{\Sh ,\rho} ^{(\omega )}(\cc {2d})$
we define ellipticity and weak ellipticity as follows.

\par

\begin{defn}\label{Def:WickShubinEllipticity}
Let $\rho >0$, $\omega \in \mascP _{\Sh ,\rho}(\cc d)$ and 
$a\in \wideparen \maclA _{\Sh ,\rho} ^{(\omega)}(\cc {2d})$.
Then $a$ is called \emph{weakly elliptic}
of order $\rho _0\ge 0$, or \emph{$\rho _0$-weakly elliptic}, if for some $R > 0$
$$
|a(z,z)|\gtrsim \eabs z^{-\rho _0}\omega (\sqrt 2\overline z), \quad |z| \ge R.
$$
If $a$ is weakly elliptic of order $0$ then $a$ is called \emph{elliptic}.
\end{defn}

\par


\begin{thm}\label{Thm:ElliptEquiv}
Let $\omega \in \mascP (\rr {2d})\simeq \mascP (\cc d)$,
$\rho >0$ and $\fka \in \Sh _\rho ^{(\omega )}(\rr {2d})$.
Then the following is true:
\begin{enumerate}
\item if $z=x+i\xi$, $x,\xi \in \rr d$, then
\begin{equation}\label{Eq:RelAssignDiag}
| \mathsf S_{\mathfrak V} \fka (z ,z) -\fka (\sqrt 2\, x,-\sqrt 2\, \xi )|
\lesssim
\omega (\sqrt 2\, \overline z)
\eabs z^{-2\rho}\text ;
\end{equation}

\vrum

\item if $\rho _0\in [0,2\rho )$, then $\mathsf S_{\mathfrak V}$ is
bijective from the set of weakly elliptic symbols in
$\Sh _\rho ^{(\omega )}(\rr {2d})$ of order $\rho _0$ to the set of
weakly elliptic symbols in
$\wideparen A _{\Sh ,\rho}^{(\omega )}(\cc {2d})$ of order $\rho _0$. 
\end{enumerate}
\end{thm}

\par

As a consequence of (2) in the previous theorem we get the following.

\par

\begin{cor}\label{Cor:ElliptEquiv}
Let $\fka$ be as in Theorem \ref{Thm:ElliptEquiv}. Then
the following is true:
\begin{enumerate}
\item if $\rho _0\in [0,2\rho )$, then $\fka \in \Sh _\rho ^{(\omega )}(\rr {2d})$ is weakly elliptic of order
$\rho _0$, if and only if
$\mathsf S_{\mathfrak V}\fka \in \wideparen A _{\Sh ,\rho}^{(\omega )}(\cc {2d})$ is weakly elliptic of order $\rho _0$;

\vrum

\item $\fka \in \Sh _\rho ^{(\omega )}(\rr {2d})$ is elliptic if and only if
$\mathsf S_{\mathfrak V}\fka \in \wideparen A _{\Sh ,\rho}^{(\omega )}(\cc {2d})$ is elliptic.
\end{enumerate}
\end{cor}

\par

For the proof of Theorem \ref{Thm:ElliptEquiv} we need the following proposition,
related to Propositions \ref{Prop:WickToAntiWick} and
\ref{Prop:AsymptoticExpansion1}.

\par

\begin{prop}\label{Prop:WickShubinTaylorExp}
Let $N\ge 0$ be an integer, $\rho \ge 0$, $\omega \in \mascP _{\Sh ,\rho}(\rr {2d})\simeq
\mascP  _{\Sh ,\rho}(\cc d)$, $\omega _k(x,\xi )=\omega (x,\xi )\eabs {(x,\xi )}^{-2\rho k}$
and $\fka \in \Sh _\rho ^{(\omega )}(\rr {2d})$. Then for some
$\fkc _N\in \Sh _\rho ^{(\omega _{N+1})}(\rr {2d})$ and
constants $\{ c_{\alpha} \}_{|\alpha| \le 2N}$ with $c_{0}=1$, it holds
\begin{equation}\label{Eq:WickShubinTaylorExp}
\mathsf S_{\mathfrak V}\fka  (2^{-\frac 12}z ,2^{-\frac 12}z)
=
\sum _{k=0}^N\fka _k(x,-\xi ) + \fkc _N(x,-\xi ),
\qquad
\fka _k = \sum _{|\alpha |=2k} c_{\alpha} \partial ^\alpha \fka .
\end{equation}
\end{prop}

\par

\begin{proof}
Let $\psi$ be as in Proposition \ref{Prop:SBaTaRel}.
If we put $z=w$, then \eqref{eq:BargmannsymbolSTFT} and
Taylor's formula give
\begin{multline}\label{Eq:TaylorNonWick}
(2\pi )^{d} \mathsf S_{\mathfrak V}\fka  (2^{-\frac 12}z ,2^{-\frac 12}z)
=
(2\pi )^{\frac {3d}2} \cT _\psi \fka  (x,-\xi ,0,0)
\\[1ex]
=
2^d\iint _{\rr {2d}} \fka (t+x,\tau -\xi )e^{-(|t|^2+|\tau |^2)}\, dtd\tau
=
\sum _{k=0}^{2N+1} \fkb _k(x,-\xi ) +\fkc (x,-\xi )
\end{multline}
where
\begin{align*}
\fkb _k(x,\xi ) &= \frac {2^d}{k!}\iint _{\rr {2d}}
\eabs {\fka ^{(k)}(x,\xi);(t,\tau ),\dots ,(t,\tau )}
e^{-(|t|^2+|\tau |^2)}\, dtd\tau 
\intertext{and}
\fkc (x,\xi )&= 
\frac 1{(2N+1)!}\int _0^1 (1-\theta )^{2N+1} \fkc _\theta (x,\xi )\, d\theta ,
\intertext{with}
\fkc _\theta (x,\xi ) &= 2^d\iint _{\rr {2d}}
\eabs {\fka ^{(2N+2)}(x+\theta t,\xi +\theta \tau  );(t,\tau ),\dots ,(t,\tau )}
e^{-(|t|^2+|\tau |^2)}\, dtd\tau .
\end{align*}

\par

If $k$ is odd, then
$$
(t,\tau )\mapsto \eabs {\fka ^{(k)}(x,\xi);(t,\tau ),\dots ,(t,\tau )}
e^{-(|t|^2+|\tau |^2)} 
$$
is odd which implies that the integral is zero. Hence $\fkb _k(x,\xi )=0$ when
$k$ is odd. For $k=0$ we observe that the integral for $\fkb _0$ becomes
$$
2^d\iint _{\rr {2d}} e^{-(|t|^2+|\tau |^2)}\, dtd\tau
=
(2\pi )^d,
$$
and it follows from these relations that
$$
(2\pi )^{-d}\sum _{k=0}^{2N+1}\fkb _k = \sum _{k=0}^N\fka _k,
$$
with $\fka _k$ as in \eqref{Eq:WickShubinTaylorExp} and $c_{0}=1$. Hence
the result follows if we prove that the last term in \eqref{Eq:TaylorNonWick} satisfies $\fkc _N \in \Sh _\rho ^{(\omega _{N+1})}(\rr {2d})$. 

\par

For $\theta \in [0,1]$ and $\alpha \in \nn {2d}$ we have
\begin{multline*}
|\partial ^\alpha \fkc _\theta (x,\xi )|
\lesssim
\iint _{\rr {2d}} |\partial^\alpha \fka ^{(2N+2)}(x+\theta t,\xi +\theta \tau  )|
\eabs {(t,\tau )}^{2N+2} e^{-(|t|^2+|\tau |^2)}\, dtd\tau 
\\[1ex]
\lesssim
\iint _{\rr {2d}} \omega (x+\theta t,\xi +\theta \tau  )
\eabs {(x+\theta t,\xi +\theta \tau  )}^{-(2N+2+|\alpha |)\rho}
\eabs {(t,\tau )}^{2N+2}
e^{-(|t|^2+|\tau |^2)}\, dtd\tau 
\\[1ex]
\lesssim
\omega (x,\xi)
\eabs {(x,\xi )}^{-(2N+2+|\alpha |)\rho}
\iint _{\rr {2d}}
\eabs {(t,\tau )}^{N_0}
e^{-(|t|^2+|\tau |^2)}\, dtd\tau 
\\[1ex]
\asymp
\omega (x,\xi)
\eabs {(x,\xi )}^{-(2N+2+|\alpha |)\rho}
\end{multline*}
for some $N_0>0$. In the last inequality we have used the fact that $\omega$
is polynomially moderate.

\par

This implies
$$
|\partial ^\alpha \fkc  (x,\xi )|\lesssim \int _0^1 |\partial ^\alpha \fkc _\theta (x,\xi )|
\, d\theta
\lesssim
\omega (x,\xi)
\eabs {(x,\xi )}^{-(2N+2+|\alpha |)\rho},
$$
which shows that $\fkc, \fkc _N \in \Sh _\rho ^{(\omega _{N+1})}(\rr {2d})$.
\end{proof}

\par

\begin{proof}[Proof of Theorem \ref{Thm:ElliptEquiv}]
Let $\psi$ be as in Proposition \ref{Prop:SBaTaRel} and $N=0$
in Proposition \ref{Prop:WickShubinTaylorExp}. Then
\begin{equation}\tag*{(\ref{Eq:RelAssignDiag})$'$}
| \mathsf S_{\mathfrak V}\fka (2^{-\frac 12}z ,2^{-\frac 12}z) -\fka (x,-\xi )|
\lesssim
\omega (x,-\xi)
\eabs {(x,-\xi )}^{-2\rho},
\end{equation}
and (1) follows.

\par

Suppose $\rho _0\in [0,2\rho)$. Then it follows from the latter inequality that
$$
| \mathsf S_{\mathfrak V}\fka  (z ,z)|
\gtrsim \eabs z^{-\rho _0}\omega (\sqrt 2\, \overline z),
\qquad |z|\ge R
$$
for some $R>0$, if and only if
$$
|\fka (x,\xi )|\gtrsim \eabs {(x,\xi )}^{-\rho _0}\omega (x,\xi ), \qquad |z|\ge R
$$
for some $R>0$, and the asserted equivalence in (2) follows.
\end{proof}

\par

\subsection{Shubin hypoellipticity in a Wick setting}

\par




\par

\begin{defn}\label{Def:ShubinWickHypoelliptic}
Let $\rho >0$, $\rho _0\ge 0$,
$\omega \in \mascP _{\Sh ,\rho}(\cc d)$ and
$a\in \wideparen \maclA _{\Sh ,\rho}^{(\omega )}(\cc {2d})$.
Then $a$ is called \emph{hypoelliptic} (in the \emph{Shubin-Wick sense}
in $\wideparen \maclA _{\Sh ,\rho}^{(\omega )}(\cc {2d})$) of order $\rho _0$,
if there is an $R>0$
such that for every $\alpha ,\beta \in \nn d$, it holds
\begin{alignat}{2}
|\partial _z^\alpha \overline \partial _w^\beta  a(z,z)|
&\lesssim
|a(z,z)|\eabs z^{-\rho |\alpha +\beta |}, &
\qquad |z| &\ge R \text .
\label{Eq:ShubinWickHypoelliptic1}
\intertext{and}
|a(z,z)| &\gtrsim \omega _0(\sqrt 2 \overline z)
\eabs z^{-\rho _0}, &
\qquad |z| &\ge R.
\label{Eq:ShubinWickHypoelliptic2}
\end{alignat}
\end{defn}

\par

According to Definition \ref{Def:Hypoelliptic}, 
if $\omega$, $\rho$ and $\rho _0$ are as in
the definition, then $\fka \in \Sh _\rho ^{(\omega )}(\rr {2d})$
is hypoelliptic
of order $\rho _0$
means that
there is an $R>0$ such that for every $\alpha  \in \nn {2d}$, it holds
\begin{alignat}{2}
|\partial ^\alpha \fka (x,\xi )|
&\lesssim
|\fka (x,\xi )|\eabs {(x,\xi )}^{-\rho |\alpha |}, &
\qquad |(x,\xi )| &\ge R \text .
\label{Eq:ShubinWickHypoellipticReal1}
\intertext{and}
|\fka (x,\xi )| &\gtrsim \omega (x,\xi )
\eabs {(x,\xi )}^{-\rho _0}, &
\qquad |(x,\xi )| &\ge R.
\label{Eq:ShubinWickHypoellipticReal2}
\end{alignat}

\par

Similar to Theorem \ref{Thm:ElliptEquiv} we have the following.

\par

%
%

\begin{thm}\label{Thm:HypoelliptShubinEquiv}
Let $\rho >0$, $\rho _0\ge 0$,
$\omega \in \mascP _{\Sh ,\rho}(\rr {2d})\simeq
\mascP _{\Sh ,\rho}(\cc d)$,
$\fka \in \Sh _\rho ^{(\omega )}(\rr {2d})$ and $a=\mathsf S_{\mathfrak V}\fka$.
Then $\fka$ is hypoelliptic of order $\rho _0$ in
$\Sh _\rho ^{(\omega )}(\rr {2d})$, if and only if
$a$ is \emph{hypoelliptic} of order $\rho _0$
in $\wideparen \maclA _{\Sh ,\rho}^{(\omega )}(\cc {2d})$.
\end{thm}

\par

\begin{proof}
Suppose that $\fka \in \Sh _\rho ^{(\omega )}(\rr {2d})$ is hypoelliptic of order $\rho _0$, and choose $N\ge 0$ such that
$2N\rho >\rho _0$. Suppose that
$R>0$ is chosen such that \eqref{Eq:ShubinWickHypoellipticReal1}
and \eqref{Eq:ShubinWickHypoellipticReal2} are fulfilled.
Then Proposition \ref{Prop:WickShubinTaylorExp} gives for $z=x+i\xi$ with
$|z|\ge R$ where $R > 0$ is sufficiently large
%
%
\begin{multline*}
|a(2^{-\frac 12}z,2^{-\frac 12}z)|
\gtrsim
|\fka (x,-\xi )|-
\sum _{k=1}^N\sum _{|\alpha |=2k}
(|\partial ^\alpha \fka (x,-\xi )| +
|\fkc (x,-\xi )| )
\\[1ex]
\gtrsim
|\fka (x,-\xi )| -|\fka (x,-\xi )|\eabs {(x,-\xi )}^{-2\rho}-
\omega (x,-\xi )\eabs {(x,-\xi )}^{-\rho (2N+2)}
\\[1ex]
\gtrsim
|\fka (x,-\xi )| -|\fka (x,-\xi )|\eabs {(x,-\xi )}^{-2\rho}
\\[1ex]
\gtrsim 
|\fka (x,-\xi )| \gtrsim \omega (x,-\xi )\eabs {(x,-\xi )}^{-\rho _0},
\end{multline*}
and \eqref{Eq:ShubinWickHypoelliptic2} follows. In particular
it follows from the previous estimates that
\begin{equation}\label{Eq:IneqHypoShubinHypoWick}
|a(2^{-\frac 12}z,2^{-\frac 12}z)|
\gtrsim
|\fka (x,-\xi )|,
\qquad |z|\ge R.
\end{equation}

\par

For fixed $\alpha ,\beta \in \nn d$, let
$\Omega _k$ be the set of all $(\gamma ,\delta )\in \nn {2d}\times \nn {2d}$
such that $|\gamma |=2k$ and $|\delta |=|\alpha +\beta |$. By Proposition \ref{Prop:WickShubinTaylorExp} and
\eqref{Eq:IneqHypoShubinHypoWick} we have for
some $R$ large enough and $|z|\ge R$,
\begin{multline*}
|(\partial _z^\alpha \overline \partial _w^\beta a)(2^{-\frac 12}z,2^{-\frac 12}z)|
\lesssim
\sum _{k=0}^N\sum _{(\gamma ,\delta )\in \Omega _k}
(|\partial ^{\gamma +\delta}\fka (x,-\xi )| +
|\partial ^\delta \fkc (x,-\xi )| )
\\[1ex]
\lesssim
\sum _{k=0}^N\sum _{(\gamma ,\delta )\in \Omega _k}
\big ( |\fka (x,-\xi )|\eabs {(x,-\xi )}^{-\rho (2k+|\alpha +\beta |)} +
\omega (x,-\xi )\eabs {(x,-\xi )}^{-\rho (2N+|\alpha +\beta |)} \big )
\\[1ex]
\asymp
|\fka (x,-\xi )|\eabs {(x,-\xi )}^{-\rho |\alpha +\beta |} +
\omega (x,-\xi )\eabs {(x,-\xi )}^{-\rho (2N+|\alpha +\beta |)}
\\[1ex]
\lesssim
|\fka (x,-\xi )|\eabs {(x,-\xi )}^{-\rho |\alpha +\beta |} +
|\fka (x,-\xi )|\eabs {(x,-\xi )}^{\rho _0-\rho (2N+|\alpha +\beta |)}
\\[1ex]
\asymp
|\fka (x,-\xi )|\eabs {(x,-\xi )}^{-\rho |\alpha +\beta |}
\lesssim
|a(2^{-\frac 12}z,2^{-\frac 12}z)|\eabs {(x,-\xi )}^{-\rho |\alpha +\beta |},
\end{multline*}
which implies that \eqref{Eq:ShubinWickHypoelliptic1} holds.

\par

This shows that
$a$ is hypoelliptic of order $\rho _0$ in
$\wideparen \maclA _{\Sh ,\rho}^{(\omega )}(\cc {2d})$
when $\fka$ is hypoelliptic of order $\rho _0$ in
$\Sh _\rho ^{(\omega )}(\rr {2d})$.

\par

Suppose instead that $a$ is hypoelliptic of order $\rho _0$ in
$\wideparen \maclA _{\Sh ,\rho}^{(\omega )}(\cc {2d})$. By using
Proposition \ref{Prop:AsymptoticExpansion1},
\eqref{Eq:cNdef} and \eqref{Eq:Hdef}
instead of Proposition
\ref{Prop:WickShubinTaylorExp}, similar
computations as in the first part of the proof shows that
\eqref{Eq:ShubinWickHypoellipticReal1} and \eqref{Eq:ShubinWickHypoellipticReal2}
hold for some $R>0$. This shows that
$\fka$ is hypoelliptic of order $\rho _0$ in
$\Sh _\rho ^{(\omega )}(\rr {2d})$ when
$a$ is hypoelliptic of order $\rho _0$ in
$\wideparen \maclA _{\Sh ,\rho}^{(\omega )}(\cc {2d})$, and the result follows.
\end{proof}

\par

\subsection{Ellipticity in the case of polynomial symbols}

\par

Next we discuss ellipticity for polynomial symbols, i.{\,}e. 
\begin{alignat}{2}
\fka (x,\xi ) &= \sum _{|\alpha +\beta | \le N}c(\alpha ,\beta )
x^\alpha \xi ^\beta , &
\quad
x,\xi &\in \rr d,
\label{Eq:RealPolSymbol}
\intertext{and}
a(z,w) &= \sum _{|\alpha +\beta | \le N}c(\alpha ,\beta )z^\alpha {\overline w}^\beta ,
&
\quad
z,w &\in \cc d.
\label{Eq:ComplexPolSymbol}
\intertext{The corresponding principal symbols are}
\fka _p (x,\xi ) &= \sum _{|\alpha +\beta | = N}c(\alpha ,\beta )x^\alpha \xi ^\beta , &
\quad
x,\xi &\in \rr d,
\label{Eq:RealPrincipalPolSymbol}
\intertext{and}
a_p (z,w ) &= \sum _{|\alpha +\beta | = N}c(\alpha ,\beta )
z^\alpha {\overline w} ^\beta , &
\quad
z,w &\in \cc d,
\label{Eq:ComplexPrincipalPolSymbol}
\end{alignat}
respectively.

\par

First we relate polynomials on $\rr {2d}$ to Shubin classes.

\par

\begin{prop}\label{Prop:DiffOpShubin}
Let $\fka$ and $\fka _p$ be as in \eqref{Eq:RealPolSymbol}
and \eqref{Eq:RealPrincipalPolSymbol}
for some $c(\alpha ,\beta )\in \mathbf C$, $\alpha ,\beta \in \nn d$
and $N\ge 0$, and let $\omega _{N}(x,\xi )=\eabs {(x,\xi )}^{N}$,
$x,\xi \in \rr d$. Then the following is true:
\begin{enumerate}
\item $\fka \in \Sh _1^{(\omega _{N})}(\rr {2d})$;

\vrum

\item $\fka$ is elliptic with respect to $\omega _{N}$, if and only if
$\fka _p(x,\xi )\neq 0$ when $(x,\xi )\neq 0$.
\end{enumerate}
\end{prop}

\par

The result can be considered folklore. In order to be self-contained we present
the arguments.

\par

\begin{proof}
First we prove (1). Let $t=\max (|x_1|,\dots ,|x_d|,|\xi _1|,\dots ,|\xi _d|)$ when
$x=(x_1,\dots ,x_d)\in \rr d$ and $\xi =(\xi _1,\dots ,\xi _d)\in \rr d$. Then
$$
|\fka (x,\xi )| \le  \sum _{|\alpha +\beta | \le N_0}|c(\alpha ,\beta )| t^{|\alpha +\beta |}
\lesssim  1+t^{N}
\le \eabs {(x,\xi )}^{N},
$$
which gives the desired bound for $|\fka (x,\xi)|$. Since the degree of a polynomial
is lowered by at least one for every differentiation we get
$$
|\partial^\alpha \fka (x,\xi )| \lesssim \eabs {(x,\xi )}^{N-|\alpha |}
$$
for every $\alpha \in \nn {2d}$, which gives (1).

\par

In order to prove (2) we let $\fka _p$ be as in \eqref{Eq:RealPrincipalPolSymbol}.
First suppose that $\fka _p(x,\xi )\neq 0$ when $(x,\xi )\neq (0,0)$, and let $g$ be the continuous
function on $\rr {2d}\setminus 0$ given by
$$
g(x,\xi ) =\frac {|\fka _p(x,\xi )|}{|(x,\xi )|^{N}},\qquad (x,\xi )\neq (0,0).
$$
Since $g$ is continuous and positive, and the sphere
$$
\mathbf S^{2d-1}=\sets {(x,\xi )\in \rr {2d}}{|x|^2+|\xi |^2 = 1}
$$
is compact, it follows that there are constants $c_1,c_2>0$ such that
$$
c_1\le g(x,\xi )\le c_2,\qquad (x,\xi )\in \mathbf S^{2d-1}.
$$ 
By homogeneity it now follows
$$
c_1|(x,\xi )|^{N} \le |\fka _p(x,\xi )| \le c_2|(x,\xi )|^{N},\qquad x,\xi \in \rr d.
$$

Hence, if
$$
\fkb (x,\xi )= \fka (x,\xi )-\fka _p(x,\xi )
= \sum _{|\alpha +\beta | \le N-1}c(\alpha ,\beta )x^\alpha \xi ^\beta ,
$$
then the first part of the proof implies that for some constants $C>0$ and $R>0$
we have
\begin{equation*}
|\fka (x,\xi )| \ge |\fka _p(x,\xi )| -|\fkb (x,\xi )| \ge c_1|(x,\xi )|^{N_0} - C\eabs {(x,\xi )}^{N-1}
\gtrsim \eabs {(x,\xi )}^{N}
\end{equation*}
when $|(x,\xi )|\ge R$. Hence $\fka$ is elliptic with respect to $\omega _{N_0}$.

\par

Suppose instead $\fka _p(x_0,\xi _0)=0$ for some $(x_0,\xi _0)\neq (0,0)$. For any
$(x,\xi )=(tx_0,t\xi _0)$ we have
\begin{multline*}
|\fka (x,\xi )| \le |\fka _p(x,\xi )| +|\fkb (x,\xi )|
=
|t^N\fka _p(x_0,\xi _0)| +|\fkb (x,\xi )|
\\[1ex]
=
|\fkb (x,\xi )| \lesssim \eabs {(x,\xi )}^{N-1},
\end{multline*}
giving that $|\fka (x,\xi )|\gtrsim \eabs {(x,\xi )}^{N}$, $|(x,\xi )|\ge R$,
cannot hold for any $R>0$.
\end{proof}

\par

By Theorems \ref{Thm:ElliptEquiv}, \ref{Thm:HypoelliptShubinEquiv}
and Proposition \ref{Prop:DiffOpShubin} we get the following.
The details are left for the reader.

\par

\begin{prop}\label{Prop:DiffOpBargmannShubin}
Let $a$ and $a_p$ be as in \eqref{Eq:ComplexPolSymbol}
and \eqref{Eq:ComplexPrincipalPolSymbol}
for some $c(\alpha ,\beta )\in \mathbf C$, $\alpha ,\beta \in \nn d$ 
and $N\ge 0$, and let $\omega _{N}(x,\xi )=\eabs {(x,\xi )}^{N}$,
$x,\xi \in \rr d$.
Then the following is true:
\begin{enumerate}
\item $a\in \maclA _{\Sh ,1}^{(\omega _{N})}(\cc {2d})$;

\vrum

\item $a$ is elliptic in $\maclA _{\Sh ,1}^{(\omega _{N})}(\cc {2d})$ if and only if $a_p(z,z)\neq 0$
when $z\neq 0$.
\end{enumerate}
\end{prop}

\par

\begin{rem}\label{Rem:Elliptic}
Let $\fka$, $\fka _p$, $a$ and $a_p$ be as in
\eqref{Eq:RealPolSymbol}--\eqref{Eq:ComplexPrincipalPolSymbol}.
Then it follows from Propositions \ref{Prop:DiffOpShubin} and
Proposition \ref{Prop:DiffOpBargmannShubin} that $\fka$ is elliptic, if
and only if $\fka _p$ is elliptic, and that $a$ is elliptic, if and only if $a_p$
is elliptic.
\end{rem}

\par

We have now the following.

\par

\begin{thm}\label{Thm:Ellipticity}
Let $\fka \in \Sh _1^{(\omega _{N})}(\rr {2d})$ and $\fka _p$ be as in \eqref{Eq:RealPolSymbol}
and \eqref{Eq:RealPrincipalPolSymbol}
for some $c(\alpha ,\beta )\in \mathbf C$, $\alpha ,\beta \in \nn d$
and $N\ge 0$.
Then the following is true:
\begin{enumerate}
\item the principal symbol $a_p(z,w)$ of $\mathsf S_{\mathfrak V}\fka$
is given by
\begin{equation}
\label{Eq:ComplexPrincipalSymb}
a_p(z,w) = 2^{-\frac N2}
\sum _{|\alpha +\beta |= N}
c(\alpha ,\beta )i^{|\beta |}
(z+\overline w)^\alpha (z-\overline w)^\beta 
\text ;
\end{equation}

\vrum

\item $\fka$ is elliptic in $\Sh _1^{(\omega _{N})}(\rr {2d})$ if and only if
$a_p$ is elliptic in $\maclA _{\Sh ,1}^{(\omega _{N})}(\cc {2d})$;

\vrum

\item $\fka _p(x,\xi )> 0$ for every $(x,\xi )\neq (0,0)$,
if and only if
$a_p(z,z)> 0$ for every $z\neq 0$.
\end{enumerate}
\end{thm}

\par

\begin{proof}
Let $z=x+i\xi$, $x,\xi \in \rr d$, i.{\,} e. $x=\frac 12(z+\overline z)$ and
$\xi =\frac 1{2i}(z-\overline z)$. By Theorem \ref{Thm:ElliptEquiv} we
get
\begin{equation}\label{Eq:PrincipalSymbolShubinWickIdent}
a_p(2^{-\frac 12}z,2^{-\frac 12}z) = \fka _p(x,-\xi ).
\end{equation}
This implies
\begin{multline*}
a_p(z,z) = 2^\frac N2\sum _{|\alpha +\beta | = N}c(\alpha ,\beta )x^\alpha (-\xi )^\beta
\\[1ex]
= 2^\frac N2\sum _{|\alpha +\beta | = N}c(\alpha ,\beta )2^{-|\alpha |}
(z+\overline z)^\alpha (2i)^{-|\beta |}(-(z-\overline z) )^\beta ,
\end{multline*}
which gives
\begin{equation}
\tag*{(\ref{Eq:ComplexPrincipalSymb})$'$}
a_p(z,z) = 2^{-\frac N2}
\sum _{|\alpha +\beta |= N}
c(\alpha ,\beta )i^{|\beta |}
(z+\overline z)^\alpha (z-\overline z)^\beta .
\end{equation}
The formula \eqref{Eq:ComplexPrincipalSymb} now follows from
\eqref{Eq:ComplexPrincipalSymb}$'$ and analytic continuation, using the
fact that $a_p(z,w)$ is analytic in $z$ and conjugate analytic in $w$.


\par

The assertion (2) follows by a combination of Corollary \ref{Cor:ElliptEquiv},
Propositions \ref{Prop:DiffOpShubin} and \ref{Prop:DiffOpBargmannShubin},
and the assertion (3) is a direct consequence of
\eqref{Eq:PrincipalSymbolShubinWickIdent}. 
%
%
\end{proof}

\par

\section{A necessary condition for polynomially bounded
Wick symbols}\label{sec6}

\par

In \cite[Section~2.7]{Fo} Folland shows that polynomial symbols for
pseudo-differential operators correspond
to polynomial Wick and anti-Wick symbols. 
Thus partial differential operators with polynomial coefficients
corresponds to polynomial Wick symbols. 

Here we show that a Wick symbol that is polynomially bounded
must be a polynomial. This gives a characterization of Wick
symbols corresponding to polynomial symbols for pseudo-differential
operators.

\par

Cauchy's integral formula implies that an entire function
which is polynomially bounded must
be a polynomial:

\par

\begin{prop}\label{Prop:PolynomialSymbol}
Let $F\in A(\cc d)$ have Maclaurin series
$$
F(z) = \sum _{\alpha \in \nn d} c(\alpha )e_\alpha (z),\quad z\in \cc d.
$$
Suppose that for some $j\in \{ 1,\dots ,d\}$, $C>0$, $N \ge 0$,  and an open neighbourhood
$I\subseteq \mathbf C$ of the origin we have
$$
|F(z)|\le C \eabs {z_j}^N,\quad z_j\in \mathbf C,
$$
provided $z_k\in I$, $k\in \{ 1,\dots ,d\} \setminus \{ j\}$. Then $c(\alpha ) = 0$
when $\alpha _j > N$.
\end{prop}

\par

\begin{proof}
By interchanging the variables, we may assume that $j=d$. Let
$R\ge 1$ and $\ep >0$ be chosen such that
$$
D_\ep \equiv \sets {z_0\in \mathbf C}{|z_0|\le \ep} \subseteq I.
$$
Take $\alpha \in \nn d$ such that $\alpha _d>N$, let
$\beta =(\alpha _1+1,\dots ,\alpha _d+1)\in \nn d$ and
$\gamma _\ep \subseteq \mathbf C$ be the
boundary circle of $D_\ep$. Then Cauchy's integral formula gives
\begin{multline*}
\frac {|c(\alpha )|}{\alpha !^{\frac 12}}
=
\left |
\frac { \partial^\alpha F (0) }{\alpha !}
\right |
= (2\pi )^{-d}
\left |
\idotsint _{\gamma_\ep ^{d-1}}
\left (
\int _{|z_d|=R}
\frac { F(z)}{z^\beta}\, dz_d
\right )
\, dz_1\cdots dz_{d-1}
\right | 
\\[1ex]
\le
(2\pi )^{-d}
\idotsint _{\gamma_\ep ^{d-1}}
\left (
\int _{|z_d|=R}
\frac {| F(z)|}{|z^\beta |}\, |dz_d |
\right )
\, |dz_1|\cdots |dz_{d-1}|
\\[1ex]
\lesssim
R^{-\alpha _d} \eabs R^N\ep ^{-(\alpha _1+\dots +\alpha _{d-1})}
\to 0 
\end{multline*}
as $R\to \infty$.
\end{proof}

\par

\begin{cor}
Let $a\in \wideparen A(\cc {2d})$ and suppose 
\begin{equation}\label{Eq:GlobPolEst}
|a(z,w)|\lesssim \eabs{(z,w)} ^N
\end{equation}
for some $N\ge 0$. Then $a$ is a polynomial in $z\in \cc d$ and $\overline w\in \cc d$
of degree at most $N$.
\end{cor}

\par

\begin{proof}
By Proposition \ref{Prop:PolynomialSymbol} it follows that $a$ is a
polynomial of degree at most $2dN$. We need to prove that the degree
is at most $N$. In order to do this we may assume that $a$ has degree
at least one.

\par

For some integer $M\ge 1$ we have
$$
a(z,w) = a_M(z,w) +a_{M-1}(z,w),
$$
where
\begin{align*}
a_M(z,w)
&=
\sum _{|\alpha +\beta |=M}c(\alpha ,\beta )z^\alpha \overline w^\beta
\intertext{is non-trivial and}
a_{M-1}(z,w)
&=
\sum _{|\alpha +\beta |\le M-1}c(\alpha ,\beta )z^\alpha \overline w^\beta .
\end{align*}
Since $a_M$ is non-trivial, there are $z_0,w_0\in \cc d$ such that
$|z_0|^2+|w_0|^2=1$ and $|a_M(z_0,w_0)|=c_0\neq 0$. By
homogeneity we get
$$
|a_M(tz_0,tw_0)| = c_0|t|^M,\qquad t \in \ro.
$$
In the same way we get
$$
|a_{M-1}(tz_0,tw_0)| \le C(1+|t|)^{M-1},\qquad t\in \ro
$$
for some constant $C$ which is independent of $t$.

\par

Suppose contrary to the assumption that $M>N$. For $t\in \ro$ with
$|t|\ge 1$ we have
\begin{multline*}
\left |
\frac {a(tz_0,tw_0)}{\eabs {(tz_0,tw_0)}^N}
\right |
\gtrsim
|t|^{-N}\left (|a_M(tz_0,tw_0)| - |a_{M-1}(tz_0,tw_0)|    \right )
\\[1ex]
\ge
|t|^{-N}\left (c_0|t|^M - C(1+|t|)^{M-1} \right ) \to \infty
\quad \text{as}\quad
|t|\to \infty .
\end{multline*}
This contradicts \eqref{Eq:GlobPolEst}, and the hence our assumption
that $M>N$ must be false.
\end{proof}

\end{document}